%% file: rwi.tex
\documentclass{amsart}
\usepackage{amsfonts,amsmath,amssymb,amsthm,anyfontsize,array,bm,booktabs,caption,enumitem,extarrows,graphicx,hyperref,longtable,mathalfa,mathdots,mathtools,microtype,scrextend,slashed,tabularx,tikz,tikz-cd,parskip,comment,wasysym}
\usepackage[nameinlink]{cleveref}
\usetikzlibrary{arrows}

\usepackage[sc,osf]{mathpazo}
\usepackage[mathscr]{euscript}

\theoremstyle{definition}
\newtheorem{dfn}{Definition}[section]

\theoremstyle{theorem}
\newtheorem{prp}[dfn]{Proposition}
\newtheorem{lmm}[dfn]{Lemma}
\newtheorem{ass}[dfn]{Assumption}
\newtheorem{thm}[dfn]{Theorem}
\newtheorem{crl}[dfn]{Corollary}
\newtheorem*{thm*}{Theorem}

\newcounter{mainthms}

\theoremstyle{theorem}
\newtheorem{mainthm}[mainthms]{Theorem}

\numberwithin{equation}{subsection}

\newcommand{\red}[1]{{\color{red} #1}}

\newcommand{\what}[1]{\widehat{#1}}
\newcommand{\conj}[1]{\overline{#1}}
\newcommand{\wilde}[1]{\widetilde{#1}}

\newcommand{\brak}[1]{\langle #1 \rangle}
\newcommand{\vnorm}[1]{\left\lvert #1 \right\rvert}

\newcommand{\xlra}[1]{\xlongrightarrow{#1}}

\newcommand{\vdotsb}{\rotatebox[origin=c]{-90}{$\dotsb$}}
\newcommand{\ddotsb}{\rotatebox[origin=c]{-45}{$\dotsb$}}

\newcommand{\vcong}{\rotatebox[origin=c]{-90}{$\cong$}}

\newcommand{\inv}{{-1}}
\newcommand{\vphi}{\varphi}
\newcommand{\veps}{\varepsilon}
\newcommand{\sub}{\subseteq}
\newcommand{\twoslash}{/\!\!/}

\newcommand{\mc}[1]{\mathcal{#1}}
\newcommand{\mf}[1]{\mathfrak{#1}}
\newcommand{\ms}[1]{\mathscr{#1}}
\newcommand{\mb}[1]{\mathbb{#1}}
\newcommand{\cat}[1]{\mathsf{#1}}

\newcommand{\sA}{\ms{A}}
\newcommand{\sB}{\ms{B}}
\newcommand{\sC}{\ms{C}}
\newcommand{\sD}{\ms{D}}
\newcommand{\sE}{\ms{E}}
\newcommand{\sQ}{\ms{Q}}
\newcommand{\Spaces}{\ms{S}}
\newcommand{\sX}{\ms{X}}
\newcommand{\sY}{\ms{Y}}

\newcommand{\sPP}{\ms{PP}}
\newcommand{\sSRW}{\ms{CRW}}

\newcommand{\cA}{\mc{A}}

\newcommand{\cM}{\mc{M}}
\newcommand{\cN}{\mc{N}}
\newcommand{\cO}{\mc{O}}
\newcommand{\cP}{\mc{P}}
\newcommand{\cQ}{\mc{Q}}

\newcommand{\bA}{\mathbf{A}}

\newcommand{\bC}{\mathbf{C}}
\newcommand{\bCat}{\mathbf{C}\mathrm{at}}

\newcommand{\bE}{\mathbf{E}}
\newcommand{\bh}{\mathbf{h}}

\newcommand{\bMF}{\mathbf{MF}}
\newcommand{\Path}{\mathbf{P}\mathrm{ath}}

\newcommand{\Set}{\mathbf{S}}
\newcommand{\bRW}{\mathbf{RW}}
\newcommand{\bT}{\mathbf{T}}
\newcommand{\bD}{\mathbf{\Delta}} 
\newcommand{\bS}{\mathbf{\Sigma}}
\newcommand{\bL}{\mathbf{\Lambda}}
\newcommand{\bG}{\mathbf{\Gamma}}
\newcommand{\bsd}{\mathbf{\Theta}}
\newcommand{\bvx}{\mathbf{\Xi}}
\newcommand{\CC}{\mb{C}}
\newcommand{\KK}{\mb{K}}
\newcommand{\LL}{\mb{L}}
\newcommand{\NN}{\mb{N}}

\newcommand{\TT}{\mb{T}}
\newcommand{\ZZ}{\mb{Z}}
\newcommand{\bd}{\mf{d}} 

\newcommand{\Hom}{\mathrm{Hom}}

\newcommand{\Fun}{\mathrm{Fun}}
\newcommand{\Map}{\mathrm{Map}}
\newcommand{\Nat}{\mathrm{Nat}}
\newcommand{\Cat}{\mathrm{Cat}}
\newcommand{\Seg}{\mathrm{Seg}}
\newcommand{\CAlg}{\mathrm{CAlg}}
\newcommand{\Mon}{\mathrm{Mon}}
\newcommand{\RKan}{\mathrm{RKan}}
\newcommand{\LKan}{\mathrm{LKan}}
\DeclareMathOperator{\Spec}{Spec}
\newcommand{\Tw}{\mathrm{Tw}}
\newcommand{\id}{\mathrm{id}}
\DeclareMathOperator{\dom}{dom}
\DeclareMathOperator{\im}{im}
\DeclareMathOperator*{\colim}{colim}

\newcommand{\rmcl}{\mathrm{cl}}
\newcommand{\rmcart}{\mathrm{cart}}
\newcommand{\rmcocart}{\mathrm{cocart}}
\newcommand{\rmint}{\mathrm{int}}
\newcommand{\rminj}{\mathrm{inj}}
\newcommand{\opp}{{\mathrm{op}}}
\newcommand{\crep}{\mathrm{cr}}
\newcommand{\rmc}{\mathrm{c}}
\newcommand{\Art}{\mathrm{Art}}

\newcommand{\sCat}{\ms{C}\mathrm{at}}
\newcommand{\GS}{\mathrm{GS}}
\newcommand{\PP}{\mathrm{PP}}
\newcommand{\SRW}{\ms{CRW}}
\newcommand{\Lag}{\mathrm{Lag}}
\newcommand{\dSt}{\mathrm{d}\ms{S}\mathrm{t}}
\newcommand{\dAff}{\mathrm{d}\ms{A}\mathrm{ff}}
\newcommand{\Mod}{\ms{M}\mathrm{od}}
\newcommand{\sCAlg}{\ms{CA}\mathrm{lg}}
\newcommand{\QCoh}{\ms{QC}\mathrm{oh}}
\newcommand{\PreSymp}{\ms{P}\mathrm{re}\ms{S}\mathrm{ymp}}

\title[Higher categories of push-pull spans, I]{Higher categories of push-pull spans, I: \\ Construction and applications}
\date{\today}
\author{Lorenzo Riva}
\address{Department of Mathematics \\ University of Notre Dame \\ South Bend, IN 46556, USA.}
\email{lriva@nd.edu}
\urladdr{https://sites.google.com/nd.edu/lorenzo-riva/home}

\begin{document}

\maketitle

\begin{abstract}
	This is the first part of a project aimed at formalizing Rozansky-Witten models in the functorial field theory framework. Motivated by work of Calaque-Haugseng-Scheimbauer, we construct a family of symmetric monoidal $(\infty,3)$-categories $\sPP(\sC; Q^\otimes)$, parametrized by an $\infty$-category $\sC$ with finite limits and a representable functor $\sQ^\otimes = \sC(-, Q^\otimes) : \sC^\opp \to \CAlg(\sCat_1)$ with pushforwards, which contains correspondences in $\sC$ with local systems in $\sQ^\otimes$ that compose via a push-pull formula. We apply this general construction to provide an approximation $\sSRW$ to the $3$-category of Rozansky-Witten models whose existence was conjectured by Kapustin-Rozansky-Saulina; this approximation behaves like a ``commutative'' version of the conjectured $3$-category and is related to work of Stefanich on higher quasicoherent sheaves.
\end{abstract}

\tableofcontents

\include{sec_intro.tex}

\include{sec_construction.tex}

\include{sec_approx-rw.tex}

\include{sec_appendix.tex}

\end{document}

%% file: sec_intro.tex
\section{Introduction} \label{sec:introduction}

\subsection{Motivation}

\subsubsection{Field theories and higher categories}

There is an interesting interplay between algebraic topology and modern mathematical physics whose formalization goes back to the work of Atiyah in \cite{Atiyah1988}. The objects in question are \emph{functorial topological field theories}: these are symmetric monoidal functors
\begin{equation*}
	F : (\ms{B}\mathrm{ord}_n, \sqcup) \to (\sC, \otimes)
\end{equation*}
from a certain $(\infty,n)$-category of bordisms, possibly with additional tangential or geometric structures, to a target symmetric monoidal $(\infty,n)$-category $\sC$, usually a higher version of the category of vector spaces or chain complexes. In recent years there has been an explosion of interest in field theories due to their connection with the higher algebra of $\infty$-categories via the cobordism hypothesis (see \cite{BD1995, Lurie2009, GP2022}). The construction of functorial field theories starts by choosing or producing an appropriate target $(\infty,n)$-category $\sC$; one can then attempt to directly define $F$, which is usually very complicated for $n \geq 2$, or study the fully dualizable objects of $\sC$ and then appeal to the cobordism hypothesis to guarantee the existence of $F$. Some examples of both techniques can be found in \cite{Abrams1996, Costello2007, SP2009, ST2011, Juhasz2018, DSSP2020, CHS2022, BEP2023} in a variety of low dimensional or non-extended settings.

Unfortunately, $(\infty, n)$-categories are notoriously hard to construct due to the difficulty in making the composition operations satisfy an infinite collection of homotopy coherence axioms. The step-by-step technique used for building weak $n$-categories, which consists in specifying the collection of $k$-morphisms and then verifying a small set of axioms, does not work anymore, and the situation is even more complicated when one wants to add extra structure such as a symmetric monoidal product. This is a major roadblock for the study of field theories: many higher categories of physical interest, which arise from the study of classical or quantum field theory, involve a complicated interplay of algebra, geometry, and topology; collecting together this mixture of different mathematical objects in one coherent structure is difficult given the current techniques for building higher categories. In this paper we will solve this problem in a specific case of interest.

\subsubsection{The Rozansky-Witten $3$-category}

In \cite{KRS2009} Kapustin, Rozansky, and Saulina studied a $3$-dimensional field theory arising from physics which they call the \emph{Rozansky-Witten model}. Its rich structure is connected to K\"ahler geometry, derived categories of coherent sheaves, matrix factorizations, and deformation quantization. Motivated by the results in \cite{KRS2009}, Kapustin and Rozansky set out to describe a symmetric monoidal $3$-category $\bRW$\footnote{In Kapustin and Rozansky's notation $\bRW$ is $\dddot{\cat{L}}$. Given how we usually denote categories in this paper we prefer the more evocative notation of \cite{BCR2023}.} that would classify Rozansky-Witten models in the functorial field theory framework. In \cite{KR2010} they provide an extensive sketch of the construction:
\begin{enumerate}[label=(\alph*)]
	\item the objects are holomorphic symplectic manifolds and the morphisms include (but are not restricted to) Lagrangian spans; there is a dualization operation 
	\begin{equation*}
		(X, \omega) \mapsto (X, \omega)^\vee := (X, -\omega)
	\end{equation*}
	and a product
	\begin{equation*}
		(X, \omega_X) \odot (Y, \omega_Y) := (X \times Y, \pi_X^\ast \omega_X + \pi_Y^\ast \omega_Y)
	\end{equation*}
	and we have
	\begin{equation*}
		\bRW(X, Y) \simeq \bRW(\ast, X^\vee \odot Y)
	\end{equation*}
	for any symplectic manifolds $X$ and $Y$, where $\ast$ is the point with the trivial symplectic form;
	\item when restricting to cotangent bundles $X = \bT^\ast U$ with their canonical symplectic structure, the category of $2$-morphisms $\Hom_{\bRW(\ast, X)}(L_1, L_2)$ between two Lagrangians $L_1, L_2$ in $X$ is a deformation of the $\ZZ_2$-graded derived category of the intersection of $L_1$ and $L_2$, and in fact it restricts to matrix factorizations in the special affine case of $X = \bT^\ast \CC^n$; one then uses the fact that holomorphic symplectic manifolds have Darboux charts isomorphic to cotangent bundles to glue together these mapping categories in the case of an arbitrary $X \approx \bigcup_{i \in I} \bT^\ast U_i$;
	\item compositions are given by pullbacks at the level of Lagrangian spans and push-pull formulas (see \Cref{sec:aside-push-pull}) at the level of sheaves.
\end{enumerate}

Some excellent progress in describing $\bRW$ already appears in the literature. Using the powerful framework of derived algebraic geometry, Calaque, Haugseng, and Scheimbauer provide a ``semi-classical'' approximation in \cite{CHS2022} in the form of an $(\infty,3)$-category of iterated Lagrangian spans of symplectic derived stacks; this generalizes classical Rozansky-Witten theory from \cite{RW1997} for algebraic targets but does not account for the derived categories in the $2$-morphisms. Brunner, Carqueville, and Roggenkamp obtain a bicategory of ``affine Rozansky-Witten models'' in \cite{BCR2023} by using a categorification of matrix factorizations, which have been well-studied since their introduction in \cite{Eisenbud1980} and have many applications all throughout homological algebra; their construction should embed in the homotopy $2$-category of $\bRW$. The aforementioned works, together with the sequel \cite{BCFR2023} to \cite{BCR2023}, also contain strong dualizability theorems, making the constructions excellent for field theoretic applications.

\subsection{Main results}

Our first result is a general construction of a symmetric monoidal $(\infty,3)$-category $\sPP(\sC;Q^\otimes)$, dependent on two parameters with minimal requirements. We then specialize this result to give another approximation to $\bRW$ and bridge the gap between \cite{CHS2022} and \cite{BCR2023}. 

Fix an $\infty$-category $\sC$ with finite limits and a ``representable'' functor $\sQ^\otimes : \sC^\opp \to \CAlg(\sCat_1)$ -- by this we mean a functor $\sQ^\otimes : \sC^\opp \to \Fun(\bG^\opp \times \bD^\opp, \Spaces)$ (see \Cref{sec:notation} for the notation or \Cref{sec:commutative-monoids} for our model of symmetric monoidal $\infty$-categories) such that there exists another functor $Q^\otimes : \bG^\opp \times \bD^\opp \to \sC$, the representing object, and an equivalence $\sQ^\otimes(-)_{\bullet, \bullet} \simeq \sC(-, Q^\otimes_{\bullet, \bullet})$. We require that for any morphism $f : c \to d$ in $\sC$ the induced symmetric monoidal functor $f^\ast : \sQ^\otimes(d) \to \sQ^\otimes(c)$ has a (lax monoidal) right adjoint.

\begin{mainthm} \label{thm:thmA}
	There exists a a symmetric monoidal $(\infty,3)$-category $\sPP(\sC;Q^\otimes)$ of ``push-pull spans in $\sC$ valued in $Q^\otimes$'' satisfying the following three analogues of properties (a), (b), and (c) above:
	\begin{enumerate}[label=({\alph*}')]
		\item the bottom categorical layer ($0$- and $1$-morphisms) is the $\infty$-category of spans in $\sC$, whose objects are those of $\sC$ and morphisms are given by spans; the symmetric monoidal structure is induced by the product in $\sC$ and by the tensor product of the $\infty$-categories $\sQ^\otimes(-)$;
		\item given two spans $X \leftarrow Z_i \to Y$ ($i = 1, 2$) the top categorical layer ($2$- and $3$-morphisms) is given by the $\infty$-category $\sQ(Z_1 \cap Z_2)$, where $Z_1 \cap Z_2 = Z_1 \times_{X \times Y} Z_2$ and $\sQ(-)$ denotes the underlying $\infty$-category of $\sQ^\otimes(-)$;
		\item the composition operation is given by pullbacks in the bottom layer and by a push-pull formula (see \Cref{sec:aside-push-pull} for more details) in the top layer.
	\end{enumerate}
\end{mainthm}

This result can be generalized to obtain an $(\infty, n+1)$-category where the first $n$ levels are given by spans in $\sC$ and the $\infty$-categories of morphisms in the top levels are given by $\sQ^\otimes(-)$. The existence part of \Cref{thm:thmA} follows from \Cref{thm:push-pull-3-cat} and the properties follow directly from the construction. 

We can apply \Cref{thm:thmA} to obtain the promised approximation to $\bRW$. This is done in \Cref{sec:rw-3-cat} after some technical results at the end of \Cref{sec:construction}.

\begin{mainthm} \label{thm:thmB}
	Fix an algebraically closed field $\KK$ of characteristic $0$. Then there is a symmetric monoidal $(\infty, 3)$-category $\sSRW$ whose objects are symplectic derived stacks over $\KK$, $1$-morphisms are Lagrangian spans, and the $\infty$-category of $2$-morphisms between Lagrangian spans $X \leftarrow L_i \to Y$ ($i = 1,2$) is the $\infty$-category of $\ZZ_2$-graded quasicoherent sheaves on the derived intersection $L_1 \cap L_2$.
\end{mainthm}
Note that $\sSRW$ depends on the field $\KK$ chosen as the base for our derived stacks. We call $\sSRW$ the \emph{$(\infty,3)$-category of commutative Rozansky-Witten theories}. Before adding the symplectic structure, if we restrict to \emph{$1$-affine} stacks (see \cite{Stefanich2021}) then the resulting hom $(\infty,2)$-categories from $X$ to $Y$ resemble the $(\infty,2)$-categories $\mathrm{2}\mathcal{Q}\mathrm{Coh}(X \times Y)$ constructed by Stefanich in \cite{Stefanich2021}.

The reasoning behind the name comes from the discrepancy between predictions for the behavior of the hom categories in $\bRW$ and their actual behavior in $\sSRW$. Concretely, fix $X = \KK^n$ and a polynomial map $p : X \to \KK$. Then the derivative $\partial p$ induces a Lagrangian submanifold $L_{p} \sub \bT^\ast X$ of the cotangent bundle of $X$, the \emph{graph of $\partial p$}, and it is expected that the $2$-morphisms in $\bRW$ between $L_p$ and $L_0$, the graph of the zero polynomial, correspond to \emph{matrix factorizations of $p$}:
\begin{equation} \label{eqn:mf-prediction}
	\Hom_{\bRW(X)}(L_p, L_0) \simeq \mathrm{MF}(\KK[x_1, \dotsc, x_n]; p).
\end{equation}
Here $\bRW(X) := \Hom_{\bRW}(\ast, \bT^\ast X)$, where the cotangent bundle has its canonical symplectic form. From results of Dyckerhoff (see \cite{Dyckerhoff2011}) we know that in certain cases the right-hand side of \Cref{eqn:mf-prediction} is equivalent to a category of $\ZZ_2$-graded dg-modules over a \emph{non-commutative} dg-algebra. For example, in the very explicit case of $p(x) = x^{n+1}$ ($n \geq 2$) we have a dg-algebra $(A,d)$ defined by
\begin{equation*}
	A := \KK[x;\theta, \partial_\theta]/(\partial_\theta \theta = 1, \theta^2 = \partial_\theta^2 = 0), \quad dx = 0, d\theta = x, d\partial_\theta = (n+1)x^n
\end{equation*}
with $x$ in even degree and $\theta, \partial_\theta$ in odd degree; then
\begin{equation*}
	\mathrm{MF}\left( \KK[x]; \frac{x^{n+1}}{n+1} \right) \simeq \Mod^{\ZZ_2}_A.
\end{equation*}
On the other hand, the $\infty$-categories of $2$-morphisms in $\sSRW$ are modeled by $\ZZ_2$-graded dg-modules over \emph{commutative} dg-algebras: using the same example,
\begin{equation*}
	\Hom_{\sSRW(X)}(L_p, L_0) \simeq \QCoh^{\ZZ_2}(L_p \cap L_0)
\end{equation*}
where $L_p \cap L_0$ is the derived critical locus of $p$, which can be explicitly calculated to be the spectrum of the $\ZZ_2$-graded dg-algebra $(B, d)$ defined by
\begin{equation*}
	B = \KK[x;\veps], \quad dx = 0, d\veps = x^n
\end{equation*}
with $x$ in even degree and $\veps$ in odd degree; then
\begin{equation*}
	\QCoh^{\ZZ_2}(L_p \cap L_0) \simeq \Mod^{\ZZ_2}_{B}.
\end{equation*}
One discrepancy between $\bRW$ and $\sSRW$ thus arises at least at the level of the extra non-commutative variables introduced by the matrix factorizations. We note that there is a map $B \to A$ inducing a functor $\Mod^{\ZZ_2}_A \to \Mod^{\ZZ_2}_B$, which suggests that we can pass from the $2$-morphisms in the non-commutative picture to the $2$-morphisms in the commutative picture -- see \Cref{sec:connections-mf} for a generalization of this idea.

At the moment it is unclear how to overcome these differences, but we expect that it amounts to a careful manipulation of the geometric data in the local systems from \Cref{thm:thmA}. In particular, should there be a functorial way to encode all the desired geometric data that appears in the definitions of \cite[Sections 3, 4]{KR2010} then our framework would immediately (or maybe with some minor modifications) allow us to get a better approximation of $\bRW$.

\subsubsection{Aside: the push-pull formula and generalized spans} \label{sec:aside-push-pull}

Assuming their well-definedness, the composition operations in $\sPP(\sC;Q^\otimes)$ and in $\sSRW$ are clear for at least some of the morphisms. The composition of two $1$-morphisms represented by spans in $\sC$ is given by pulling back along the common vertex:
\begin{equation*}
	\left(
	\begin{tikzcd}[column sep = tiny, row sep = small]
		& L_2 \ar[dl] \ar[dr] & \\
		Y & & Z 
	\end{tikzcd}
	\right) \circ \left(
	\begin{tikzcd}[column sep = tiny, row sep = small]
		& L_1 \ar[dl] \ar[dr] & \\
		X & & Y 
	\end{tikzcd}
	\right)
	=
	\begin{tikzcd}[column sep = tiny, row sep = small]
		& & L_1 \times_{Y} L_2 \ar[dl] \ar[dr] & & \\
		& L_1 \ar[dl] \ar[dr] & & L_2 \ar[dl] \ar[dr] & \\
		X & & Y & & Z 
	\end{tikzcd}
\end{equation*}
The composition of two $3$-morphisms represented by arrows $f, g$ in some $\infty$-categories $\sQ^\otimes(W_1)$ and $\sQ^\otimes(W_2)$ is given by pulling back those morphisms along the canonical maps and then composing them as arrows within some $\sQ^\otimes(W)$.

The composition of $2$-morphisms is trickier. Let's say we have two objects $X, Y \in \sPP(\sC;Q^\otimes)$, three $1$-morphisms $Z_1, Z_2, Z_3 \in \Hom_{\sPP(\sC;Q^\otimes)}(X,Y)$ seen as spans from $X$ to $Y$, and two $2$-morphisms $\cM \in \sQ^\otimes(Z_1 \cap Z_2)$ and $\cN \in \sQ^\otimes(Z_2 \cap Z_3)$. This information can all be summarized in the diagram in \Cref{fig:gen-span-comp}, where $Z_{I} = \bigcap_{i \in I} Z_i$ is the intersection ($=$ appropriate pullback) of the $Z_i$s.
\begin{figure}[h!]
	\begin{tikzpicture}[scale = 1.5]
		\node (B0) at (-2,0) {$Z_1$};
		\node (B1) at (2,0) {$Z_2$};
		\node (B2) at (0,{2*sqrt(3)}) {$Z_3$};
		\node (B01) at (0,0) {$Z_{1,2}$};
		\node (B12) at (1,{sqrt(3)}) {$Z_{2,3}$};
		\node (B02) at (-1,{1*sqrt(3)}) {$Z_{1,3}$};
		\node (B012) at (0,{2*sqrt(3)/3}) {$Z_{1,2,3}$};
		\node at (0,-0.4) {$\cM$};
		\node at (1.4,{sqrt(3)+0.2}) {$\cN$};
		\node at (-1.4,{sqrt(3)+0.2}) {$?$};
		\draw[->] (B01) -- (B0);
		\draw[->] (B01) -- (B1);
		\draw[->] (B12) -- (B1);
		\draw[->] (B12) -- (B2);
		\draw[->] (B02) -- (B0);
		\draw[->] (B02) -- (B2);
		\draw[->] (B012) -- (B0);
		\draw[->] (B012) -- (B1);
		\draw[->] (B012) -- (B2);
		\draw[->] (B012) -- node[left] {$f$} (B01);
		\draw[->] (B012) -- node[above] {$g$} (B12);
		\draw[->] (B012) -- node[midway,above] {$h$} (B02);
	\end{tikzpicture}
	\caption{A ``generalized span'' encoding the push-pull composition of $2$-morphisms.}
	\label{fig:gen-span-comp}
\end{figure}

The diagram, which resembles the barycentric subdivision of a standard $2$-simplex, is what we call a ``generalized span''. The composition $\cN \circ \cM$ should occupy the spot taken by the question mark, and in fact we want it to be given by a \emph{push-pull formula}
\begin{equation} \label{eqn:desired-push-pull-formula}
	\cN \circ \cM = h_\ast (f^\ast \cM \otimes g^\ast \cN) \in \sQ^\otimes(Z_1 \cap Z_3).
\end{equation}
Obviously we need to assume that $\sQ^\otimes$ has pushforwards in order for $h_\ast$ to exist. It turns out that this assumption is enough if we also have the extra data contained in the generalized spans. While the formula seems easy at first glance, ensuring that it behaves coherently and that it presents the composition of $2$-morphisms in $\sPP(\sC;Q^\otimes)$ will take a non-trivial amount of combinatorics. 

\subsubsection{Connections to matrix factorizations} \label{sec:connections-mf}

Finally we note that, despite the discrepancy between the categories of $2$-morphisms described above in the discussion following \Cref{thm:thmB}, there is a strong connection between the $2$-category $\bMF$\footnote{Originally defined for matrix factorizations of polynomials over $\mathbb{C}$, this $2$-category and the results concerning it can be generalized to any algebraically closed field $\KK$ of characteristic $0$. We leave this implicit in the discussion.} of truncated affine Rozansky-Witten models studied by Brunner, Carqueville, Fragknos, and Roggenkamp in \cite{BCR2023,BCFR2023} into the homotopy $2$-category of $\sSRW$ which precisely mimics the algebraic approach described in \cite[Section 1.3]{KR2010}. $\bMF$ is expected to correctly model Rozansky-Witten field theories in the case of the affine cotangent bundles $\bT^\ast \mathbb{C}^n$, so the following result gives a positive indication that $\sSRW$ is not far off from being a good model either:
\begin{thm*}[Theorem A in \cite{Riva2024}]
	There is a symmetric monoidal functor $\mf{e} : \bMF \to \bh_2 \sSRW$ whose essential image sits in the subcategory of $\bh_2 \sSRW$ spanned by the symplectic derived stacks of the form $\bT^\ast \KK^n$, $n \geq 0$.
\end{thm*}

The details appear in full in the sequel \cite{Riva2024} to this paper. It is important to note that the existence of $\mf{e}$, together with the results of \cite{BCR2023} about the dualizable objects in $\bMF$, implies that there are many non-trivial fully extended oriented $2$-dimensional field theories valued in $\bh_2 \sSRW$. In fact, one can find the computations of those field theories in \cite{Riva2024}.

\subsection{Summary of the paper}

The whole paper is written in the language of $\infty$-categories. \Cref{sec:appendix} contains the results we need in terms of higher categories, monoidal structures, miscellaneous techniques involving (co)ends and (co)cartesian fibrations, and other abstract machinery. We use \emph{complete $n$-fold Segal spaces} as our model for $(\infty, n)$-categories (see \Cref{sec:higher-cats} for a quick review or \cite{Barwick2005} for a more complete treatment), and in particular we use complete Segal spaces to model $\infty$-categories. It has been shown that all known models of $(\infty,n)$-categories -- including complete $n$-fold Segal spaces -- have the same homotopy theory; see \cite{RV2022} for $n = 1$ or \cite{BR2013, BR2020, BSP2021} for $n \geq 1$. In particular this implies that our main results do not depend on the choice of model we made.

The bulk of the construction is carried out in \Cref{sec:construction}. We start with an $\infty$-category $\sC$ with finite limits and a symmetric monoidal category object $Q^\otimes$ valued in $\sC$ that satisfies a condition on the existence of the pushforwards determined by it (see \Cref{ass:assumption}). We modify Haugseng's $(\infty,n)$-category of iterated spans in $\sC$ from \cite{Haugseng2018} to build a cocartesian fibration over $\bG^\opp \times \bD^{m+n, \opp}$ which describes \emph{generalized} iterated spans with local systems and a symmetric monoidal structure. The first $m$ categorical levels are spans with local systems exactly as in \cite{Haugseng2018}: $1$-morphisms are spans in $\sC$, $2$-morphisms are spans between spans in $\sC$, $3$-morphisms are spans between spans between spans in $\sC$, and so on. We take a different approach for the monoidal structure. The last $n$ categorical levels contain diagrams which are modeled after barycentric subdivisions of a standard simplex and serve to setup the framework for the push-pull formulas to hold. The local systems on each generalized span associate specific algebraic objects at each categorical level. Creating these special local systems and proving that they assemble to a higher category is the content of \Cref{sec:enforcing-push-pull}. In the end we'll have obtained a symmetric monoidal $(\infty,3)$-category $\sPP(\sC;Q^\otimes)$ of generalized spans with ``push-pull'' local systems. We then construct a relative version of $\sPP(\sC;Q^\otimes)$ dependent on a symmetric monoidal functor $U : \sD^\odot \to \sC^\times$ which allows us to talk about generalized $\odot$-cartesian spans in $\sD$ such that their image in $\sC$ under $U$ has a push-pull local system.

In \Cref{sec:applications} we provide some applications of the results of \Cref{sec:construction}. First we review enough derived algebraic geometry to talk about the $\infty$-categories $\dSt$ of derived stacks and $\PreSymp$ of presymplectic derived stacks together with the functor $\QCoh^{\ZZ_2} : \dSt^\opp \to \CAlg(\sCat_1)$ of $\ZZ_2$-graded quasicoherent sheaves. Then we turn the crank built in \Cref{sec:construction} to obtain $\sSRW$, our approximation to $\bRW$.

\subsection{Notation} \label{sec:notation}

For the reader's convenience we provide a list of objects that are commonly found in the literature or repeatedly used in this paper, together with some conventions that we will try to stick to.

Some notation concerning categories:
\begin{itemize}
	\item $\bD$ is a skeleton of the category of linearly ordered finite sets. Its objects are the sets $[n] = \{0 \leq 1 \leq \dotsb \leq n\}$, for $n$ a non-negative integer, and morphisms are order-preserving functions. A moprhism $\vphi$ in $\bD$ is \emph{inert} if $\vphi(i) = \vphi(0) + i$ for any $i$ in the domain of $\vphi$. The special inert morphisms
	\begin{equation*}
		\rho_i : [1] \to [n], \quad (0,1) \mapsto (i, i+1)
	\end{equation*}
	are called the \emph{Segal maps}. If $\vphi$ is a map in $\bD$ then we call the corresponding map in $\bD^\opp$ a \emph{face map} if $\vphi$ is injective and a \emph{degeneracy map} if $\vphi$ is surjective. Every map in $\bD^\opp$ can be factored uniquely as a face map followed by a degeneracy map. 
	\item $\bG^\opp$ is a skeleton of the category of pointed finite sets (i.e. $\bG$ is a skeleton of the \emph{opposite} of pointed finite sets). Its objects are the pointed sets $\brak{n} = \{\ast, 1, \dotsc, n\}$, for $n$ a non-negative integer, and a morphism $\brak{m} \to \brak{n}$ of pointed sets. This category has two symmetric monoidal structures, induced respectively by the wedge product $\vee$ and the smash product $\wedge$ of finite pointed sets. An \emph{inert} morphism $\vphi$ is one such that $\vnorm{\vphi^\inv(j)} = 1$ for any $j \neq \ast$ in the codomain of $\vphi$; that is, an inert morphism $\vphi$ is injective on $\dom(\vphi) \setminus \vphi^\inv(\ast)$. The special inert morphisms
	\begin{equation*}
		\tau_i : \brak{n} \to \brak{1}, \quad k \mapsto
		\begin{cases}
			1 & k = i, \\
			\ast & k \neq i
		\end{cases}
	\end{equation*}
	are called the \emph{Segal maps}.
	\item $\sCat_n$ is the $\infty$-category of $(\infty, n)$-categories modeled as complete $n$-fold Segal spaces -- see \Cref{dfn:complete-segal-space}.
	\begin{itemize}[label=$\circledast$]
		\item $\Spaces \simeq \sCat_0 \sub \sCat_1$ is the $\infty$-category spaces. The ordinary category $\Set$ of sets embeds into $\Spaces$ as the discrete spaces.
		\item $\bCat_1 \sub \sCat_1$ is the full subcategory containing the nerves of ordinary categories. When using an ordinary category we will implicitly pass to its nerve and consider it as an $\infty$-category. We stick to a general notational distinction: script calligraphic letters ($\ms{A}$, $\ms{B}$, $\ms{C}$, etcetera) will denote $\infty$-categories and bold letters ($\mathbf{A}$, $\mathbf{B}$, $\mathbf{C}$, etcetera) will denote ordinary categories.
		\item $\bCat_2 \sub \sCat_2$ is the category of categories enriched in small categories, embedded into $(\infty,2)$-categories via a binerve -- see \Cref{sec:path-category}.
	\end{itemize}
	\item If $x \in \sC$ is an object we will use $\sC_{/x}$ and $\sC_{x/}$ to denote the overcategory of maps into $x$ and the undercategory of maps from $x$, respectively.
	\item We will reserve $\Hom_\sC(-,-)$ for the space of morphisms in an $\infty$-category $\sC$; in particular, $\Map := \Hom_{\sCat_1}$. As usual, we denote by $\Fun(-,-)$ the $\infty$-category of functors between two $\infty$-categories and by $\Nat = \Hom_{\Fun(-,-)}$ the space of natural transformations between such functors. 
	\item $\LKan$ and $\RKan$ denote left and right Kan extensions -- see \cite[Section 7.3]{kerodon} or \Cref{sec:ends} for more details and formulas.
	\item If $T \in \Spaces$ and $c \in \sC$ then we denote by
	\begin{equation*}
		T \pitchfork c := \lim \left( T \xlra{\mathrm{const}_c} \sC \right), \quad T \otimes c := \colim \left( T \xlra{\mathrm{const}_c} \sC \right)
	\end{equation*}
	the cotensoring and tensoring of $c$ by $T$ whenever these (co)limits exist. They satisfy
	\begin{equation*}
		\sC(T \otimes c, d) \simeq \Spaces(T, \sC(c, d)) \simeq \sC(c, T \pitchfork d)
	\end{equation*}
	for $c, d \in \sC$ whenever these (co)limits exist.
	\item The end and coend of a functor $F : \sA^\opp \times \sA \to \sC$ (see \Cref{sec:ends} for a definition) are denoted by
	\begin{equation*}
		\int_\sA F \quad \text{and} \quad \int^\sA F
	\end{equation*}
	respectively.
\end{itemize}

Throughout the course of the paper we define or recall the definition of:
\begin{itemize}
	\item span-like posets:
	\begin{itemize}[label=$\circledast$]
		\item $\bS^k$, $\bL^k$ (\Cref{dfn:poset-cats}): the posets used to parametrize cartesian iterated spans;
		\item $\bsd^l$, $\bvx^l$ (\Cref{dfn:poset-cats}): the posets used to parametrize cartesian iterated barycentric subdivisions of standard simplices;
	\end{itemize}
	\item span fibrations and related categories:
	\begin{itemize}[label=$\circledast$]
		\item $\GS(\sC)^{m,n}$ (\Cref{dfn:gen-span-functor}): generalized spans in $\sC$;
		\item $\GS(\sC)^{m,n,\rmcart}$ (\Cref{dfn:gen-cartesian-span}): generalized cartesian spans in $\sC$;
		\item $\GS(\sC;\mc{P})^{m,n}$ (\Cref{dfn:fib-span-with-local-system}): generalized cartesian spans in $\sC$ with local systems induced by a  functor $\mc{P} : \bD^{m + n, \opp} \to \CAlg(\sC)$;
		\item $\PP(\sC;Q^\otimes)$ (\Cref{dfn:push-pull}), $\sPP(\sC;Q^\otimes)$ (\Cref{dfn:push-pull-3-cat}): push-pull spans in $\sC$ valued in a symmetric monoidal category object $Q^\otimes$;
		\item $\PP(U : \sD \to \sC; Q^\otimes)$ (\Cref{dfn:relative-push-pull-fibration}), $\sPP(U : \sD \to \sC; Q^\otimes)$ (\Cref{dfn:relative-push-pull-spans}): push-pull spans in $\sD$ valued in a symmetric monoidal category object $Q^\otimes$ in $\sC$;
		\item $\SRW(\PreSymp;Q^\otimes)$ (\Cref{dfn:horizontal-Lagrangian}), $\sSRW$ (\Cref{crl:srw}): algebraic Rozansky-Witten $(\infty,3)$-category;
	\end{itemize}
	\item path category for push-pull local systems:
	\begin{itemize}[label=$\circledast$]
		\item $\Path(l)$ (\Cref{dfn:path-cat}), $\bd[l]$ (\Cref{dfn:nerve-path-cat}): the path category of the poset $[l]$ and its bisimplicial nerve;
		\item $\bd X$ (\Cref{dfn:nerve-path-with-labels}): the path category labelled by a bisimplicial object $X$;
		\item $\cQ^\otimes$ (\Cref{eqn:formula-for-cQ}): a heavy modification of a symmetric monoidal category object $Q^\otimes$ in $\sC$;
	\end{itemize}
	\item basics of derived symplectic algebraic geometry (\Cref{sec:derived-stacks-def}, \Cref{sec:symplectic-dag}):
	\begin{itemize}[label=$\circledast$] 
		\item $\Mod_\KK$, $\sCAlg_\KK$: derived modules and derived commutative algebras;
		\item $\dAff$, $\dSt$: derived affine stacks and derived stacks;
		\item $\Mod_A$, $\Mod_A^{\ZZ_2}$: derived ($\ZZ$-graded) modules and $\ZZ_2$-graded modules over a derived algebra $A$;
		\item $\QCoh(X)$, $\QCoh^{\ZZ_2}(X)$, $\cO_X$: quasicoherent sheaves and $\ZZ_2$-graded sheaves on a derived stack $X$, and the structure sheaf of $X$;
		\item $\LL_X$, $\TT_X$: cotangent and tangent stack on $X$;
		\item $\cA^p(X)$, $\cA^{p, \rmcl}(X)$: $p$-forms and closed $p$-forms on a derived stack $X$;
		\item $\PreSymp$ : presymplectic derived stacks.
	\end{itemize}
\end{itemize}

And finally some useful conventions:
\begin{itemize}
	\item Given any object $A$, repeated superscripts $(\dotsb((A^{r_1})^{r_2})\dotsb)^{r_k}$ will be turned into a multi-superscript $A^{r_1, r_2, \dotsc, r_k}$ with entries separated by commas; similarly for subscripts. We will only use parentheses to distinguish between possibly ambiguous meanings of the repeated super/subscripts. We also abbreviate $(A^{\times n})^\opp$ as $A^{n, \opp}$.
	\item When dealing with commutative monoids $M^\otimes : \bG^\opp \to \sC$ in a cartesian monoidal $\infty$-category $\sC$ we reserve the superscript $\otimes$ for the functor out of $\bG^\opp$, which encodes the totality of the monoidal structure, and remove the superscript when we are interested in the underlying object of $M^\otimes$, namely $M = M^\otimes_{1}$.
	\item If $F : \sD \to \sCat_1$ is a functor, the domain $\infty$-category of a (cartesian or cocartesian) fibration classified by it will be denoted by $\what{F}$ -- see \Cref{sec:functors-vs-fibrations}.
	\item We use $\simeq$ for equivalence within an $\infty$-category and $\cong$ for isomorphism within an ordinary category, and we try to reserve the equality symbol $=$ for definitional equalities.
\end{itemize}

\subsection{Acknowledgements}

This work was greatly improved by many helpful conversations with friends and colleagues. Amongst them, we would like to acknowledge Chris Schommer-Pries and Stephan Stolz for being patient, insightful, and supportive mentors, and for suggesting \cite{CHS2022} as a starting point for this project; Nils Carqueville for providing early feedback on the strategy; Damien Calaque, Rune Haugseng, and Claudia Scheimbauer for writing \cite{CHS2022} and greatly influencing the writing style of this work; Adam Dauser for providing feedback, alternative constructions, and ``reality checks'' during the middle and last stages of writing the paper; Connor Malin and Mark Behrens for graciously offering their intuition on various homotopy-theoretic intricacies; Justin Hilburn for helpful comments regarding the connections with Rozansky-Witten theories; and all the speakers and participants of the conference ``Higher Invariants in Functorial Field Theory'' that took place in Regensburg in the summer of 2023, who provided a stimulating environment within which the crucial steps of \Cref{sec:enforcing-push-pull} were carried out.

%% file: sec_construction.tex
\section{Construction of the $(\infty, 3)$-category of push pull spans} \label{sec:construction}

Throughout the whole section we assume that $\sC$ is an $\infty$-category with finite limits. We will construct $\sPP(\sC;Q^\otimes)$ in steps. First we define a higher category of generalized cartesian spans with local systems similar to the one in \cite{Haugseng2018} and then we craft the local system that we will use to describe push-pull spans. Finally we put everything together and force the push-pull formulas to hold, obtaining an $(\infty,3)$-category. All of this will be done coherently with respect to the natural cartesian monoidal structure $\sC^\times$ on $\sC$, which essentially goes along for the ride, so that in the end our $(\infty,3)$-category will have a symmetric monoidal structure as well. We conclude by describing relative push-pull spans associated to a symmetric monoidal functor $U : \sD^\odot \to \sC^\times$ for $\sD^\odot$ another symmetric monoidal $\infty$-category.

\subsection{The generalized span fibration} \label{sec:gen-spans}

In this section we will obtain a cocartesian fibration of ``generalized iterated spans'' in $\sC$.

\subsubsection{Indexing categories}

Recall that an \emph{inert morphism} in $\bD$ is the inclusion of a connected subinterval, i.e. a function $\vphi : [m] \to [n]$ such that $\vphi(i) = \vphi(0) + i$ for $0 \leq i \leq m$.

\begin{dfn} \label{dfn:poset-cats}
	Let $\bD_{\rmint}$ and $\bD_{\rminj}$ denote the subcategories of $\bD$ containing the inert and the injective maps, respectively. Then define
	\begin{equation*}
		\bS^n := (\bD_{\rmint, /[n]})^\opp, \quad \bsd^n := (\bD_{\rminj, /[n]})^\opp.
	\end{equation*}
	There are two special full subcategories
	\begin{equation*}
		\bL^n \sub \bS^n, \quad \bvx^n \sub \bsd^n
	\end{equation*}
	spanned by the inert maps $(\vphi : [i] \to [n]) \in \bS^n$ with source $[i] = [0], [1]$ and the injective maps $(\psi : [i] \to [n]) \in \bsd^n$ with source $[i] = [0]$, respectively.
\end{dfn}

\begin{prp}
	The associations $[n] \mapsto \bS^n$ and $[n] \mapsto \bsd^n$ are functorial in $\bD$.
\end{prp}

\begin{proof}
	Let $\alpha : [m] \to [n]$ be a map in $\bD$. 
	
	If $\vphi : [i] \to [m]$ is inert then there is an inert map $\alpha_\ast(\vphi) : [\alpha(\vphi(i)) - \alpha(\vphi(0))] \to [n]$ defined by $\alpha_\ast(\vphi)(0) = \alpha(\vphi(0))$. Moreover, if $\vphi$ and $\psi$ are inert and fit into a commutative diagram
	\begin{equation*}
		\begin{tikzcd}
			{[i]} \ar[dr, "\vphi"'] \ar[rr, "\chi"] & & {[j]} \ar[dl, "\psi"] \\
			& {[m]} &
		\end{tikzcd}
	\end{equation*}
	in $\bD_{\rmint, /[m]}$ -- which, if there is one, is the \emph{unique} such commutative diagram -- then there is a inert map $\alpha_\ast(\chi)$ making the diagram
	\begin{equation*}
		\begin{tikzcd}
			{[\alpha(\vphi(i)) - \alpha(\vphi(0))]} \ar[dr, "\alpha_\ast(\vphi)"'] \ar[rr, "\alpha_\ast(\chi)"] & & {[\alpha(\psi(j)) - \alpha(\psi(0))]} \ar[dl, "\alpha_\ast(\psi)"] \\
			& {[n]} &
		\end{tikzcd}
	\end{equation*}
	commute, specified by 
	\begin{equation*}
		\alpha_\ast(\chi)(0) = \alpha(\psi(0) + \chi(0)) - \alpha(\psi(0)) = \alpha(\vphi(0)) - \alpha(\psi(0)).	
	\end{equation*}
	It is easy to see from the formula for $\alpha_\ast(\chi)$ that $\alpha_\ast$ commutes with composition of morphisms in $\bD_{\rmint, /[m]}$. Hence $\alpha_\ast : \bD_{\rmint, /[m]} \to \bD_{\rmint, /[n]}$ is a functor, and so after taking opposites we obtain a functor $\bS^\bullet : \bD \to \bCat_1$.
	
	If $\vphi : [i] \to [m]$ is injective then there is a unique $k_\vphi \in \NN$ and a unique isomorphism $[k_\vphi] \cong \im (\alpha \circ \vphi)$ of linearly ordered sets. Define $\alpha_\ast(\vphi)$ to be the injection $[k_\vphi] \cong \im(\alpha \circ \vphi) \hookrightarrow [n]$. If $\vphi$ and $\psi$ are injective and fit into a commutative diagram
	\begin{equation*}
		\begin{tikzcd}
			{[i]} \ar[dr, "\vphi"'] \ar[rr, "\chi"] & & {[j]} \ar[dl, "\psi"] \\
			& {[m]} &
		\end{tikzcd}
	\end{equation*}
	in $\bD_{\rminj, /[m]}$ then there is an injection $\alpha_\ast(\chi)$ making the diagram
	\begin{equation*}
		\begin{tikzcd}
			{[k_\vphi]} \ar[dr, "\alpha_\ast(\vphi)"'] \ar[rr, "\alpha_\ast(\chi)"] & & {[k_\psi]} \ar[dl, "\alpha_\ast(\psi)"] \\
			& {[n]} &
		\end{tikzcd}
	\end{equation*}
	commute, defined by
	\begin{equation*}
		\alpha_\ast(\chi) : [k_\vphi] \cong \im(\alpha \circ \vphi) = \im(\alpha \circ \psi \circ \chi) \hookrightarrow \im(\alpha \circ \psi) \cong [k_\psi].
	\end{equation*}
	If $\chi = \chi_2 \circ \chi_1$ then $\im(\alpha \circ \tau \circ \chi) \hookrightarrow \im(\alpha \circ \tau)$ factors as
	\begin{equation*}
		\begin{tikzcd}[row sep = small, column sep = small]
			\im(\alpha \circ \tau \circ \chi) \ar[r, hook] \ar[d, phantom, "\vcong"] & \im(\alpha \circ \tau \circ \chi_2) \ar[r, hook] \ar[d, phantom, "\vcong"] & \im(\alpha \circ \tau) \ar[d, phantom, "\vcong"] \\
			{[k_\vphi]} & {[k_\psi]} & {[k_\tau]}
		\end{tikzcd}
	\end{equation*}
	and so $\alpha_\ast$ commutes with composition. Hence it is a functor $\bD_{\rminj, /[m]} \to \bD_{\rminj, /[n]}$ and so, after taking opposites, we have a functor $\bsd^\bullet : \bD \to \bCat_1$.
\end{proof}

The category $\bS^n$ is an easily visualizable poset. For example, here is the case $n = 3$ (with $\bL^3$ spanned by the red asterisks):
\begin{figure}[h!]
	\begin{tikzpicture}
		\foreach \x in {0,...,3} {
		\draw[color = red] node (A\x) at (3*\x,0) {$\ast$};}
		\draw[color = red] node (A01) at (1.5,1.5) {$\ast$};
		\draw[color = red] node (A12) at (4.5,1.5) {$\ast$};
		\draw[color = red] node (A23) at (7.5,1.5) {$\ast$};
		\draw node (A012) at (3,3) {$\bullet$};
		\draw node (A123) at (6,3) {$\bullet$};
		\draw node (A0123) at (4.5,4.5) {$\bullet$};
		
		\draw[->, color = red] (A01) -- (A0);
		\draw[->, color = red] (A01) -- (A1);
		\draw[->, color = red] (A12) -- (A1);
		\draw[->, color = red] (A12) -- (A2);
		\draw[->, color = red] (A23) -- (A2);
		\draw[->, color = red] (A23) -- (A3);
		\draw[->] (A012) -- (A01);
		\draw[->] (A012) -- (A12);
		\draw[->] (A123) -- (A12);
		\draw[->] (A123) -- (A23);
		\draw[->] (A0123) -- (A012);
		\draw[->] (A0123) -- (A123);
		
		\draw node at (-0.5, 2.5) {$\bS^3:$};
	\end{tikzpicture}
	\label{fig:span}
\end{figure}

On a pyramid of height $n$ and base $n + 1$, an element $(\vphi : [i] \to [n]) \in \bS^n$ is represented by a dot at height $i$ and positioned at $\vphi(0)$ units from the left. Each dot at height $i$ has two arrows going down to height $i-1$ which correspond to the restriction of an inert map $\vphi : [i] \to [n]$ to its only two subintervals isomorphic to $[i-1]$, namely $\{0, \dotsc, i-1\}$ and $\{1, \dotsc, i-1\}$.

The category $\bsd^n$ has a nice geometric interpretation as well: its objects are the vertices of the barycentric subdivision of the standard $n$-simplex $\Delta^n$. Each object $[i] \hookrightarrow [n]$ corresponds to a distinct embedding of the $i$-simplex $\Delta^i$ into $\Delta^n$. The subcategory $\bvx^n$ consists of all vertices of $\Delta^n$, i.e. the embeddings $\Delta^0 \hookrightarrow \Delta^n$. For small values of $n$, $\bsd^n$ and $\bvx^n$ (spanned by the red asterisks) look like this:
\begin{figure}[h!]
	\begin{tikzpicture}
		\node (P) at (-5,3) {$\red{\ast}$};
		\node at (-6,3.08) {$\bsd^0:$};
		\node (A0) at (-6,1) {$\red{\ast}$};
		\node (A1) at (-4,1) {$\red{\ast}$};
		\node (A01) at (-5,1) {$\bullet$};
		\node at (-7,1.08) {$\bsd^1:$};
		\node (B0) at (-2,0) {$\red{\ast}$};
		\node (B1) at (2,0) {$\red{\ast}$};
		\node (B2) at (0,{2*sqrt(3)}) {$\red{\ast}$};
		\node (B01) at (0,0) {$\bullet$};
		\node (B12) at (1,{sqrt(3)}) {$\bullet$};
		\node (B02) at (-1,{1*sqrt(3)}) {$\bullet$};
		\node (B012) at (0,{2*sqrt(3)/3}) {$\bullet$};
		\node at (-2.5,2) {$\bsd^2:$};
		
		\draw[->] (A01) -- (A0);
		\draw[->] (A01) -- (A1);
		\draw[->] (B01) -- (B0);
		\draw[->] (B01) -- (B1);
		\draw[->] (B12) -- (B1);
		\draw[->] (B12) -- (B2);
		\draw[->] (B02) -- (B0);
		\draw[->] (B02) -- (B2);
		\draw[->] (B012) -- (B0);
		\draw[->] (B012) -- (B1);
		\draw[->] (B012) -- (B2);
		\draw[->] (B012) -- (B01);
		\draw[->] (B012) -- (B12);
		\draw[->] (B012) -- (B02);
	\end{tikzpicture}
\end{figure}

The arrows denote restrictions along smaller faces. We can represent the elements of $\bsd^n$ as tuples of integers: an injection $\vphi : [i] \to [n]$ is equivalently an $(i+1)$-tuple $(r_0, \dotsc, r_i)$ of strictly increasing positive integers from $[n]$. The $1$-tuples $(r_0)$ correspond to elements of $\bvx^n$, and when $\alpha : [m] \to [n]$ is a map in $\bD$ the corresponding pushforward $\alpha_\ast : \bsd^m \to \bsd^n$ sends a tuple $(r_0, \dotsc, r_i)$ to the tuple $(\alpha(r_{0}), \dotsc, \alpha(r_{i}))$ where we remove repeated elements - that is, we first pass to the \emph{set} $\{\alpha(r_{0}), \dotsc, \alpha(r_{i})\}$ and then turn it into a tuple of size possibly smaller than $i+1$.

Note that there is an embedding $\bS^n \hookrightarrow \bsd^n$ since every inert map is an injection. In view of this embedding we will think of $\bsd^n$ as a generalization of $\bS^n$.

\subsubsection{Generalized spans}

Here is our main object of study for this section:
\begin{dfn}
	Let $\sD$ be an $\infty$-category. A \emph{generalized span} valued in $\sD$ is a functor
	\begin{equation*}
		F : \bS^{k_1, \dotsc, k_m} \times \bsd^{l_1, \dotsc, l_n} = \bS^{k_1} \times \dotsb \times \bS^{k_m} \times \bsd^{l_1} \times \dotsb \times \bsd^{l_n} \to \sD.
	\end{equation*}
\end{dfn}
Using the terminology of Haugseng, (iterated) spans valued in $\sD$ are functors $\bS^{k_1, \dotsc, k_m}  \to \sD$. Since $\bsd^p$ is a generalization of $\bS^p$ we allow ourselves to think of functors
\begin{equation*}
	 \bS^{k_1, \dotsc, k_m} \times \bsd^{l_1, \dotsc, l_n} \to \sD
\end{equation*}
as (iterated) generalized spans. When no confusion can arise we will just refer to these objects as \emph{spans} or \emph{generalized spans}, and when the indices don't play much of a role we will abbreviate the tuples $(k_1, \dotsc, k_m)$ and $(l_1, \dotsc, l_n)$ as $\conj{k}$ and $\conj{l}$.

\begin{dfn} \label{dfn:gen-span-functor}
	Let $\sD^\odot$ be a symmetric monoidal $\infty$-category. The functor 
	\begin{equation*}
		\GS(\sD^\odot)^{m, n} : \bG^\opp \times \bD^{m + n, \opp} \to \sCat_1
	\end{equation*}
	defined by
	\begin{equation*}
		(\brak{t}, [k_1], \dotsc, [k_m], [l_1], \dotsc, [l_n]) \mapsto \Fun(\bS^{k_1, \dotsc, k_m} \times \bsd^{l_1, \dotsc, l_n}, \sD^\odot_t)
	\end{equation*}
	unstraightens to a cocartesian fibration $\what{\GS}(\sD^\odot)^{m,n} \to \bG^\opp \times \bD^{m + n, \opp}$ called the \emph{generalized span fibration} or just the \emph{span fibration} for short.
\end{dfn}

We will need this general construction for \Cref{dfn:relative-push-pull-fibration} but until then we will only concentrate only on the case where $\sD^\odot = \sC^\times$ is the cartesian monoidal category $\sC$ -- see \Cref{sec:commutative-monoids} -- which is the easiest setting where \Cref{dfn:gen-cartesian-span} makes sense. We will abbreviate $\GS(\sC^\times)$ as $\GS(\sC)$ from now on.

Recall that $\sC^\times_t \simeq \sC^{\times t}$ has finite limits by \Cref{prp:cart-mon-str-on-prod}, and recall from \Cref{sec:functors-vs-fibrations} that the elements of $\what{\GS}(\sC)^{m,n}$ are precisely the generalized spans $\bS^{\conj{k}} \times \bsd^{\conj{l}} \to \sC_t$.

\begin{dfn} \label{dfn:gen-cartesian-span}
	A span $F : \bS^{\conj{k}} \times \bsd^{\conj{l}} \to \sC^\times_{t}$ is \emph{cartesian} if $F$ is a right Kan extension of its restriction to the subcategory 
	\begin{equation*}
		\bL^{\conj{k}} \times \bvx^{\conj{l}} = \bL^{k_1, \dotsc, k_m} \times \bvx^{l_1, \dotsc, l_n} := \bL^{k_1} \times \dotsb \times \bL^{k_m} \times \bvx^{l_1} \times \dotsb \times \bvx^{l_n}.
	\end{equation*}
	Explicitly, using the formula for right Kan extensions (see for example \Cref{eqn:formula-right-Kan}), $F$ is cartesian if the canonical map
	\begin{equation*}
		F(\conj{x}, \conj{y}) \to \lim_{(\conj{u}, \conj{v}) \in (\bL^{\conj{k}} \times \bvx^{\conj{l}})_{/(\conj{x}, \conj{y})}} F(\conj{u}, \conj{v})
	\end{equation*}
	is an equivalence for every $(\conj{x}, \conj{y}) \in \bS^{\conj{k}} \times \bsd^{\conj{l}}$. Let $\what{\GS}(\sC)^{m,n, \rmcart} \sub \what{\GS}(\sC)^{m,n}$ denote the full subcategory of the span fibration spanned by the cartesian spans.
\end{dfn}

In the case $(m,n) = (1,0)$, a cartesian span valued in $\sD$ looks like a pyramid of spans in $\sD$ such that the top vertices of the spans in the $(i+1)$th row are obtained by taking pullbacks of the vertices in the $i$th row, for all $i \geq 1$. For example, when $k_1 = 3$ the model cartesian span looks like \Cref{fig:cart-span-1}. In the desired $(\infty, 1)$-category of spans this means that composition of morphisms (the spans in the second row from the bottom) is obtained by taking pullbacks.

\begin{figure}[h!]
	\begin{tikzpicture}
		\foreach \x in {0,...,3} {
		\draw node (A\x) at (3*\x,0) {$A_{\x}$};}
		\draw node (A01) at (1.5,1.5) {$A_{01}$};
		\draw node (A12) at (4.5,1.5) {$A_{12}$};
		\draw node (A23) at (7.5,1.5) {$A_{23}$};
		\draw node (A012) at (3,3) {$A_{01} \times_{A_1} A_{12}$};
		\draw node (A123) at (6,3) {$A_{12} \times_{A_2} A_{23}$};
		\draw node (A0123) at (4.5,4.5) {$A_{01} \times_{A_1} A_{12} \times_{A_2} A_{23}$};
		
		\draw[->] (A01) -- (A0);
		\draw[->] (A01) -- (A1);
		\draw[->] (A12) -- (A1);
		\draw[->] (A12) -- (A2);
		\draw[->] (A23) -- (A2);
		\draw[->] (A23) -- (A3);
		\draw[->] (A012) -- (A01);
		\draw[->] (A012) -- (A12);
		\draw[->] (A123) -- (A12);
		\draw[->] (A123) -- (A23);
		\draw[->] (A0123) -- (A012);
		\draw[->] (A0123) -- (A123);
	\end{tikzpicture}
	\caption{A cartesian span $\bS^3 \to \sD$. Each square is a pullback.}
	\label{fig:cart-span-1}
\end{figure}

When $m \geq 1$ we have such pyramids of spans in $m$ possible directions and each kind of composition is carried out by taking pullbacks.

In the case $(m,n) = (0,1)$, a cartesian span looks like an $l_1$-simplex each of whose $k$-dimensional faces is labelled by an object $B_{i_0} \times \dotsb \times B_{i_k} \in \sD$ corresponding to the product of the elements $B_i \in \sD$ associated to the vertices of the simplex. For example, when $l_1 = 2$ the model generalized cartesian span looks like the diagram in \Cref{fig:gen-cart-span-ex}.
\begin{figure}[h!]
	\begin{tikzpicture}[scale = 0.8]
		\node (B0) at (-4,0) {$B_0$};
		\node (B1) at (4,0) {$B_1$};
		\node (B2) at (0,{4*sqrt(3)}) {$B_2$};
		\node (B01) at (0,0) {$B_0 \times B_1$};
		\node (B12) at (2,{2*sqrt(3)}) {$B_1 \times B_2$};
		\node (B02) at (-2,{2*sqrt(3)}) {$B_0 \times B_2$};
		\node (B012) at (0,{4*sqrt(3)/3}) {$B_0 \times B_1 \times B_2$};
		
		\draw[->] (B01) -- (B0);
		\draw[->] (B01) -- (B1);
		\draw[->] (B12) -- (B1);
		\draw[->] (B12) -- (B2);
		\draw[->] (B02) -- (B0);
		\draw[->] (B02) -- (B2);
		\draw[->] (B012) -- (B0);
		\draw[->] (B012) -- (B1);
		\draw[->] (B012) -- (B2);
		\draw[->] (B012) -- (B01);
		\draw[->] (B012) -- (B12);
		\draw[->] (B012) -- (B02);
	\end{tikzpicture}
	\caption{A generalized cartsian span $\bsd^2 \to \sD$. Each $k$-dimensional face, for $0 \leq k \leq 3$, is a limit diagram.}
	\label{fig:gen-cart-span-ex}
\end{figure}

\subsubsection{Cartesian spans are functorial}

We ultimately want to prove that cartesian spans are assemble to a functor. More precisely, so as to invoke straightening-unsraightening we want exactly \Cref{crl:cart-spans-are-cocart-fibration}, which states that the projection $\what{\GS}(\sC)^{m,n, \rmcart} \to \bG^\opp \times \bD^{m+n, \opp}$ is a cocartesian fibration. The proof is split in three parts: for each morphism $\vphi$ in one of the three factors $\bG^\opp$, $\bD^{m,\opp}$, and $\bD^{n,\opp}$ we will show that $\vphi$ has a cocartesian lift in $\what{\GS}(\sC)^{m,n, \rmcart}$. The arguments used for each factor are different\footnote{The terminology is different as well and we hope the reader will not mind. For the first factor we talk about pushing forward by a map in $\bG^\opp$, while for the second and third factors we talk about pulling back by maps in $\bD$. This is because a span $F : \bS^{\conj{k}} \times \bsd^{\conj{l}} \to \sC^\times_t$ gets pulled back along maps to its domain and pushed forward along maps from its codomain.}  since they parametrize different kinds of spans. 

\emph{The first factor: $\bG^\opp$}. Since $\sC^\times_t \simeq \sC^{\times t}$ has a canonical cartesian monoidal structure where products are computed componentwise (\Cref{prp:cart-mon-str-on-prod}) we can measure the cartesian-ness of a span $F$ by passing to its components:

\begin{lmm} \label{lmm:cartesian-iff-components-are-cartesian}
	A generalized span $F$ valued in $\sC^\times_t$ is cartesian if and only if each component $F_i$, obtained by composing $F$ with the Segal map $(\tau_i)_\ast : \sC^\times_t \to \sC$, is cartesian.
\end{lmm}

\begin{proof}
	Limits in $\sC^\times_t$ are computed componentwise, meaning that $x \in \sC^\times_t$ is a limit of a diagram $D : \sA \to \sC_t$ if and only if the image of $x$ under the Segal equivalence $\sC^\times_t \simeq \sC^{\times t}$ is a limit of the composite diagram $\sA \xlra{D} \sC^\times_t \to \sC^{\times t}$ -- see part of the proof of \Cref{prp:cart-mon-str-on-prod}. The claim then follows because right Kan extensions are limits.
\end{proof}

\begin{prp} \label{prp:push-of-cart-is-cart}
	Let $F : \bS^{\conj{k}} \times \bsd^{\conj{l}} \to \sC^\times_s$ be a cartesian span and let $\psi : \brak{s} \to \brak{t}$ be a morphism in $\bG^\opp$. Then $\psi_\ast F : \bS^{\conj{k}} \times \bsd^{\conj{l}} \to \sC^\times_t$ is a cartesian span and there is a cocartesian arrow $F \to \psi_\ast F$ in $\what{\GS}(\sC)^{m,n,\rmcart}$ sitting over $\psi$.
\end{prp}

\begin{proof}
	There is a canonical cartesian arrow $F \to \psi_\ast F$ in $\what{\GS}(\sC)^{m,n}$, so it remains to be proven that $\psi_\ast F$ is a cartesian span. We have the following commutative square:
	\begin{equation*}
		\begin{tikzcd}
			\sC^{\times}_s \ar[r, "\psi_\ast"] \ar[d, "\simeq"'] & \sC^\times_t \ar[d, "\simeq"] \\
			\sC^{\times s} \ar[r, "\pi_\psi"] & \sC^{\times t}
		\end{tikzcd}
	\end{equation*}
	where the $i$th component of $\pi_\psi$, for $1 \leq i \leq t$, is given by
	\begin{equation*}
		(c_1, \dotsc, c_s) \mapsto \prod_{k \in \psi^\inv(i)} c_k.
	\end{equation*}
	Therefore the $i$th component of $\psi_\ast F$ is
	\begin{equation*}
		(\tau_i)_\ast(\psi_\ast F) \simeq \prod_{k \in \psi^\inv(i)} (\tau_k)_\ast F \simeq \prod_{k \in \psi^\inv(i)} F_k.
	\end{equation*}
	Since right Kan extensions commute with products we see that $(\tau_i)_\ast(\psi_\ast F)$ is a cartesian span in $\sC$ for every $i$, and so $\psi_\ast F$ is a cartesian span in $\sC^\times_t$ by \Cref{lmm:cartesian-iff-components-are-cartesian}.
\end{proof}

\emph{The second factor: $\bD^{m, \opp}$.} We start with a useful lemma.

\begin{lmm}[Generalization of {\cite[Lemma 5.6]{Haugseng2018}}] \label{lmm:inductive-cart-span}
	Let $\sA_0 \sub \sA$ and $\sB_0 \sub \sB$ denote two inclusions of $\infty$-categories and let $F : \sA \times \sB \to \sD$ be a functor into another $\infty$-category with finite limits. Then the following conditions are equivalent:
	\begin{enumerate}
		\item $F$ is a right Kan extension of $F\vert_{\sA_0 \times \sB_0}$;
		\item $F$ is a right Kan extension of $F\vert_{\sA_0 \times \sB}$ and for every $a_0 \in \sA_0$ the restriction $F\vert_{\{a_0\} \times \sB} : \sB \to \sD$ is a right Kan extension of $F\vert_{\{a_0\} \times \sB_0}$;  
		\item the mate $F' : \sA \to \Fun(\sB, \sD)$ is a right Kan extension of $F' \vert_{\sA_0}$ and for every $a_0 \in \sA_0$ the functor $F'(a_0) : \sB \to \sD$ is a right Kan extension of $F'(a_0)\vert_{\sB_0}$.
	\end{enumerate}
\end{lmm}

\begin{proof}
	The arguments of the proof in \cite{Haugseng2018} for functors $F : \bS^{n_1, \dotsc, n_k} \to \sD$ with $k = 2$, $\sA = \bS^{n_1}$, $\sB = \bS^{n_2}$, $\sA_0 = \bL^{n_1}$ and $\sB_0 = \bL^{n_2}$ do not depend on the values of $\sA, \sB, \sA_0, \sB_0$, and so the proof carries over identically for this more general case.
\end{proof}

\begin{prp} \label{prp:pull-of-cart-on-sigma-is-cart}
	Let $\sD$ be an $\infty$-category with finite limits. If $F : \bS^{\conj{k}} \times \bsd^{\conj{l}} \to \sD$ is a cartesian span and $\conj{\alpha} : [\conj{k'}] \to [\conj{k}]$ is a morphism in $\bD^{\times m}$, then
	\begin{equation*}
		(\conj{\alpha}, \id_{[\conj{l}]})^\ast F : \bS^{\conj{k'}} \times \bsd^{\conj{l}} \to \sD
	\end{equation*}
	is also a cartesian span and there is a cocartesian arrow $F \to (\conj{\alpha}, \id_{[\conj{l}]})^\ast F$ in $\what{\GS}(\sD)^{m,n,\rmcart}$ sitting over $(\conj{\alpha}, \id_{[\conj{l}]})$.
\end{prp}

\begin{proof}
	There is a cocartesian arrow $F \to (\conj{\alpha}, \id_{[\conj{l}]})^\ast F$ in $\what{\GS}(\sD)^{m,n}$, so we just have to prove that $(\conj{\alpha}, \id_{[\conj{l}]})^\ast F$ is a cartesian span. By applying \Cref{lmm:inductive-cart-span} to adjoin over the $\bsd^{\conj{l}}$ factor we get that (A) the mate
	\begin{equation*}
		F' : \bS^{\conj{k}} \to \Fun(\bsd^{\conj{l}}, \sD)
	\end{equation*}
	is cartesian and that (B) each image $F'(\conj{x}) : \bsd^{\conj{l}} \to \sD$ is cartesian, for $\conj{x} \in \bS^{\conj{k}}$. From (A), \cite[Proposition 5.9]{Haugseng2018} allows us to conclude that 
	\begin{equation*}
		\conj{\alpha}^\ast F' : \bS^{\conj{k'}} \to \Fun(\bsd^{\conj{l}}, \sD)
	\end{equation*} 
	is cartesian. For every $\conj{y} \in \bS^{\conj{k'}}$, if we set $\conj{x} = \conj{\alpha}_\ast \conj{y} \in \bS^{\conj{k}}$ we see that
	\begin{equation*}
		\conj{\alpha}^\ast F'(\conj{y}) \simeq F(\conj{x})
	\end{equation*}
	is cartesian by (B). \Cref{lmm:inductive-cart-span} then guarantees that the mate to $\conj{\alpha}^\ast F'$, which is precisely $(\conj{\alpha}, \id_{[\conj{l}]})^\ast F$, is cartesian.
\end{proof}

\emph{The third factor: $\bD^{n,\opp}$.} Pulling back the $\bsd^{\conj{l}}$ factors along maps in $\bD^{\times n}$ does \emph{not} preserve cartesian-ness, in general. In the simple case of $n = 1$, if $F : \bsd^{l} \to \sD$ is a cartesian span and $\vphi$ is an $i$-face of $\Delta^l$ then $F(\vphi)$ is a product of $i+1$ many terms which are all possibly distinct. When $\beta : [l'] \to [l]$ is not injective, say $\beta(a) = \beta(b)$, the pullback of $F$ along $\psi$ will mess up those faces of $\Delta^{l'}$ containing the vertices $a$ and $b$: explicitly, if $\vphi : [1] \to [l']$ is the embedding selecting $\{a, b\} \sub [l']$ and $\nu_a, \nu_b : [0] \to [l']$ select $\{a\}, \{b\} \sub [l']$ respectively, then
	\begin{align*}
		\beta^\ast F(\nu_x) & \simeq F(\nu_x) \text{ for $x = a, b$}, \\ 
		\beta^\ast F(\vphi) & \simeq F(\nu_a) \not \simeq F(\nu_a) \times F(\nu_b).
	\end{align*}

Fortunately, there is a fix.

\begin{dfn}
	Given an $\infty$-category $\sD$ with finite limits and a generalized span $F : \bS^{\conj{k}} \times \bsd^{\conj{l}} \to \sD$, the \emph{cartesian replacement} $\crep(F)$ is the span $\bS^{\conj{k}} \times \bsd^{\conj{l}} \to \sD$ defined by
	\begin{equation*}
		\crep(F) := \RKan_{I}(F \circ I),
	\end{equation*}
	where $I$ is the inclusion $\bL^{\conj{k}} \times \bvx^{\conj{l}} \sub \bS^{\conj{k}} \times \bsd^{\conj{l}}$.
\end{dfn}

Since $I$ is fully faithful we have $\crep(F) \circ I \simeq F \circ I$ (see the dual of \cite[Corollary 7.3.2.7]{kerodon}) and so we immediately see that $\crep(F)$ is a cartesian span, justifying the name:
\begin{equation*}
	\RKan_I (\crep(F) \circ I) \simeq \RKan_I (F \circ I) = \crep(F).
\end{equation*}
If $F$ is cartesian then obviously $\crep(F) \simeq F$. Moreover, the canonical natural transformation $F \to \crep(F)$ (obtained as the unit of $I^\ast \dashv \RKan_I$) determines an arrow $F \to \crep(F)$ in $\what{\GS}(\sD)^{m,n}$ sitting over the identity morphism of $([\conj{k}], [\conj{l}])$ in $\bD^{\times (m + n)}$ with the following property: if $F \to G$ is an arrow in $\what{\GS}(\sD)^{m,n}$ then there is an essentially unique arrow $\crep(F) \to \crep(G)$ in $\what{\GS}(\sD)^{m,n}$ making the square
\begin{equation*}
	\begin{tikzcd}
		F \ar[r] \ar[d] & G \ar[d] \\
		\crep(F) \ar[r] & \crep(G)
	\end{tikzcd}
\end{equation*}
commute. This follows immediately from the description of cartesian replacements as right Kan extensions.

\begin{prp} \label{prp:cart-extension}
	Let $\sD$ be an $\infty$-category with finite limits, let $F : \bS^{\conj{k}} \times \bsd^{\conj{l}} \to \sD$ be a cartesian generalized span, and let $\conj{\beta} : [\conj{l'}] \to [\conj{l}]$ be a morphism in $\bD^{\times n}$. Then the arrow $f : F \to \crep((\id_{[\conj{k}]},\conj{\beta})^\ast F)$ in $\what{\GS}(\sD)^{m,n, \rmcart}$ obtained by composing the two arrows
	\begin{equation*}
		F \to (\id_{[\conj{k}]},\conj{\beta})^\ast F, \quad (\id_{[\conj{k}]},\conj{\beta})^\ast F \to \crep((\id_{[\conj{k}]},\conj{\beta})^\ast F)
	\end{equation*}
	in $\what{\GS}(\sD)^{m,n}$ is a cocartesian arrow sitting over $(\id_{[\conj{k}]},\conj{\beta})$.
\end{prp}

\begin{proof}
	Set $G := \crep((\id_{[\conj{k}]},\conj{\beta})^\ast F)$. Let $(\conj{\gamma},\conj{\delta}) : ([\conj{k}'],[\conj{l}'']) \to ([\conj{k}],[\conj{l}'])$ be another map in $\bD^{\times (m+n)}$ and let $g : F \to K$ be an arrow in $\what{\GS}(\sD)^{m,n,\rmcart}$ sitting over $(\id_{[\conj{k}]}, \conj{\beta}) \circ (\conj{\gamma},\conj{\delta}) = (\conj{\gamma}, \conj{\beta \circ \delta})$. To prove that $f$ is a cocartesian arrow we must show that there is an essentially unique arrow $h : G \to K$ such that $h \circ f \simeq g$. The construction of $h$ is outlined in this commutative diagram:
	\begin{equation*}
		\begin{tikzcd}
			F \ar[d] \ar[dd, "f"', bend right = 80] \ar[rrd, "g", bend left] \ar[dr] & & \\
			(\id_{[\conj{k}]}, \beta)^\ast F \ar[r] \ar[d] & (\conj{\gamma}, \conj{\beta \circ \delta})^\ast F \ar[r] \ar[d] & K \ar[d, "\simeq"] \\
			G \ar[r, "h_1"'] \ar[rr, "h", bend right = 40] & H \ar[r, "h_2"'] & \crep(K)
		\end{tikzcd}
	\end{equation*}
	The vertical arrows in the large rectangle are the cartesian replacements and $H = \crep((\conj{\gamma}, \conj{\beta \circ \delta})^\ast F)$. The arrows $h_1$ and $h_2$ are uniquely determined by the universal property of the cartesian replacement. The top right triangle commutes by a basic property of cocartesian fibrations, see \Cref{sec:functors-vs-fibrations}. Finally, the right vertical arrow is an equivalence because $K$ is a cartesian span.
\end{proof}

We can combine these three propositions to prove the following:
\begin{crl} \label{crl:cart-spans-are-cocart-fibration}
	The projection $\pi : \what{\GS}(\sC)^{m,n,\rmcart} \to \bG^\opp \times \bD^{m+n, \opp}$ is a cocartesian fibration.
\end{crl}

\begin{proof}
	If we start with a cartesian span $F : \bS^{\conj{k}} \times \bsd^{\conj{l}} \to \sC^\times_s$ and morphisms 
	\begin{equation*}
		\psi : \brak{s} \to \brak{t}, \quad \conj{\alpha} : [\conj{k'}] \to [\conj{k}], \quad \conj{\beta} : [\conj{l'}] \to [\conj{l}]
	\end{equation*}
	in $\bG^\opp$, $\bD^{\times m}$, and $\bD^{\times n}$, respectively, then proving the claim amounts to showing that there is a cartesian span $G : \bS^{\conj{k}'} \times \bsd^{\conj{l}'} \to \sC^\times_t$ and a cocartesian morphism $F \to G$ in $\what{\GS}(\sC)^{m,n, \rmcart}$. We can find this chain of cocartesian morphisms in $\\what{GS}(\sC)^{m,n,\rmcart}$:
	\begin{enumerate}
		\item $F \to F' := \psi_\ast F$, where $F'$ is a cartesian span by \Cref{prp:push-of-cart-is-cart};
		\item $F' \to F'' := (\conj{\alpha}, \id_{[\conj{l}]})^\ast F'$, where $F''$ is a cartesian span by \Cref{prp:pull-of-cart-on-sigma-is-cart};
		\item $F'' \to G := \crep((\id_{[\conj{k'}]}, \conj{\beta})^\ast F'')$, where $G$ is a cartesian span by definition and the arrow is cocartesian by \Cref{prp:cart-extension}.
	\end{enumerate}
	The composition of the above morphisms gives a cocartesian morphism $F \to G$, as desired.
\end{proof}

\subsection{Generalized spans with local systems} \label{sec:gen-span-local-systems}

In this section we will extend the fibration $\what{\GS}(\sC)^{m,n}$ classifying generalized spans to a fibration $\what{\GS}(\sC)^{m,n}_{\twoslash \sigma}$ that classifies generalized spans with local systems valued in some section $\sigma$ of $\what{\GS}(\sC)^{m,n}$. The specific $\sigma$ we intend to use will be constructed later, but here is the general idea:

Say $P \in \sC$. A $P$-valued local system on an object $X \in \sC$ is simply a map $X \to P$. We think of $P$ as a classifying object of some sort, much like $BG$ classifies principal $G$-bundles in $\Spaces$ and $\sCat_1$ classifies cocartesian fibrations in a large $\infty$-category of $\infty$-categories. Now assume $P$ is a functor $\sD \to \sE$. A $P$-valued local system on a functor $F : \sD \to \sE$ is a natural transformation $F \to P$ of functors. In the setting of generalized spans,
\begin{itemize}
	\item $\sD$ is $\bS^{\conj{k}} \times \bsd^{\conj{l}}$,
	\item $\sE$ is $\sC_t^\times$,
	\item and the section $\sigma$ picks out a particular span $P_{t, \conj{k}, \conj{l}} := \sigma(\brak{t}, [\conj{k}], [\conj{l}]) : \bS^{\conj{k}} \times \bsd^{\conj{l}} \to \sC^\times_t$ for every element $(\brak{t}, [\conj{k}], [\conj{l}]) \in \bG^\opp \times \bD^{m + n}$
\end{itemize}
Then a vertex of $\what{\GS}(\sC)^{m,n}_{\twoslash \sigma}$ will be exactly a generalized span $F : \bS^{\conj{k}} \times \bsd^{\conj{l}} \to \sC_t^\times$ together with a natural transformation $F \to P_{t, \conj{k}, \conj{l}}$.

\subsubsection{From functors to sections} \label{sec:from-functors-to-sections}

We start with a useful proposition.

\begin{prp}[Generalization of {\cite[Proposition 7.3]{GHN2017}}] \label{prp:section-of-span-fib}
	Let $F : \sA \to \sCat_1$ and $G : \sB \to \sCat_1$ be functors which correspond to a cartesian fibration $\what{F} \to \sA^\opp$ and a cocartesian fibration $\what{G} \to \sB$. Define $S : \sA^\opp \times \sB \to \sCat_1$ by $S(a, b) = \Fun(F(a), G(b))$ and let $\pi : \what{S} \to \sA^\opp \times \sB$ be a corresponding cocartesian fibration. Then for any functor $H : \sD \to \sA^\opp \times \sB$ we have an equivalence between the space of functors $\sD \to \what{S}$ over $\sA^\opp \times \sB$ and the space of functors $\sD \times_{\sA^\opp} \what{F} \to \what{G}$ over $\sB$:
	\begin{equation*}
		\Map_{/\sA^\opp \times \sB}(\sD, \what{S}) \simeq \Map_{/\sB}(\sD \times_{\sA^\opp} \what{F}, \what{G}).
	\end{equation*}
	In particular, sections of $\what{S}$ are equivalently given by functors $\what{F} \to \Fun_{/\sB}(\sB, \what{G})$.
\end{prp}

\begin{proof}
	Consider the free cocartesian fibration 
	\begin{equation*}
		\mf{F}(H) : \mf{F}(\sD) \to \sA^\opp \times \sB
	\end{equation*}
	induced by the functor $H$, where $\mf{F}(\sD)$ is defined by the pullback
	\begin{equation*}
		\begin{tikzcd}
			\mf{F}(\sD) \ar[r] \ar[d] & \Fun(\Delta^1, \sA^\opp \times \sB) \ar[d, "\mathrm{source}"] \\
			\sD \ar[r, "H"] & \sA^\opp \times \sB 
		\end{tikzcd}
	\end{equation*}
	and $\mf{F}(H)$ is the composition 
	\begin{equation*}
		\mf{F}(\sD) \to \Fun(\Delta^1, \sA^\opp \times \sB) \xlra{\mathrm{target}} \sA^\opp \times \sB.
	\end{equation*} 
	This was originally defined for \emph{cartesian} fibrations in \cite[Definition 4.1]{GHN2017}, and here we reversed the appropriate source and target maps to obtain a version for cocartesian fibrations. The cocartesian variant of \cite[Theorem 4.5]{GHN2017} states that
	\begin{equation*}
		\Map_{/\sA^\opp \times \sB}(\sD, \what{S}) \simeq \Map^{\rmcocart}_{/\sA^\opp \times \sB}(\mf{F}(\sD), \what{S}) 
	\end{equation*}
	where the space on the right is that of maps of cocartesian fibrations. Notice, moreover, that $\mf{F}(H)$ classifies the overcategory functor
	\begin{equation*}
		\sD_{/(-)} := \sD \downarrow (-) : \sA^\opp \times \sB \to \sCat_1
	\end{equation*}
	defined on objects $(a, b) \in \sA^\opp \times \sB$ by the pullback square
	\begin{equation*}
		\begin{tikzcd}
			\sD_{/(a,b)} \ar[r] \ar[d] & (\sA^\opp \times \sB)_{/(a,b)} \ar[d] \\
			\sD \ar[r, "H"] & \sA^\opp \times \sB
		\end{tikzcd}
	\end{equation*}
	This can be seen by examining a fiber of $\mf{F}(H)$: we have a commutative diagram
	\begin{equation*}
		\begin{tikzcd}
			\ms{P}_{(a,b)} \ar[r] \ar[d] & \mf{F}(\sD) \ar[d] \ar[r] & \sD \ar[d, "H"] \\
			(\sA^\opp \times \sB)_{/(a,b)} \ar[r] \ar[d] & \Fun(\Delta^1, \sA^\opp \times \sB) \ar[r, "\mathrm{source}"] \ar[d, "\mathrm{target}"] & \sA^\opp \times \sB \\
			\Delta^0 \ar[r, "{\{(a,b)\}}"] & \sA^\opp \times \sB &  
		\end{tikzcd}
	\end{equation*}
	in which every square is a pullback; this implies that the top horizontal rectangle is a pullback, so in particular $\ms{P}_{(a,b)}$ is both the fiber of $\mf{F}(H)$ over $(a,b)$ and the pullback $\sD_{/(a,b)}$ defined above. Given this information, we can manipulate various end formulas (see \Cref{sec:ends}) to get
	\begin{align*}
		\Nat(\sD_{/(-)}, S(-)) & \simeq \int_{(a,b) \in \sA^\opp \times \sB} \Map(\sD_{/(a,b)}, \Fun(F(a), G(b))) \\
		& \simeq \int_{b \in \sB} \int_{a \in \sA^\opp} \Map((\sD_{/(a,b)}) \times F(a), G(b)) \\
		& \simeq \int_{b \in \sB} \Map \left( \int^{a \in \sA^\opp} ((\sD \times_{\sA^\opp} (\sA^\opp)_{/a}) \times F(a))_{/b}, G(b) \right) \\ 
		& \simeq \int_{b \in \sB} \Map \left( \left((\sD \times_{\sA^\opp} \left( \int^{a \in \sA^\opp} F(a) \times (\sA^\opp)_{/a} \right) \right)_{/b}, G(b) \right) \\
		& \simeq \int_{b \in \sB} \Map ((\sD \times_{\sA^\opp} \what{F})_{/b}, G(b)) \\
		& \simeq \Nat((\sD \times_{\sA^\opp} \what{F})_{/(-)}, G(-))
	\end{align*}
	where the coend formula
	\begin{equation*}
		\what{F} \simeq \int^{a \in \sA^\opp} F(a) \times (\sA^\opp)_{/a} = \colim \left(\Tw(\sA)^\opp \to \sA \times \sA^\opp \xlra{F(-) \times (\sA^\opp)_{/(-)}} \sCat_1 \right)
	\end{equation*}
	appears as \cite[Corollary 7.6]{GHN2017} with slightly different notation\footnote{In \cite{GHN2017} the twisted arrow category $\Tw(\sA)$ is defined such that $\Tw(\sA) \to \sA \times \sA^\opp$ is a right fibration, while for us $\Tw(\sA) \to \sA^\opp \times \sA$ is a left fibration -- see \Cref{sec:ends}. After a careful rearrangement of (co)ends, opposites, and (co)limits we obtain the desired formula.}. Hence we have
	\begin{align*}
		\Map_{/\sA^\opp \times \sB}(\sD, \what{S}) & \simeq \Map^{\rmcocart}_{/\sA^\opp \times \sB}(\mf{F}(\sD), \what{S}) \\
		& \simeq \Nat(\sD_{/(-)}, S(-)) & \text{(straightening)} \\
		& \simeq \Nat((\sD \times_{\sA^\opp} \what{F})_{/(-)}, G(-)) \\
		& \simeq \Map^{\rmcocart}_{/\sB}(\mf{F}(\sD \times_{\sA^\opp} \what{F}), \what{G}) & \text{(unstraightening)} \\
		& \simeq \Map_{/\sB}(\sD \times_{\sA^\opp} \what{F}, \what{G}) & \text{(free $\dashv$ forgetful).}
	\end{align*}
	To prove the last assertion let $H = \id_{\sA^\opp \times \sB}$. Then $\Map_{/\sA^\opp \times \sB}(\sA^\opp \times \sB, \what{S})$ is the space of sections of $\what{S}$ and
	\begin{align*}
		\Map_{/\sA^\opp \times \sB}(\sA^\opp \times \sB, \what{S}) & \simeq \Map_{/\sB}((\sA^\opp \times \sB) \times_{\sA^\opp} \what{F}, \what{G}) \\
		& \simeq \Map_{/\sB}(\sB \times \what{F}, \what{G}) \\
		& \simeq \Map(\what{F}, \Fun_{/\sB}(\sB, \what{G})). \qedhere
	\end{align*}
\end{proof}

Set 
\begin{equation*}
	F = \bS^{\bullet, \dotsc, \bullet} \times \bsd^{\bullet, \dotsc, \bullet} : \bD^{m + n} \to \bCat_1 \sub \sCat_1, \quad G = \sC^\times : \bG^\opp \to \sCat_1
\end{equation*}
Then we can use \Cref{prp:section-of-span-fib} and \Cref{prp:calg-are-sections} to construct a section of $\what{S} = \what{\GS}(\sC)^{m,n}$ from a functor
\begin{equation*}
	\what{\bS^{\bullet, \dotsc, \bullet}} \times \what{\bsd^{\bullet, \dotsc, \bullet}} \to \CAlg(\sC) \to \Fun_{/\bG^\opp}(\bG^\opp, \what{\sC^\times})
\end{equation*}
with domain the category of elements of $\bS^{\bullet, \dotsc, \bullet} \times \bsd^{\bullet, \dotsc, \bullet}$. This is the product of the categories of elements of $\bS^\bullet$ and $\bsd^\bullet$:
\begin{itemize}
	\item the objects of $\what{\bS^\bullet}$ are pairs $([n], \vphi : [i] \to [n])$ where $\vphi$ is inert, and a morphism $([m], \vphi) \to ([n], \psi)$ is pair consisting of a morphism $\alpha : [n] \to [m]$ in $\bD$ and a morphism $\vphi \to \alpha_\ast \psi$ in $\bS^m$;
	\item the objects of $\what{\bsd^\bullet}$ are pairs $([n], \vphi : [i] \to [n])$ where $\vphi$ is injective, and a morphism $([m], \vphi) \to ([n], \psi)$ is pair consisting of a morphism $\beta : [n] \to [m]$ in $\bD$ and a morphism $\vphi \to \beta_\ast \psi$ in $\bsd^m$.
\end{itemize}

Note now that there are ``face'' functors
\begin{equation*}
	\begin{tikzcd}[row sep = 0.2em, column sep = small]
		\what{\bS^\bullet} \ar[r, "f_\bS"] & \bD^\opp & \phantom{aaa} & \what{\bsd^\bullet} \ar[r, "f_{\bsd}"] & \bD^\opp \\
		([m], \vphi : [i] \to [m]) \ar[r, mapsto] \ar[ddd, "{(\alpha, \mu)}"] & {[i]} & & ([m], \vphi : [i] \to [m]) \ar[r, mapsto] \ar[ddd, "{(\beta, \nu)}"] & {[i]} \\
		\phantom{a} & & & & \\
		\phantom{a} & & & & \\
		([n], \psi : [j] \to [n]) \ar[r, mapsto] & {[j]} \ar[uuu, "f_\bS(\alpha)"'] & & ([n], \psi : [j] \to [n]) \ar[r, mapsto] & {[j]} \ar[uuu, "{f_{\bsd}(\beta)}"']
	\end{tikzcd}
\end{equation*}
where
\begin{equation*}
	f_\bS(\alpha)(r) = \alpha(\psi(r)) - \vphi(0), \qquad f_\bsd(\beta)(r) = \max \{s \in [i] \mid \vphi(s) \leq \beta(\psi(r))\}.
\end{equation*}

It is worth pointing out that, since $\bS^n \sub \bsd^n$, the total space $\what{\bS^\bullet}$ embeds into $\what{\bsd^\bullet}$ and the face functor $f_\bS$ is exactly $f_{\bsd}$ restricted to $\what{\bS^\bullet}$, as can be observed by calculating $f_\bsd(\beta)$ when $\vphi$ and $\psi$ are inert morphisms. We decided to keep the constructions separate to highlight the difference that ordinary spans and generalized spans play in the paper.

The functor $\bS^{\bullet, \dotsc, \bullet} \times \bsd^{\bullet, \dotsc, \bullet} : \bD^{\times (m + n)} \to \bCat_1$ factors as
\begin{equation*}
	\bD^{\times (m + n)} \to \bCat_1^{\times (m + n)} \xlra{\times} \bCat_1
\end{equation*}
and so its Grothendieck construction is obtained by taking the product of the individual total spaces:
\begin{equation*}
	\what{\bS^{\bullet, \dotsc, \bullet}} \times \what{\bsd^{\bullet, \dotsc, \bullet}} \cong \what{\bS^\bullet}^{\times m} \times \what{\bsd^\bullet}^{\times n}
\end{equation*}
It follows that we have functors
\begin{equation*}
	f^{m,n} := f_\bS^{\times m} \times f_{\bsd}^{\times n} : \what{\bS^{\bullet, \dotsc, \bullet}} \times \what{\bsd^{\bullet, \dotsc, \bullet}} \to \bD^{m + n, \opp}.
\end{equation*}
Composing $f^{m,n}$ with a functor $\cP : \bD^{m + n, \opp} \to \CAlg(\sC) \to \Fun_{\bG^\opp}(\bG^\opp, \what{\sC^\times})$ will yield a section $\sigma_{\cP}$ of $\what{\GS}(\sC)^{m,n}$. This functor $\cP$ will be constructed later, so for now we will work in complete generality assuming that we have such a $\sigma_{\cP}$.

\begin{prp} \label{prp:GS-over-section}
	There is a cocartesian fibration $\what{\GS}(\sC)^{m,n}_{\twoslash \sigma_\cP} \to \bG^\opp \times \bD^{m+n, \opp}$ with the following properties:
	\begin{enumerate}
		\item a vertex $(F, \eta) \in \what{\GS}(\sC)^{m,n}_{\twoslash \sigma_\cP}$ in the fiber over $(\brak{t}, [\conj{k}], [\conj{l}])$ is a generalized span $F : \bS^{\conj{k}} \times \bsd^{\conj{l}} \to \sC^\times_t$ together with a natural transformation
		\begin{equation*}
			\eta : F \to \sigma_\cP(t,\conj{k}, \conj{l})
		\end{equation*}
		of functors $\bS^{\conj{k}} \times \bsd^{\conj{l}} \to \sC^\times_t$;
		\item $\what{\GS}(\sC)^{m,n}_{\twoslash \sigma_\cP}$ comes with a functor to $\what{\GS}(\sC)^{m,n}$ that remembers the underlying generalized span $F$.
	\end{enumerate}
\end{prp}

\begin{proof}
	This follows from the much more general statement in \cite[Lemma 6.5]{Haugseng2018}. Explicitly, given a cocartesian fibration $\sE \to \sB$ and a section $s : \sB \to \sE$ there is a cocartesian fibration $\sE_{\twoslash s} \to \sB$ with the universal property that, for any functor $\sD \to \sB$, the square
	\begin{equation*}
		\begin{tikzcd}
			\Fun_{/\sB}(\sD, \sE_{\twoslash s}) \ar[r] \ar[d] & \Fun_{/\sB}(\sD \times [1] \sqcup_{\sD \times \{1\}} \sB, \sE) \ar[d] \\
			\{s\} \ar[r] & \Fun_{/\sB}(\sB, \sE)
		\end{tikzcd}
	\end{equation*}
	is a pullback. After setting $\sD \to \sB$ equal to the functor $\Delta^0 \to \sB$ picking out a point $b \in \sB$ the diagram tells us that the vertices of $\sE_{\twoslash s}$ over $b$ correspond to vertices $e \in \sE_b$ together with an arrow $e \to s(b)$ in $\sE_b$, proving (1). After setting $\sD = \sE_{\twoslash s}$ we can consider the identity functor $\sD \to \sE_{\twoslash s}$ on the top left corner of the diagram, which corresponds to a functor $(\sE_{\twoslash s} \times [1]) \sqcup_{\sE_{\twoslash s} \times \{1\}} \sB \to \sE$ on the top right corner; restricting the latter to $\sE_{\twoslash s} \times \{0\}$ yields the desired forgetful functor $\sE_{\twoslash s} \to \sE$, proving (2). The claim then follows by taking $\sE = \what{\GS}(\sC)^{m,n}$, $\sB = \bG^\opp \times \bD^{m+n, \opp}$, and $s = \sigma_\cP$.
\end{proof}

\subsubsection{Generalized spans with local systems and cartesian base}

Together with the results of the previous section we get a diagram
\begin{equation*}
	\begin{tikzcd}
		& \what{\GS}(\sC)^{m,n}_{\twoslash \sigma_\cP} \ar[d] \ar[ddr] & \\
		\what{\GS}(\sC)^{m,n, \rmcart} \ar[r, hook] \ar[rrd] & \what{\GS}(\sC)^{m,n} \ar[dr] \\
		& & \bG^\opp \times \bD^{m+n, \opp}
	\end{tikzcd}
\end{equation*}
where the diagonal arrows are cocartesian fibrations and the two orthogonal arrows preserve cocartesian morphisms.

\begin{dfn} \label{dfn:fib-span-with-local-system}
	Denote by $\what{\GS}(\sC; \cP)^{m,n} \to \bG^\opp \times \bD^{m+n, \opp}$ a cocartesian fibration obtained by taking the pullback of the above diagram:
	\begin{equation*}
		\begin{tikzcd}
			\what{\GS}(\sC; \cP)^{m,n} \ar[d] \ar[r] & \what{\GS}(\sC)^{m,n}_{\twoslash \sigma_\cP} \ar[d] \ar[ddr] & \\
			\what{\GS}(\sC)^{m,n, \rmcart} \ar[r, hook] \ar[rrd] & \what{\GS}(\sC)^{m,n} \ar[dr] \\
			& & \bG^\opp \times \bD^{m+n, \opp}
		\end{tikzcd}
	\end{equation*}
	This contains generalized spans in $\sC$ with local systems determined by $\sigma_\cP$ and with an underlying cartesian span.
\end{dfn}

One immediate consequence of this definition is that given $F \in \what{\GS}(\sC)^{m,n, \rmcart}_{t, \conj{k}, \conj{l}}$, the space of local systems that can be put on $F$ (i.e. the fiber of the left vertical map) is
\begin{equation*}
	\Nat(F, \sigma_\cP(t, \conj{k}, \conj{l})).
\end{equation*}
It follows that if $G_1$ and $G_2$ share the same underlying cartesian span $F$ then $G_1 \simeq G_2$ if and only if their local systems are equivalent.

In our applications the section $\sigma_\cP$ does not factor through cartesian spans so we can only assume that the base is cartesian -- compare this with \cite[Lemma 6.4]{Haugseng2018}. However, in the case $n = 0$ if we look at the fibers over $\brak{1} \in \bG^\opp$ we recover the constructions of \cite[Sections 5, 6]{Haugseng2018} (see there for the notation): 
\begin{equation*}
	\what{\GS}(\sC)^{m,0}_1 = \conj{\mathrm{SPAN}}_m^+(\sC), \quad \what{\GS}(\sC)^{m,0, \rmcart}_1 = \mathrm{SPAN}_m^+(\sC),
\end{equation*}
and 
\begin{equation*}
	(\what{\GS}(\sC)^{m,0}_{\twoslash  \sigma_{\cP}})_1 =
	\begin{cases}
		\mathrm{SPAN}_m^+(\sC;\cP) & \text{$\cP$ is an $m$-uple category object in $\sC$,} \\
		\conj{\mathrm{SPAN}}_m^+(\sC;\cP) & \text{else}.
	\end{cases}
\end{equation*}

When $n \neq 0$ is arbitrary we can recover \cite[Proposition 6.7]{Haugseng2018} when $\cP$ is only partially Segal.
\begin{prp} \label{prp:P-Segal-implies-GS-Segal}
	Assume $\cP$ satisfies the Segal condition in the first $m$ variables from $\bD^\opp$. Then so does $\what{\GS}(\sC)^{m,n}_{\twoslash  \sigma_{\cP}}$ -- more precisely, any functor classified by the cocartesian fibration $\what{\GS}(\sC)^{m,n}_{\twoslash  \sigma_{\cP}} \to \bG^\opp \times \bD^{m+n, \opp}$ satisfies the Segal condition in the first $m$ variables from $\bD^{\opp}$ and in the $\bG^\opp$ variable.
\end{prp}

\begin{proof}
	To verify the first part of the claim it will be enough to check the Segal condition on the fibers after we fix the remaining $n$ variables $[\conj{l}] \in \bD^{n, \opp}$ and the variable $\brak{t} \in \bG^\opp$. But then we can adjoin over the factors $\bsd^{\conj{l}}$ to obtain
	\begin{equation} \label{eqn:eqn1}
		(\what{\GS}(\sC)^{m,n}_{\twoslash  \sigma_{\cP}})_{t, \conj{\bullet}, \conj{l}} \simeq {\mathrm{SPAN}}_m^+(\Fun(\bsd^{\conj{l}}, \sC^\times_t); \cP(-, \conj{l})_t)_{\conj{\bullet}}
	\end{equation}
	because $\cP(-, [\conj{l}])_t : \bD^{m, \opp} \to \sC^\times_t \sub \what{\sC^\times}$ is an $m$-uple category object by assumption and so the section $\sigma_{\cP(-, [\conj{l}])_t}$ factors through the cartesian span fibration by \cite[Lemma 6.4]{Haugseng2018}. The result now follows from \cite[Proposition 6.7]{Haugseng2018}.
	
	From the formula in \Cref{eqn:eqn1} and from the fact that $\cP$ lands in $\CAlg(\sC)$ it is also easy to see that the Segal condition holds for the variable $\brak{t}$ as well, proving the second part of the claim.
\end{proof}

\subsection{Modifying symmetric monoidal category objects} \label{sec:modifying-smco}

In this section we will concentrate on the case $m = 2, n = 1$. Our goal is to construct a functor $\cP = \cQ^\otimes : \bG^\opp \times \bD^{2 + 1, \opp} \to \sC$ (which will adjoin over to a functor $\cQ^\otimes : \bD^{2 + 1, \opp} \to \CAlg(\sC)$) from which to build $\sigma_\cP$ using the results of the previous section. This functor is a modification of a symmetric monoidal category object $Q^\otimes : \bG^\opp \times \bD^\opp \to \sC$ which is better suited to serve as a local system for generalized spans, in particular for those indexed by $\bsd^\bullet$. 

\subsubsection{The path category as a modifying diagram} \label{sec:path-category}

We begin by introducing the path category, which will serve as the ``shape'' that we want part of $\cQ^\otimes$ to have.

\begin{dfn} \label{dfn:path-cat}
	Let $\Path(l) \in \bCat_2$ denote the \emph{path $2$-category} of the poset $[0,l] = \{0 < 1 < \dotsb < l\} \in \bCat_1$. This has objects the positive integers $0, 1, \dotsc, l$ and morphism categories
	\begin{equation*}
		\Hom_{\Path(l)}(i,j) :=
		\begin{cases}
			\{S \sub [i,j] \mid i,j \in S\} & i \leq j, \\
			\emptyset & i > j,
		\end{cases}
	\end{equation*} 
	where $\{S \sub [i,j] \mid i,j \in S\}$ is a poset under inclusion of subsets, and hence an ordinary category
\end{dfn}
The usual definition of $\Path(l)$ presents it as a simplicially enriched category -- see \cite[Definition 1.2.56]{Land2021} for example -- because of its usefulness in defining the simplicial nerve. Just by examining the definition we see that $\Path(0)$, $\Path(1)$, and $\Path(2)$ look like 
\begin{equation} \label{eqn:path-l}
	\begin{tikzcd}
 		& & & & & \bullet_2 & & \\
 		\bullet_0 & & \bullet_0 \arrow[r, "01"] & \bullet_1 & & & & \bullet_1 \arrow[llu, "12"'] \\
 		& & & & & \bullet_0 \arrow[rru, "01"'] \arrow[uu, "012"', bend right=40, ""{name = R, inner sep = 2pt, above}] \arrow[uu, "02", bend left=40, ""{name = L, inner sep = 2pt, below}] & & 
 		\ar[Rightarrow, from = L, to = R]
	\end{tikzcd}
\end{equation}
where the (non-trivial) $2$-morphism in $\Path(2)$ corresponds to the (non-degenerate) inclusion $\{0, 2\} \sub \{0,1,2\}$. In general, the mapping categories only depend on the difference $d = j-i$:
\begin{equation*}
	\Hom_{\Path(l)}(i,j) \cong (\Delta^1)^{\times (d-1)} \cong \Hom_{\Path(d)}(0,d).
\end{equation*}
When passing from $\Path(l)$ to $\Path(l+1)$, one only has to add a new mapping category (up to isomorphism) between $0$ and $l+1$, and the other mapping categories will be determined by the old ones.

Consider now the bi-cosimplicial object $(-,-) : \bD^{2} \to \bCat_2$ where $([m], [n])$ is the $2$-category generated by
\begin{align*}
	\text{objects: } & x_0, \dotsc, x_m; \\
	\text{$1$-morphisms: } & f_{i,j} : x_i \to x_{i+1}, & 0 \leq i \leq m-1, 0 \leq j \leq n; \\
	\text{$2$-morphisms: } & \alpha_{i,j} : f_{i,j} \to f_{i,j+1}, & 0 \leq i \leq m-1, 0 \leq j \leq n-1.
\end{align*}
The graph underlying $([m], [n])$ looks like $n+1$ parallel horizontal chains of $1$-morphisms between $m+1$ objects together with $m$ parallel vertical chains of $2$-morphisms:
\begin{equation} \label{eqn:m-n-chains}
	\begin{tikzcd}
		\bullet \arrow[rr, bend right=80, ""{name = D2, inner sep = 1pt, above}] \arrow[rr, bend right=30, ""{name = U2, inner sep = 1pt, below}] \arrow[rr, bend left=25, ""{name = D1, inner sep = 1pt, above}] \arrow[rr, bend left=80, ""{name = U1, inner sep = 1pt, below}] & & \bullet \arrow[rr, bend right=80, ""{name = D4, inner sep = 1pt, above}] \arrow[rr, bend right=30, ""{name = U4, inner sep = 1pt, below}] \arrow[rr, bend left=25, ""{name = D3, inner sep = 1pt, above}] \arrow[rr, bend left=80, ""{name = U3, inner sep = 1pt, below}] & & \bullet \ar[rr, "\dotsb" description, bend left = 80] \ar[rr, "\dotsb" description, bend left = 25] \ar[rr, "\dotsb" description, bend right = 30] \ar[rr, "\dotsb" description, bend right = 80] & & \bullet \arrow[rr, bend right=80, ""{name = D8, inner sep = 1pt, above}] \arrow[rr, bend right=30, ""{name = U8, inner sep = 1pt, below}] \arrow[rr, bend left=25, ""{name = D7, inner sep = 1pt, above}] \arrow[rr, bend left=80, ""{name = U7, inner sep = 1pt, below}] & & \bullet
 		\arrow[Rightarrow, from=U1, to=D1] \arrow[Rightarrow, from=U2, to=D2] \arrow[phantom, from=D1, to=U2, "\vdotsb" description] \arrow[Rightarrow, from=U3, to=D3] \arrow[Rightarrow, from=U4, to=D4] \arrow[phantom, from=D3, to=U4, "\vdotsb" description] \arrow[Rightarrow, from=U7, to=D7] \arrow[Rightarrow, from=U8, to=D8] \arrow[phantom, from=D7, to=U8, "\vdotsb" description]
 	\end{tikzcd}
\end{equation}

Much like $[-] : \bD \to \bCat_1$ is the cosimplicial object representing the nerve functor into simplicial sets, $(-, -)$ is the bi-cosimplicial object representing a $2$-categorical nerve into bisimplicial sets:
\begin{equation*}
	N : \bCat_2 \to \Fun(\bD^{2, \opp}, \Set), \quad N(\bC)_{m,n} := \Hom_{\bCat_2}(([m], [n]), \bC).
\end{equation*}
An $(m,n)$-simplex of $N(\bC)$ is given by $n+1$ horizontal sequences of chains of $m$ composable $1$-morphisms bounding $m$ vertical sequences of chains of $n$ composable $2$-morphisms, all in $\bC$.

\begin{dfn} \label{dfn:nerve-path-cat}
	Let $\bd : \bD^\opp \to \Fun(\bD^{2, \opp}, \Set)$ denote the functor $[l] \mapsto N(\Path(l))$. Explicitly,
	\begin{align*}
		\bd[l]_{m,n} & = \{\text{$(m,n)$-simplices of $\bd[l]$}\} \\
		& = \{\text{diagrams of shape $([m], [n])$ in $\Path(l)$}\}.
	\end{align*}
\end{dfn}
We want to think of $\bd[l]$ as an ``indexing diagram'', so we introduce the following construction:
\begin{dfn} \label{dfn:nerve-path-with-labels}
	For any $X : \bD^{2, \opp} \to \sC$, define
	\begin{equation*}
		\bd X : \bD^\opp \to \sC, \quad \bd X_l := \int_{\bD^{2, \opp}} \bd[l]_{\bullet,\bullet} \pitchfork X_{\bullet, \bullet} \in \sC.
	\end{equation*}
\end{dfn}
See \Cref{sec:ends} for a discussion on ends in $\infty$-categories. In particular, note that for any $c \in \sC$ we have an equivalence
\begin{equation*}
	\sC(c, \bd X_l) \simeq \int_{\bD^{2,\opp}} \sC(c, \bd[l]_{\bullet,\bullet} \pitchfork X_{\bullet, \bullet}) \simeq \int_{\bD^{2,\opp}} \sC(\bd[l]_{\bullet,\bullet}, \Spaces(c, X_{\bullet, \bullet})) \simeq \Nat(\bd[l], \sC(c,X))
\end{equation*}
between the space of maps $c \to \bd X_l$ and the space of natural transformations $\bd[l] \to \sC(c,X)$. Since $\bd[l]$ is valued in sets (finite sets, in fact) the natural transformations correspond to parametrized tuples of maps $c \to X$, meaning that for each $([u],[v]) \in \bD^{2,\opp}$ we have a corresponding tuple of maps $c \to X_{u,v}$ whose components correspond to the $(u,v)$-simplices of $\bd[l]$. 

This observation can be made more precise. Let $\bE := \bD^{2,\opp}_{\bd[l]/}$ fit into the pullback square
\begin{equation*}
	\begin{tikzcd}
		\bE \ar[r] \ar[d] & \Fun(\bD^{2,\opp}, \sC)^\opp_{\bd[l]/} \ar[d] \\
		\bD^{2,\opp} \ar[r, "y"] & \Fun(\bD^{2,\opp}, \sC)^\opp
	\end{tikzcd}
\end{equation*}
of ordinary categories, where $y$ is the co-Yoneda embedding defined on objects by $y(a) := \Hom_{\bD^{2,\opp}}(a,-)$. The objects of $\bE$ are pairs $((u,v), \sigma)$ where $([u],[v]) \in \bD^{2,\opp}$ and $\sigma \in \bd[l]_{u,v}$, and there is a morphism $((u,v), \sigma) \to ((u',v'), \sigma')$ between pairs for each morphism $\vphi : ([u], [v]) \to ([u'],[v'])$ in $\bD^{2,\opp}$ such that $\vphi^\ast (\sigma) = \sigma'$. We see from \Cref{crl:formula-for-end-of-cotensor} that
\begin{equation*}
	\bd X_l \simeq \lim (\bE \to \bD^{2,\opp} \xlra{\bd[l]} \sC) \simeq \lim_{\substack{(u,v) \in \bD^{2,\opp} \\ \sigma \in \bd[l]_{u,v}}} X_{u,v}.
\end{equation*}

It's hard in general to calculate $\bd X_l$ from this limit formula but we can simplify things considerably by showing that the limit only depends on the non-degenerate simplices of $\bd[l]$ and the face maps. We start with a lemma which helps narrow down the range of simplices to consider:

\begin{lmm} \label{lmm:range-of-nondegenerate-simplices}
	The bisimplicial set $\bd[l]$ has non-degenerate $(u,v)$-simplices only in the range $u + v \leq l$. That is, if $u + v > l$ then any $(u,v)$-simplex is degenerate.
\end{lmm}

\begin{proof}
	Recall that a $(u,v)$-simplex of $\bd[l]$ is a chain as in \Cref{eqn:m-n-chains} with $u+1$ many objects of $\Path(l)$, i.e. integers $i_0, \dotsc, i_u$ between $0$ and $l$, and a chain of $v$ many $2$-morphisms in the mapping set of any two consecutive objects, i.e. inclusions of subsets of $[0,l]$. If $i, j \in [0,l]$ then the longest chain of non-identity arrows in the category $\Hom_{\Path(l)}(i,j)$ (i.e. $2$-morphisms in $\Path(l)$) has length exactly $j-i-1$. The assumption that the simplex is non-degenerate implies, in particular, that there isn't a horizontal string of identity $2$-morphisms that can be collapsed. Hence $v$ is bounded above by the sum of the lengths of all the possible chains of $2$-morphisms:
	\begin{equation*}
		v \leq \sum_{k = 0}^{u-1} (i_{k+1} - i_k - 1) = i_u - i_0 - u.
	\end{equation*}
	But $i_u - i_0 \leq l$ and thus $u + v \leq u + l - u = l$.
\end{proof}

Now let $\bE^\star \sub \bE$ denote the full subcategory containing the pairs $((u,v), \sigma)$ where $\sigma$ is a non-degenerate simplex.

\begin{prp} \label{prp:non-deg-formula-dx}
	The inclusion $I : \bE^\star \to \bE$ is final, meaning that for any functor $p : \bE \to \sC$ we have an equivalence $\lim p \simeq \lim p \circ I$ induced by pre-composition with $I$. In particular
	\begin{equation*}
		\bd X_l \simeq \lim ( \bE^\star \to \bE \to \bD^{2,\opp} \xlra{\bd[l]} \sC) \simeq \lim_{\substack{u + v \leq l \\ \sigma \in \bd[l]^\star_{u,v}}} X_{u,v},
	\end{equation*}
	where the superscript $\star$ on $\bd[l]$ indicates that we are only considering the non-degenerate simplices.
\end{prp}

\begin{proof}
	The recognition theorem for final functors (dual to \cite[4.1.3.1]{HTT2009}) states that $I$ is final precisely if, for every $x \in \bE$, the $\infty$-category $\bE^\star \times_{\bE} \bE_{/x}$ is weakly constractible. It will be enough to prove that it has a terminal object. The objects of $\bE^\star \times_{\bE} \bE_{/x}$ are arrows $y \to x$ in $\bE$ with $y \in \bE^\star$ and the morphisms are arrows $y \to y'$ in $\bE$ over $x$. In our case $x$ corresponds to some specified simplex $\sigma \in \bd[l]_{u,v}$. Then there is a \emph{unique} non-degenerate simplex $\wilde{\sigma}$ and a map $\vphi$ in $\bD^{2,\opp}$ such that $\vphi$ is the identity or a degeneracy map (meaning that the two components of $\vphi$ in $\bD$ are surjective) and $\vphi^\ast (\wilde{\sigma}) = \sigma$. Let $y$ be the object of $\bE$ corresponding to the simplex $\wilde{\sigma}$. It is easy to see that $\vphi : y \to x$ is a terminal object in $\bE^\star \times_{\bE} \bE_{/x}$: if $\psi : y' \to x$ is another object of $\bE^\star \times_{\bE} \bE_{/x}$ corresponding to some non-degenerate simplex $\tau$ with $\psi^\ast(\tau) = \sigma$ then $\psi$ can be factored as a face map $\psi_1$ followed by a degeneracy map $\psi_2$; but faces of non-degenerate simplices are non-degenerate in $\bd[l]$ (as each morphism category in $\Path(l)$ is a poset) so we have that $\psi_1^\ast(\tau)$ is non-degenerate and thus must be $\wilde{\sigma}$ by the uniqueness property of $\wilde{\sigma}$, meaning that $\psi$ factors through $y$.
	
	The condition $u + v \leq l$ in the limit follows directly from \Cref{lmm:range-of-nondegenerate-simplices}.
\end{proof}

Note that the morphisms in $\bE^\star$ correspond to face maps so $\bE^\star$ does not contain any degeneracy maps. Using \Cref{prp:non-deg-formula-dx} we can explicitly calculate $\bd X_l$ for small values of $l$. It is helpful to refer to the drawings in \Cref{eqn:path-l} when following these computations.
\begin{itemize}
	\item $\bd[0]$ is a point and only has one $(0,0)$-simplex. Hence
	\begin{equation} \label{eqn:dx0}
		\bd X_0 \simeq X_{0,0}.
	\end{equation}
	\item $\bd[1]$ is a free edge and only has one $(1,0)$-simplex with two $(0,0)$-simplices as faces. Hence
	\begin{equation} \label{eqn:dx1}
		\bd X_1 \simeq \lim (X_{0,0} \leftarrow X_{1,0} \to X_{0,0}) \simeq X_{1,0}.
	\end{equation}
	\item $\bd[2]$ has three $(0,0)$-simplices connected by a commutative triangle of $(1,0)$-simplices (i.e. one $(2,0)$-simplex) with composite $f$, one additional $(1,0)$-simplex $g$, and one $(1,1)$-simplex from $g$ to $f$. Hence
	\begin{equation} \label{eqn:dx2}
		\bd X_2 \simeq \lim \left(
		\begin{tikzcd}
			X_{2,0} \ar[d, "\mathrm{legs}"'] \ar[r, "\mathrm{composite}"] & X_{1,0} & X_{1,1} \ar[l, "t"'] \ar[d, "s"] \\
			X_{1,0} \times_{X_{0,0}} X_{1,0} \ar[r, "{(s,t)}"] & X_{0,0}^{\times 2} & X_{1,0} \ar[l, "{(s,t)}"']
		\end{tikzcd}
		\right) \simeq X_{2,0} \times_{X_{1,0}} X_{1,1}
	\end{equation}
	where $s$ and $t$ are the appropriate source and target maps.
	\item For a general $l \geq 3$, $\bd[l]$ contains the $(l,0)$-simplex corresponding to the composable chain
	\begin{equation*}
		(\{0, 1\}, \{1, 2\}, \dotsc, \{l-1, l\})
	\end{equation*}
	of $1$-morphisms in $\Path(l)$ and the $(1,1)$-simplex corresponding to the inclusion
	\begin{equation*}
		\{0, l\} \sub [0,l]
	\end{equation*}
	representing a $2$-morphism in $\Hom_{\Path(l)}(0, l)$. Since the composition of the $(l,0)$-simplex and the target of the $(1,1)$-simplex coincide we obtain a map
	\begin{equation*}
		\bd X_l \to X_{l, 0} \times_{X_{1,0}} X_{1,1}.
	\end{equation*}
	In particular we have a projection
	\begin{equation} \label{eqn:xi}
		\xi : \bd X_l \to X_{1,1}
	\end{equation}
	onto the second factor.
\end{itemize}

We end this section with an inductive result describing the structure of $\bd X_l$, which is really a convoluted (but useful) reformulation of the statement that $\Path(l)$ is obtainable from copies of $\Path(l-1)$ together with a new mapping set isomorphic to $(\Delta^{1})^{\times (l-1)}$.
\begin{lmm} \label{lmm:inductive-structure-dxl}
	Assume $X : \bD^{2,\opp} \to \sC$ is a $2$-category object in $\sC$. Let $l \geq 3$. Denote by $\wp(l)$ the poset of non-empty subsets of $[0,l]$ and fix a set $\{c_0, \dotsc, c_l\}$ of $l+1$ many elements of $\sC$. There are two functors
	\begin{align*}
		A_l : \wp(l)^\opp & \to \sC, \quad S \mapsto \bd X_{\vnorm{S}-1}, \\
		B_l : \wp(l)^\opp & \to \sC, \quad S \mapsto \prod_{i \in S} c_i.
	\end{align*}
	Both restrict to functors $A_l'$ and $B_l'$ defined on $\wp(l)^\opp \setminus \{[0,l]\} \sub \wp(l)^\opp$. Then the canonical restriction map
	\begin{equation*}
		\Nat(B_l, A_l) \to \Nat(B_l', A_l')
	\end{equation*}
	has empty or contractible fibers.
\end{lmm}

\begin{proof}
	Note that $\wp(l)^\opp$ is a cone on $\wp(l)^\opp \setminus \{[0,l]\}$, where $[0,l]$ is the initial object. From \Cref{crl:nat-trans-inductive} we have a pullback square
	\begin{equation*}
		\begin{tikzcd}
			\Nat(B_l,A_l) \ar[r] \ar[d] & \Nat(B_l', A_l') \ar[d] \\
			\displaystyle \sC \left( \prod_{i \in [0,l]} c_i, \bd X_l \right) \ar[r, "\zeta"] & \displaystyle \lim_{S \in \wp(l)^\opp \setminus \{[0,l]\}} \sC \left( \prod_{i \in [0,l]} c_i, \bd X_{\vnorm{S}-1} \right)
		\end{tikzcd}
	\end{equation*}
	It will be enough to prove that $\zeta$, the bottom map in the diagram, has empty or contractible fibers.
	
	Let $c = \prod_{i \in [0,l]} c_i$. From the limit formula in \Cref{prp:non-deg-formula-dx} we know that there are equivalences
	\begin{align*}
		\sC \left( c, \bd X_l \right) & \simeq \lim_{\substack{u + v \leq l \\ \sigma \in \bd[l]_{u,v}^\star}} \sC \left( c, X_{u,v} \right), \\
		\sC \left( c, \bd X_{\vnorm{S}-1} \right) & \simeq \lim_{\substack{u + v \leq \vnorm{S}-1 \\ \sigma \in \bd[\vnorm{S}-1]_{u,v}^\star}} \sC \left( c, X_{u,v} \right)
	\end{align*}
	and since $X$ satisfies the Segal condition we can simplify these limits as
	\begin{align*}
		\lim_{\substack{u + v \leq l \\ \sigma \in \bd[l]_{u,v}^\star}} \sC \left( c, X_{u,v} \right) & \simeq \lim_{\substack{u, v \in \{0,1\} \\ \sigma \in \bd[l]_{u,v}^\star}} \sC \left( c, X_{u,v} \right), \\
		 \lim_{\substack{u + v \leq \vnorm{S}-1 \\ \sigma \in \bd[\vnorm{S}-1]_{u,v}^\star}} \sC \left( c, X_{u,v} \right) & \simeq \lim_{\substack{u,v \in \{0,1\} \\ \sigma \in \bd[\vnorm{S}-1]_{u,v}^\star}} \sC \left( c, X_{u,v} \right).
	\end{align*}
	We thus have a commutative diagram
	\begin{equation*}
		\begin{tikzcd}
			\sC(c, \bd X_l) \ar[d, "\simeq"'] \ar[r, "\zeta"] & \displaystyle \lim_{S \in \wp(l)^\opp \setminus \{[0,l]\}} \sC \left( c, \bd X_{\vnorm{S}-1} \right) \ar[d, "\simeq"] \\
			\displaystyle \lim_{\substack{u, v \in \{0,1\} \\ \sigma \in \bd[l]_{u,v}^\star}} \sC \left( c, X_{u,v} \right) \ar[r, "\rho"] & \displaystyle \lim_{\substack{S \in \wp(l)^\opp \setminus \{[0,l]\} \\ u,v \in \{0,1\} \\ \sigma \in \bd[\vnorm{S}-1]_{u,v}^\star}} \sC \left( c, X_{u,v} \right)
		\end{tikzcd}
	\end{equation*}
	where $\rho$ is induced by the restrictions along inclusions $S \subset [0,l]$. It is easy to see that the fibers of $\rho$ are empty or contractible by examining the low dimensional (non-degenerate) simplices of $\bd[l]$:
	\begin{itemize}
		\item the $(0,0)$-simplices are the objects of $\Path(l)$, namely the integers $i = 0, 1, \dotsc, l$, and these are obtained from the proper inclusions $S = \{i\} \sub [0,l]$;
		\item the $(1,0)$-simplices are the $1$-morphisms in $\Path(l)$, corresponding to subsets $S \sub [0,l]$; each $S$ is canonically expressed as a pushout of sets with two elements using the ordering on $\{0, 1, \dotsc, l\}$, namely
		\begin{equation*}
			S = \{a_1, \dotsc, a_s\} = \{a_1, a_2\} \cup_{\{a_2\}} \{a_2, a_3\} \cup_{\{a_3\}} \dotsb \cup_{\{a_{s-1}\}} \{a_{s-1}, a_s\}
		\end{equation*}
		(this is the horizontal composition of $1$-morphisms in $\Path(l)$) and the latter are obtained from proper inclusions $S = \{a_{j}, a_{j+1}\} \sub [0,l]$;
		\item there are no (non-degenerate) $(0,1)$-simplices since $\Path(l)$ is the nerve of a $2$-category;
		\item the $(1,1)$-simplices are the $2$-morphisms in $\Path(l)$, corresponding to inclusions $S \sub T$ of subsets of $[0,l]$ with the same boundary; these inclusions can be canonically factored as successive single-element inclusions, namely
		\begin{equation*}
			S \subset S \cup \{b_1\} \subset S \cup \{b_1, b_2\} \subset \dotsb \subset S \cup \{b_1, \dotsc, b_r\} = T
		\end{equation*}
		(this is the vertical composition of $2$-morphisms in $\Path(l)$) and each such inclusion is pulled back from the prototypical single-element inclusion $\{0,2\} \subset \{0,1,2\}$ of subsets of $[0,l]$.
	\end{itemize}
	Therefore if $f : c \to \bd X_l$ is a morphism in $\sC$ then $\rho(f)$ completely determines $f$ via pullbacks along proper inclusions $S \subset [0,l]$ -- in fact, the argument shows that it is enough to consider those $S$ that have at most $3$ elements, justifying our assumption that $l \geq 3$. This shows that $\rho$ has empty or contractible fibers and thus so does $\zeta$, concluding the proof.
\end{proof}

\subsubsection{The construction of $\cQ^\otimes$} \label{sec:construction-of-Q}

We will now move on to the construction of the promised section $\sigma$ of $\what{\GS}(\sC)^{2,1}$. Assume that we have a symmetric monoidal category object $Q^\otimes : \bG^\opp \times \bD^\opp \to \sC$, i.e. $Q^\otimes$ satisfies the Segal condition in both variables. In particular there is a representable functor
\begin{equation*}
	\sQ^\otimes : \sC^\opp \to \CAlg(\Cat(\Spaces)), \quad \sQ^\otimes(c)_{t,k} := \sC(c, Q^\otimes_{t,k}) 
\end{equation*}
that lands in symmetric monoidal category objects in $\Spaces$.

First, by \Cref{prp:underlying-monoid} and \Cref{prp:alg-to-alg-alg} we can upgrade $Q^\otimes$ to a symmetric monoidal $2$-category object $P^\otimes = \underline{W(Q^\otimes)} : \bG^\opp \times \bD^{2, \opp} \to \sC$ where the first $\bD^\opp$ factor encodes the underlying monoid of $Q^\otimes$:
\begin{equation*}
	P^\otimes_{t,m,n} := Q^\otimes_{tm, n}.
\end{equation*}
We can then consider the category object $P^\otimes_{t, m, \bullet} : \bD^\opp \to \sC$ obtained by fixing the first two entries in $P^\otimes$ and, following \Cref{dfn:n-cubes}, we can produce a symmetric monoidal $3$-simplicial object $\square^2 P^\otimes : \bG^\opp \times \bD^{3, \opp} \to \sC$ by setting
\begin{equation*}
	(\square^2 P^\otimes)_{t,m,n,p} := (\square^2 (P^\otimes_{t, m, \bullet}))_{n,p} = (\square^2 (Q^\otimes_{tm, \bullet}))_{n,p}.
\end{equation*}
Note that $\square^2 P^\otimes$ satisfies the Segal condition in all four variables, i.e. it is a symmetric monoidal $3$-category object in $\sC$. Define $\bd \square^2 P^\otimes : \bG^\opp \times \bD^{2, \opp} \to \sC$ by
\begin{equation*}
	(\bd \square^2 P^\otimes)_{t,l,n} := \bd((\square^2 P^\otimes)_{t, \bullet, \bullet, n})_l = \int_{\bD^{2, \opp}} \bd[l]_{\bullet, \bullet} \pitchfork (\square^2 P^\otimes)_{t, \bullet, \bullet, n},
\end{equation*}
where the expression $(\square^2 P^\otimes)_{t, \bullet, \bullet, n}$ inside the end denotes the $2$-category object obtained from $\square^2 P^\otimes$ by fixing the first and last entries and letting the second and third entries vary. Now $\bd \square^2 P^\otimes$ only satisfies the Segal condition in the first and last variable -- the middle two variables were integrated out and replaced by an  extra variable for which the Segal condition does not hold. Finally, let $\cQ^\otimes : \bG^\opp \times \bD^{2 + 1, \opp} \to \sC$ be given by swapping the second and third factors of $\bD^\opp$ and upgrading $\bd \square^2 P^\otimes$ to a symmetric monoidal category object in $\Fun(\bD^{2, \opp}, \sC)$ using the constructions of \Cref{prp:underlying-monoid} and \Cref{prp:alg-to-alg-alg} again:
\begin{equation} \label{eqn:formula-for-cQ}
	\cQ^\otimes_{t, k, n, l} := (\bd \square^2 P^\otimes)_{tk,l,n}.
\end{equation}

\begin{prp} \label{prp:Q-otimes-Segal}
	The functor $\cQ^\otimes : \bG^\opp \times \bD^{2 + 1, \opp} \to \sC$ satisfies the Segal condition in the first three variables, namely the ones from the factors $\bG^\opp \times \bD^{2, \opp}$.
\end{prp}

\begin{proof}
	Two of the three constructions applied -- upgrading a commutative algebra object and passing to commutative squares -- preserve the Segal condition by the results in \Cref{sec:increase-cat-number}. The first three variables in $\cQ^\otimes$ are obtained from a symmetric monoidal category object $Q^\otimes$ using exactly those constructions, and so we are done.
\end{proof}

In particular we can think of $\cQ^\otimes$ as a functor $\bD^{2 + 1, \opp} \to \CAlg(\sC)$ and thus it defines a section $\sigma_{\cQ^\otimes}$ of $\what{\GS}(\sC)^{2, 1}$ by the discussion following \Cref{prp:section-of-span-fib}. We then obtain a cocartesian fibration $\what{\GS}(\sC; \cQ^\otimes)^{2, 1} \to \bG^\opp \times \bD^{2 + 1, \opp}$ as in \Cref{dfn:fib-span-with-local-system}. The vertices of this fibration are generalized cartesian spans $F : \bS^{a, c} \times \bsd^b \to \sC^\times_t$ together with a natural transformation $\eta : F \to \sigma_{\cQ^\otimes}(t, a, c, b)$. The choice of labels $k_1 = a, k_2 = c, l_1 = b$ is intentional since we want to think of the $\bsd$ factor as parametrizing the $2$-morphisms, not the $3$-morphisms, of our desired $(\infty,3)$-category. 

\subsubsection{A description of $\cQ^\otimes$}

To clarify all the steps in this construction we will compute $\cQ^\otimes_{t,k,n,l}$ (for small $t,k,n,l$) in terms of the underlying symmetric monoidal category object $Q := Q^\otimes_{1, \bullet}$. When we speak of an object, morphism, etcetera \emph{of} $Q$ we always mean \emph{classified by} $Q$ -- that is, the element, morphism, etcetera classified by a map $c \to Q_{\bullet, \bullet}$ belongs to the symmetric monoidal category object $\sQ^\otimes(c) \in \CAlg(\Cat(\Spaces))$.

First note that
\begin{align*}
	\cQ^\otimes_{t,k,n,0} & = \int_{\bD^{2, \opp}} \bd[0]_{\bullet,\bullet} \pitchfork (\square^2 P^\otimes)_{kt,\bullet, \bullet, n} \simeq (\square^2 P^\otimes)_{kt,0,0,n} \simeq Q^\otimes_{0, n} \simeq \ast, \\
	\cQ^\otimes_{t,k,n,1} & = \int_{\bD^{2, \opp}} \bd[1]_{\bullet,\bullet} \pitchfork (\square^2 P^\otimes)_{kt,\bullet, \bullet, n} \simeq (\square^2 P^\otimes)_{kt,1,0,n} \simeq Q^\otimes_{kt, n}, \\
	\cQ^\otimes_{t,k,n,2} & \simeq \int_{\bD^{2, \opp}} \bd[2]_{\bullet, \bullet} \pitchfork (\square^2 P^\otimes)_{tk, \bullet, \bullet, n} \\
	& \simeq (\square^2 P^\otimes)_{tk, 2, 0, n} \times_{(\square^2 P^\otimes)_{tk,1,0,n}} (\square^2 P^\otimes)_{tk, 1, 1, n} \\
	& \simeq Q^\otimes_{2tk, n} \times_{Q^\otimes_{tk, n}} (\square^2 Q^\otimes)_{tk, 1, n},
\end{align*}
for any $t, k, n$, by the calculations in \Cref{eqn:dx0,eqn:dx1,eqn:dx2}.

{\bfseries The monoidal structure} is controlled by the parameter $t$. Since $\cQ^\otimes_{t, \bullet, \bullet, \bullet}$ satisfies the Segal condition in the variable $t$, its values for $t \geq 1$ are determined by the value at $t = 1$ so we will fix $t = 1$ from now on and only focus on the underlying multi-simplicial object.

{\bfseries Objects} are classified by $\cQ^\otimes_{1, 0,0,0} \simeq \ast$, so there is a contractible choice of objects and their {\bfseries tensor product} is similarly trivial.

{\bfseries $1$-morphisms} are classified by $\cQ^\otimes_{1, 1, 0, 0} \simeq \ast$, so there is a contractible choice of $1$-morphisms and their {\bfseries horizontal composition} is similarly trivial.

{\bfseries $2$-morphisms} are classified by $\cQ^\otimes_{1,1,0,1} \simeq Q^\otimes_{1, 0} = Q_0$. Thus the $2$-morphisms are classified by the objects of $Q$. The {\bfseries horizontal composition of $2$-morphisms} is classified by $\cQ^\otimes_{1, 2, 0, 1} \simeq Q^\otimes_{2,0}$, meaning that it is just given by the tensor product of the objects of $Q$. The {\bfseries vertical composition of $2$-morphisms} is classified by
\begin{equation*}
	\cQ^\otimes_{1, 1, 0, 2} \simeq Q^\otimes_{2,0} \times_{Q^\otimes_{1,0}} Q^\otimes_{1,1},
\end{equation*}
where the pullback is taken along the tensor product map in the first factor and the target map in the second factor (see \Cref{eqn:dx2}). Thus the vertical composition of two $2$-morphisms, represented by two objects of $Q$, is classified by morphisms into their tensor product:
\begin{equation*}
	x, y \text{ in $Q$} \quad \leadsto \quad \text{morphisms } z \to x \otimes y \text{ in $Q$}.
\end{equation*}
Note that this is a possibly multivalued operation as there might be multiple morphisms into $x \otimes y$.

{\bfseries $3$-morphisms} are classified by $\cQ^\otimes_{1,1,1,1} \simeq Q^\otimes_{1,1} = Q_1$, so they are given by morphisms in $Q$. The {\bfseries horizontal composition of $3$-morphisms} is classified by $\cQ^\otimes_{1,2,1,1} \simeq Q^\otimes_{2,1}$, so it is given by the tensor product of morphisms in $Q$. The {\bfseries vertical composition of $3$-morphisms} is classified by
\begin{equation*}
	\cQ^\otimes_{1,1,1,2} \simeq Q^\otimes_{2,1} \times_{Q^\otimes_{1,1}} (\square^2 Q^\otimes_{1,\bullet})_{1,1},
\end{equation*}
i.e. by commutative squares of morphisms in $Q$ one side of which is the tensor product of the morphisms in $Q$ representing the chosen $3$-morphisms:
\begin{equation*}
	x \xlra{f} x', y \xlra{g} y' \text{ in $Q$} \quad \leadsto \quad \text{commutative squares } \begin{tikzcd} z \ar[r] \ar[d] & x \otimes y \ar[d, "f \otimes g"'] \\ z' \ar[r] & x' \otimes y' \end{tikzcd} \text{ in $Q$}.
\end{equation*}
This composition is easily seen to be compatible with the vertical composition of the $2$-morphisms bounding the chosen $3$-morphisms. The {\bfseries transversal composition of $3$-morphisms} is classified by $\cQ^\otimes_{1,1,2,1} \simeq Q^\otimes_{1,2} = Q_2$ and so it is given by the composition of arrows in $Q$, as expected.

The diagram in \Cref{fig:ex-span} is a visual representation of (most of) the data contained in a generalized span $F : \bS^{k,n} \times \bsd^{l} \to \sC_t$ with local systems valued in $\cQ^\otimes$, in the simple case $(t,k,n,l) = (1,1,0,2)$. We believe this is the most instructive case when $l = 2$, as can be deduced from the computations above, since varying $k$ and $n$ only serves to increase the number of such diagrams and varying $t$ allows us to pass to tuples of objects and morphisms without much of a change. The picture for $l \geq 3$ is more complicated and will be explained in detail shortly. For convenience we also note that the non-preferred kinds of $1$- and $2$-morphisms end up being trivial:
\begin{equation*}
	\cQ^\otimes_{1, 0, 1, 0}, \cQ^\otimes_{1, 0, 0, 1}, \cQ^\otimes_{1,1,1,0}, \cQ^\otimes_{1, 0, 1, 1} \simeq \ast.
\end{equation*}

\begin{figure}[t]
	\makebox[\textwidth][c]{
	\begin{tikzpicture}[scale = 1.8]
		\node (B0) at (-2,0) {$B_1$};
		\node (B1) at (2,0) {$B_2$};
		\node (B2) at (0,{2*sqrt(3)}) {$B_3$};
		\node (B01) at (0,0) {$B_{12}$};
		\node (B12) at (1,{sqrt(3)}) {$B_{23}$};
		\node (B02) at (-1,{1*sqrt(3)}) {$B_{13}$};
		\node (B012) at (0,{2*sqrt(3)/3}) {$B_{123}$};
		\draw[->] (B01) -- (B0);
		\draw[->] (B01) -- (B1);
		\draw[->] (B12) -- (B1);
		\draw[->] (B12) -- (B2);
		\draw[->] (B02) -- (B0);
		\draw[->] (B02) -- (B2);
		\draw[->] (B012) -- (B0);
		\draw[->] (B012) -- (B1);
		\draw[->] (B012) -- (B2);
		\draw[->] (B012) -- node[left] {$\pi_{12}$} (B01);
		\draw[->] (B012) -- node[below right] {$\pi_{23}$} (B12);
		\draw[->] (B012) -- node[below] {$\pi_{13}$} (B02);
		\draw[red, thick] (0,-1) node {$x$} circle (0.3);
		\draw[red, thick] (-2,{1.3*sqrt(3)}) node {$z$} circle (0.3);
		\draw[red, thick] (2,{1.3*sqrt(3)}) node {$y$} circle (0.3);
		\draw[red, thick] (4.7,{2*sqrt(3)/3}) circle (1.7);
		\draw[red,thick] (0,-0.6) -- (0,-0.2);
		\draw[red,thick] (-1.65,{1.18*sqrt(3)}) -- (-1.25,{1.05*sqrt(3)});
		\draw[red,thick] (1.65,{1.18*sqrt(3)}) -- (1.25,{1.05*sqrt(3)});
		\draw[red,thick] (0.3,{2*sqrt(3)/3}) edge[bend right = 20] (2.8,{2*sqrt(3)/3});
		\draw[->, red] (3.7, {2*sqrt(3)/3}) node[left] {$\pi_{13}^\ast z$} -- node[above] {$\alpha$} (4.1, {2*sqrt(3)/3}) node[right] {$\pi_{12}^\ast x \otimes \pi_{23}^\ast y$};
		\node[opacity=0.5, blue] (A2) at (3.9,{2*sqrt(3)/3+1}) {$\bullet$};
		\node[opacity=0.5, blue] (A0) at (3.9,{2*sqrt(3)/3-1}) {$\bullet$};
		\node[opacity=0.5, blue] (A1) at (6,{2*sqrt(3)/3}) {$\bullet$};
		\draw[->,blue,opacity=0.5] (A0) -- node[below, red] {$\pi^\ast_{12} x$} (A1);
		\draw[->,blue,opacity=0.5] (A1) -- node[above, red] {$\pi^\ast_{23} y$} (A2);
		\draw[->,blue,opacity=0.5] (A0) edge[bend right = 45] (A2);
		\draw[->,blue,opacity=0.5] (A0) edge[bend left = 45] (A2);
		\begin{scope}[transparency group, opacity=0.5]
			\draw[-{Implies}, double equal sign distance, blue] (3.6, {2*sqrt(3)/3-0.1}) to [out=-30, in=-150] (4.2, {2*sqrt(3)/3-0.1});
		\end{scope}
	\end{tikzpicture}}
	\caption{Example of a generalized span $F \in \what{\GS}(\sC;\cQ^\otimes)^{0,1}_{1,1,0,2}$. Here each $B_{I} := \bigcap_{i \in I} B_i$ is an object of $\sC$, $x,y,z$ are objects of $Q$, and $\alpha$ is a morphism of $Q$. The underlying cartesian span is in black and the local systems are in red, while the blue shadow is a reminder of the path category $\Path(2)$ on which the local system $\cQ^\otimes_{1,1,0,2}$ is based: namely, the local system on $B_{123}$ contains the data of the two objects $\pi_{12}^\ast x$ and $\pi_{23}^\ast y$ pulled back from $B_{12}$ and $B_{23}$, the object $\pi_{13}^\ast z$ pulled back from $B_{13}$, and the morphism $\alpha$.}
	\label{fig:ex-span}
\end{figure}

\subsection{Enforcing the push-pull condition} \label{sec:enforcing-push-pull}

In this section we will construct the symmetric monoidal $(\infty,3)$-category $\sPP(\sC; Q^\otimes)$ by first restricting ourselves to a special subcategory of $\what{\GS}(\sC; \cQ^\otimes)^{2, 1}$, proving that it induces a functor $\bG^\opp \times \bD^{2 + 1, \opp} \to \sCat_1$ which satisfies the Segal condition in all variables, and then passing to Segal spaces. 

From our previous calculations we can see that the problem in obtaining a $3$-category object lies in the vertical composition of $2$-morphisms (and hence also of $3$-morphisms), which is a priori multivalued. In order to fix this issue we will need to make an assumption about the symmetric monoidal category object $Q^\otimes$.
\begin{ass} \label{ass:assumption}
	Consider the representable functor
	\begin{equation*}
		\sQ : \sC^\opp \to \Seg^1(\Spaces) \to \sCat_1, \qquad \sQ(c)_n := \sC(c, Q^\otimes_{1,n}).
	\end{equation*}
	We require that for every $f : c \to d$ in $\sC$ there is an adjunction
	\begin{equation*}
		\begin{tikzcd}
			\sQ(d) \ar[r, bend left, "f^\ast"{name=U}] & \sQ(c). \ar[l, bend left, "f_\ast"{name=D}]
			\ar[from=U, to=D, "\rotatebox{-90}{$\dashv$}" description, phantom]
		\end{tikzcd}
	\end{equation*}
	The left adjoint $f^\ast$ will be called the \emph{pullback} by $f$ and the right adjoint $f_\ast$ will be called the \emph{pushforward} by $f$. Moreover, these adjunctions extend to ensure that
	\begin{equation*}
		\sQ^\dagger : \sC \to \Seg^1(\Spaces) \to \sCat_1, \quad \sQ^\dagger(f) = f_\ast
	\end{equation*}
	is again a functor of $\infty$-categories.
\end{ass}

We need this assumption to be able to talk about the push-pull condition, which we now turn to.

\subsubsection{Push-pull spans, abstractly}

Consider $(F, \eta) \in \what{\GS}(\sC;\cQ^\otimes)^{2,1}_{t,k,n,l}$ and assume $l \geq 2$. We treat $F$ as a functor
\begin{equation*}
	F : \bsd^l \to \Fun(\bS^{k,n}, \sC^\times_t)
\end{equation*}
and $\eta : F \to \sigma_{\cQ^\otimes}(t,k,n,l)$ as a morphism in $\Fun(\bsd^l, \Fun(\bS^{k,n}, \sC^\times_t))$. For any $S \in \bsd^l$ (i.e. $\emptyset \neq S \sub [0,l]$) let $u_S = F(S) \in \sC^\times_t$, and recall that
\begin{equation*}
	u_S \simeq \prod_{i \in S} u_i \in \sC^\times_t
\end{equation*}
since $F$ is cartesian. In particular we set $u = u_{[0,l]}$ and we denote by $\pi_S : u \to u_S$ the canonical projections.

By definition of $\cQ^\otimes$, the functor $\sigma_{\cQ^\otimes}(t,k,n,l) : \bsd^l \to \Fun(\bS^{k,n}, \sC^\times_t)$ takes the value
\begin{equation*}
	\cQ^\otimes_{t, \spadesuit, \clubsuit, l} := \bd((\square^2 P^\otimes)_{t \cdot \spadesuit, \bullet, \bullet, \clubsuit})_{l} : \bS^{k,n} \to \sC^\times_t
\end{equation*}
on the element $S = [0,l]$ of $\bsd^l$, where the variables denoted by $\bullet$ come from the bisimplicial object that the operator $\bd(-)$ is acting on (see \Cref{dfn:nerve-path-with-labels}) and the variables denoted by $\spadesuit$ and $\clubsuit$ are the ones coming from the elements $(x,y)$ of $\bS^{k,n} = \bS^k \times \bS^n$ via the ``face functors'' of \Cref{sec:from-functors-to-sections}. We call $\spadesuit$ and $\clubsuit$ the \emph{height} of $x$ and $y$, respectively, because informally they correspond to the vertical position of $x$ and $y$ in the pyramid-like posets $\bS^k$ and $\bS^n$. All together, we have a local system
\begin{equation*}
	\eta : u \to \cQ^\otimes_{t, \spadesuit, \clubsuit, l}
\end{equation*}
which can be thought of as specifying a map $u(x,y) \to \cQ^\otimes_{t, K, N, l}$ in $\sC^\times_t$ for every element $(x,y) \in \bS^{k,n} = \bS^k \times \bS^n$ where $x$ and $y$ have height $K$ and $N$, respectively.

By \Cref{eqn:xi} we have a projection
\begin{equation*}
	\xi \circ \eta : u \to \cQ^\otimes_{t, \spadesuit, \clubsuit, l} = \bd((\square^2 P^\otimes)_{t \cdot \spadesuit,\bullet, \bullet, \clubsuit})_l \to (\square^2 P^\otimes)_{t \cdot \spadesuit,1,1,\clubsuit} \simeq (\square^2 Q^\otimes_{t \cdot \spadesuit, \bullet})_{1,\clubsuit}.
\end{equation*}
The Segal condition guarantees an equivalence
\begin{equation*}
	(\square^2 Q^\otimes_{t \cdot \spadesuit, \bullet})_{1,\clubsuit} \simeq (\square^2 Q^\otimes_{1, \bullet})_{1,\clubsuit}^{\times t \cdot \spadesuit} = (\square^2 Q_\bullet)_{1,\clubsuit}^{\times t \cdot \spadesuit}
\end{equation*}
so the local system $\xi \circ \eta$ precisely classifies, for every $(x, y) \in \bS^{k,n}$, a tuple (of size enumerated by $t \cdot \spadesuit$) of vertical commutative blocks in $\sQ(u(x,y))$. The latter are towers of commutative squares
\begin{equation*}
	\begin{tikzcd}
		\conj{R}_0 \ar[r, "\conj{\vphi}_0"] \ar[d] & \conj{P}_0 \ar[d] \\
		\conj{R}_1 \ar[r, "\conj{\vphi}_1"] \ar[d] & \conj{P}_1 \ar[d] \\
		\vdots \ar[d] & \vdots \ar[d] \\
		\conj{R}_{\clubsuit-1} \ar[r, "\conj{\vphi}_{\clubsuit-1}"] \ar[d] & \conj{P}_{\clubsuit-1} \ar[d] \\
		\conj{R}_\clubsuit \ar[r, "\conj{\vphi}_\clubsuit"] & \conj{P}_\clubsuit
	\end{tikzcd}
\end{equation*} 
together with their higher composition and coherence data, where the length of the vertical sides is enumerated by $\clubsuit$. Since the components of the tuples are independent of each other we will instead think of $\xi \circ \eta$ as a classifier for one single commutative block valued in the product $\sQ(u)^{\times t \cdot \spadesuit}$.

Each $\conj{\vphi}_i$ comes from the basic projection $\xi$ described in \Cref{eqn:xi} so the sources $\conj{R}_i$ and targets $\conj{P}_i$ are constrained by some associated local systems. The sources $\conj{R}_i$ fit into a diagram\footnote{We are abusing notation a little by using the same symbol to refer to a classifying morphism (such as map $a : c \to Q_1$ in $\sC$) and that which it classifies (such as the morphism $a$ in $\sQ(c)$). We ask forgiveness from the reader for this necessary confusion introduced purely to streamline the notation and avoid getting lost in the indices.}
\begin{equation*}
	\begin{tikzcd}
		u \ar[r, "\conj{\vphi}_i"] \ar[dr, "\conj{R}_i" description] \ar[d, "\pi_{0,l}"'] & Q_1^{\times t \cdot \spadesuit} \ar[d, "\mathrm{source}"] \\
		u_{0} \times u_{l} \ar[r, "\conj{r}_i"'] & Q_0^{\times t \cdot \spadesuit}
	\end{tikzcd}
\end{equation*}
induced by the local system from the inclusion $\{0,l\} \sub [0,l]$ and thus
\begin{equation*}
	\conj{R}_i \simeq \pi_{0,l}^\ast (\conj{r}_i)
\end{equation*}
in $\sQ(u)^{\times t \cdot \spadesuit}$. 

The targets $\conj{P}_i$ fit into a diagram
\begin{equation*}
	\begin{tikzcd}
		u \ar[rrr, "\conj{\vphi}_i"] \ar[drrr, "\conj{P}_i" description] \ar[d, "\lim \pi_{j,j+1}"', "\simeq"] & & & Q_1^{\times t \cdot \spadesuit} \ar[d, "\mathrm{target}"] \\
		\displaystyle \lim_{0 \leq j \leq l-1} (u_{j} \times u_{{j+1}}) \ar[rr, "(\conj{p}_{i,j})_{0 \leq j \leq l-1}"'] & & (Q^\otimes_{l,0})^{\times t \cdot \spadesuit} \ar[r, "\otimes"'] & Q_0^{\times t \cdot \spadesuit}
	\end{tikzcd}
\end{equation*}
induced by the local system from the chain of inclusions $\{j,j+1\} \sub [0,l]$ and thus
\begin{equation*}
	\conj{P}_i \simeq \bigotimes_{j = 0}^{l-1} \pi_{j,{j+1}}^\ast (\conj{p}_{i,j})
\end{equation*}
in $\sQ(u)^{\times t \cdot \spadesuit}$; this equivalence is not necessary in our definition of the push-pull condition, but we will need it later. In conclusion we have tuples of morphisms
\begin{equation*}
	\conj{\vphi}_i : \pi_{0,l}^\ast (\conj{r}_i) \to \conj{P}_i
\end{equation*}
in $\sQ(u)^{\times t \cdot \spadesuit}$ which, thanks to the adjunction $\pi_{0, l}^\ast \dashv (\pi_{0, l})_\ast$, yield new tuples of morphisms
\begin{equation} \label{eqn:push-pull-map}
	\conj{\vphi}_i^\dagger : \conj{r}_i \to (\pi_{0, l})_\ast (\conj{P}_i)
\end{equation}
in $\sQ(u_{0} \times u_{l})^{\times t \cdot \spadesuit}$. We call $\conj{\vphi}_i^\dagger$ the \emph{push-pull (tuples of) maps determined by $(F, \eta)$}.

We can now state the condition precisely.

\begin{dfn} \label{dfn:push-pull}
	Let $Q^\otimes$ satisfy \Cref{ass:assumption}. We say that $(F,\eta) \in \what{\GS}(\sC; \cQ^\otimes)^{2,1}_{t, k, n, l}$ \emph{satisfies the push-pull condition} (or \emph{is push-pull}, for short) if 
	\begin{enumerate}
		\item $l = 0$ or $l = 1$, or
		\item $l \geq 2$ and for every map $\iota : [m] \to [l]$ the induced push-pull tuples of maps $\iota^\ast(\conj{\vphi}_i^\dagger)$ determined by $\iota^\ast (F,\eta)$ are equivalences in $\sQ(u_{\iota(0)} \times u_{\iota(m)})^{\times t \cdot \spadesuit}$.
	\end{enumerate}
	Denote by $\what{\PP}(\sC; Q^\otimes) \sub \what{\GS}(\sC; \cQ^\otimes)^{2, 1}$ the full subcategory spanned by the push-pull spans.
\end{dfn}

Less opaquely, this condition states that a diagram $F : \bS^{k,n} \times \bsd^l \to \sC_t$ with local systems $\eta$ has push-pull equivalences on every slice $F(x,y) : \bsd^l \to \sC_t$ corresponding to an object $(x,y) \in \bS^{k,n}$ and on every face or degeneracy of that slice. It turns out that we only need to consider the maps $\iota : [m] \to [l]$ which are injective, meaning that it is enough to look at the faces of $F(x,y)$. This is because if $\lambda$ is surjective then the added push-pull maps in $\lambda^\ast F(x,y)$ are equivalences already, so if we assume that the condition holds for the injections $\iota$ then it will also hold for all maps $\psi$ in $\bD$ since there is a unique factorization $\psi = \iota \circ \lambda$ of $\psi$ into a surjection followed by an injection.

\subsubsection{Push-pull spans, concretely}

Let us work out the details of the push-pull condition for a few small values of $l$. We will leave the parametrized framework for a moment and fix an object $(x, y) \in \bS^{k,n}$ such that the height of $x$ is $\spadesuit = k$ and the height of $y$ is $\clubsuit = n$.

The first interesting case is $l = 2$: the local system
\begin{equation*}
	u_0 \times u_1 \times u_2 \to \cQ^\otimes_{t,k,n,2} \simeq Q^\otimes_{2tk, n} \times_{Q^\otimes_{tk, n}} (\square^2 (Q^\otimes_{tk, \bullet}))_{1, n}
\end{equation*}
classifies a commutative block
\begin{equation*}
	\begin{tikzcd}
		\pi_{0,2}^\ast (\conj{r}_0) \ar[r, "\conj{\vphi}_0"] \ar[d] & \pi_{0,1}^\ast (\conj{p}_{0,0}) \otimes \pi_{1,2}^\ast (\conj{p}_{0,1}) \ar[d] \\
		\pi_{0,2}^\ast (\conj{r}_1) \ar[r, "\conj{\vphi}_1"] \ar[d] & \pi_{0,1}^\ast (\conj{p}_{1,0}) \otimes \pi_{1,2}^\ast (\conj{p}_{1,1}) \ar[d] \\
		\vdots \ar[d] & \vdots \ar[d] \\
		\pi_{0,2}^\ast (\conj{r}_{n-1}) \ar[r, "\conj{\vphi}_{n-1}"] \ar[d] & \pi_{0,1}^\ast (\conj{p}_{n-1,0}) \otimes \pi_{1,2}^\ast (\conj{p}_{n-1,1}) \ar[d] \\
		\pi_{0,2}^\ast (\conj{r}_n) \ar[r, "\conj{\vphi}_n"] & \pi_{0,1}^\ast (\conj{p}_{n,0}) \otimes \pi_{1,2}^\ast (\conj{p}_{n,1})
	\end{tikzcd}
\end{equation*}
in $\sQ(u_0 \times u_1 \times u_2)^{\times tk}$. If each push-pull map
\begin{equation*}
	\conj{\vphi}_i^\dagger : \conj{r}_i \to (\pi_{0,2})_\ast (\pi_{0,1}^\ast (\conj{p}_{i,0}) \otimes \pi_{1,2}^\ast (\conj{p}_{i,1}))
\end{equation*}
is an equivalence then $\conj{r}_i$ is a specified witness for the vertical composition of $\conj{p}_{i,0}$ and $\conj{p}_{i,1}$ when these are treated as $2$-morphisms in our desired $(\infty, 3)$-category. We encourage the reader to compare this description with \Cref{fig:ex-span}, where $tk = 1$ and $n = 0$.

In fact we have almost produced a proof of this statement:
\begin{lmm} \label{lmm:pp-segal-2}
	The Segal projection
	\begin{equation*}
		\rho : \what{\PP}(\sC;Q^\otimes)_{t,k,n,2} \to \what{\PP}(\sC;Q^\otimes)_{t,k,n,1} \times_{\what{\PP}(\sC;Q^\otimes)_{t,k,n,0}} \what{\PP}(\sC;Q^\otimes)_{t,k,n,1}
	\end{equation*}
	is an equivalence.
\end{lmm}

\begin{proof}
	This proof will use the ``parametrized'' formalism developed above, in the sense that we will work with functors $\bsd^l \to \Fun(\bS^{k,n}, \sC^\times_t)$ and use $\spadesuit$ and $\clubsuit$ to denote variables coming from the elements of $\bS^{k,n}$. When we make a choice (for example when we choose an inverse of an equivalence), we are implicitly making one simultaneous choice for all elements $\bS^{k,n}$ -- this is consistent since there are only finitely many such elements.
	
	We will show that each fiber of $\rho$ is contractible by showing that any two vertices of it are connected by an essentially unique path. Let $(F,\eta^F), (G, \eta^G) \in \what{\PP}(\sC;Q^\otimes)_{t,k,n,2}$ live over the same fiber of $\rho$: then they share two $1$-dimensional faces $(E_0,\eta^{E_0}), (E_1,\eta^{E_1})$ in $\what{\PP}(\sC;Q^\otimes)_{t,k,n,1}$. It follows that
	\begin{enumerate}
		\item The underlying cartesian spans $F$ and $G$ are equivalent. The two faces $E_0, E_1$ determine three $0$-dimensional faces $D_0, D_1, D_2 \in \what{\PP}(\sC;Q^\otimes)_{t,k,n,0}$ -- just three cartesian spans without any local systems -- which agree for both $F$ and $G$. There is a commutative diagram
		\begin{equation*}
			\begin{tikzcd}
				\displaystyle \what{\PP}(\sC;Q^\otimes)_{t,k,n,2} \ar[d] \ar[ddr, bend left, "\text{underlying span}"] \ar[dd, bend right = 80, "\text{underlying $0$-faces}"'] & \\
				\displaystyle \prod_{j = 0}^2 \what{\PP}(\sC;Q^\otimes)_{t,k,n,0} \ar[d] & \\
				\displaystyle \prod_{j = 0}^2 \what{\GS}(\sC)^{2,1, \rmcart}_{t,k,n,0} & \what{\GS}(\sC)^{2,1, \rmcart}_{t,k,n,2} \ar[l, "\simeq"']
			\end{tikzcd}
		\end{equation*}
		where the bottom map is an equivalence by the cartesian condition and the vertical maps take both $(F,\eta^F)$ and $(G,\eta^G)$ to the triple $(D_0,D_1,D_2)$. The cartesian span in $\what{\GS}(\sC)^{2,1, \rmcart}_{t,k,n,2}$ induced by $(D_0,D_1,D_2)$ is the underlying span of both $F$ and $G$, and thus they must agree. 
		\item The local systems $\eta^F$ and $\eta^G$ have equivalent $0$-dimensional and $1$-dimensional faces. The first assertion is vacuous since $\cQ^\otimes_{t,k,n,0} \simeq \ast$. By assumption $\eta^{E_0}$ and $\eta^{E_1}$ are two out of the three $1$-dimensional faces of both $\eta^F$ and $\eta^G$, and thus must agree by assumption. In particular, the target
		\begin{equation*}
			\conj{P}_i \simeq \pi^\ast_{0,1}(\conj{p}_{i,0}) \otimes \pi^\ast_{1,2}(\conj{p}_{i,1}) \in \sQ(u_0 \times u_2 \times u_1)^{\times t \cdot \spadesuit}
		\end{equation*}
		of the push-pull maps is the same for both $\eta^F$ and $\eta^G$, since it is obtained from $\eta^{E_0}$ and $\eta^{E_1}$. The last $1$-dimensional faces are, say, $\theta^F$ and $\theta^G$, and as local systems these determine two chains of morphisms
		\begin{equation*}
			\begin{tikzcd}
				\conj{r^F}_0 \ar[d] & \conj{r^G}_0 \ar[d] \\
				\vdots \ar[d] & \vdots \ar[d] \\
				\conj{r^F}_n & \conj{r^G}_n 
			\end{tikzcd}
		\end{equation*}
		in $\sQ(u_0 \times u_2)^{\times t \cdot \spadesuit}$. But since both $\eta^F$ and $\eta^G$ are push-pull we have equivalences
		\begin{equation*}
			\conj{r^F}_i \xlongrightarrow[\simeq]{\conj{\vphi^F}_i^\dagger} (\pi_{0,2})_\ast (\conj{P}_i) \xlongleftarrow[\simeq]{\conj{\vphi^G}_{i}^\dagger} \conj{r^G}_i
		\end{equation*}
		in $\sQ(u_0 \times u_2)^{\times t \cdot \spadesuit}$ commuting with the vertical maps in the chains, meaning that we can choose essentially unique horizontal equivalences
		\begin{equation*}
			\begin{tikzcd}
				\conj{r^F}_0 \ar[d] \ar[r, "\psi_0", "\simeq"'] & \conj{r^G}_0 \ar[d] \\
				\conj{r^F}_1 \ar[d] \ar[r, "\psi_1", "\simeq"'] & \conj{r^G}_1 \ar[d] \\
				\vdots \ar[d] & \vdots \ar[d] \\
				\conj{r^F}_{\clubsuit-1} \ar[d] \ar[r, "\psi_{\clubsuit-1}", "\simeq"'] & \conj{r^G}_{\clubsuit-1} \ar[d] \\
				\conj{r^F}_\clubsuit \ar[r, "\psi_{\clubsuit}", "\simeq"'] & \conj{r^G}_\clubsuit 
			\end{tikzcd}
		\end{equation*}
		in $\sQ(u_0 \times u_2)^{\times t \cdot \spadesuit}$ commuting with the push-pull maps and making each square commute. Hence $\theta^F \simeq \theta^G$.
		\item The local systems $\eta^F$ and $\eta^G$ are equivalent on their only $2$-dimensional face. These local systems are determined by the three $1$-dimensional faces $\eta^{E_0}, \eta^{E_1}, \theta^F \simeq \theta^G$ shared by $\eta^F$ and $\eta^G$ and the two commutative blocks
		\begin{equation*}
			\begin{tikzcd}[row sep = small]
				\pi_{0,2}^\ast(\conj{r^F}_0) \ar[r, "\conj{\vphi^F}_0"] \ar[d] & \conj{P}_0 \ar[d] & \pi_{0,2}^\ast (\conj{r^G}_0) \ar[l, "\conj{\vphi^G}_{0}"'] \ar[d] \\
				\pi_{0,2}^\ast(\conj{r^F}_1) \ar[r, "\conj{\vphi^F}_1"] \ar[d] & \conj{P}_1 \ar[d] & \pi_{0,2}^\ast (\conj{r^G}_1) \ar[l, "\conj{\vphi^G}_{1}"'] \ar[d] \\
				\vdots \ar[d] & \vdots \ar[d] & \vdots \ar[d] \\
				\pi_{0,2}^\ast(\conj{r^F}_{\clubsuit-1}) \ar[r, "\conj{\vphi^F}_{\clubsuit-1}"] \ar[d] & \conj{P}_{\clubsuit-1} \ar[d] & \pi_{0,2}^\ast (\conj{r^G}_{\clubsuit-1}) \ar[l, "\conj{\vphi^G}_{\clubsuit-1}"'] \ar[d] \\
				\pi_{0,2}^\ast(\conj{r^F}_\clubsuit) \ar[r, "\conj{\vphi^F}_\clubsuit"] & \conj{P}_\clubsuit & \pi_{0,2}^\ast (\conj{r^G}_\clubsuit) \ar[l, "\conj{\vphi^G}_{\clubsuit}"']
			\end{tikzcd}
		\end{equation*}
		which they classify. But by the previous point this large commutative block can be extended to a commutative triangular prism whose front side consists entirely of equivalences:
		\begin{equation*}
			\begin{tikzcd}[column sep = large, row sep = small]
				\pi_{0,2}^\ast(\conj{r^F}_0) \ar[rr, "\pi_{0,2}^\ast(\psi_0)", "\simeq"', near end] \ar[rd, "\conj{\vphi^F}_0"'] \ar[dd] & & \pi_{0,2}^\ast (\conj{r^G}_0) \ar[ld, "\conj{\vphi^G}_{0}"] \ar[dd] \\
				& \conj{P}_0 \ar[dd] & \\
				\pi_{0,2}^\ast(\conj{r^F}_1) \ar[rr, "\pi_{0,2}^\ast(\psi_1)", "\simeq"', near end, crossing over] \ar[rd, "\conj{\vphi^F}_1"'] \ar[dd] & & \pi_{0,2}^\ast (\conj{r^G}_1) \ar[ld, "\conj{\vphi^G}_{1}"] \ar[dd] \\
				& \conj{P}_1 \ar[dd] & \\
				\vdots \ar[dd] & & \vdots \ar[dd] \\
				& \vdots \ar[dd] & \\
				\pi_{0,2}^\ast(\conj{r^F}_{\clubsuit-1}) \ar[rr, "\pi_{0,2}^\ast(\psi_{\clubsuit-1})", "\simeq"', near end, crossing over] \ar[rd, "\conj{\vphi^F}_{\clubsuit-1}"'] \ar[dd] & & \pi_{0,2}^\ast (\conj{r^G}_{\clubsuit-1}) \ar[ld, "\conj{\vphi^G}_{\clubsuit-1}"] \ar[dd] \\
				& \conj{P}_{\clubsuit-1} \ar[dd] & \\
				\pi_{0,2}^\ast(\conj{r^F}_\clubsuit) \ar[rr, "\pi_{0,2}^\ast(\psi_\clubsuit)", "\simeq"', near end, crossing over] \ar[rd, "\conj{\vphi^F}_\clubsuit"'] & & \pi_{0,2}^\ast (\conj{r^G}_\clubsuit) \ar[ld, "\conj{\vphi^G}_{\clubsuit}"] \\
				& \conj{P}_\clubsuit &
			\end{tikzcd}
		\end{equation*}
		Since the triangular slices and the square faces are all commutative, and since the horizontal maps are equivalences, the two back sides are equivalent commutative blocks. Hence $\eta^F \simeq \eta^G$, as needed.
	\end{enumerate}
	All of the choices made in establishing the equivalence between $\eta^F$ and $\eta^G$ are essentially unique, since the $\psi_i$'s were chosen from a contractible space of inverses all at once for every element of $\bS^{k,n}$, and compatible, since we made them all at once for all possible values of $\spadesuit$ and $\clubsuit$.
\end{proof}

Going back for one last example (where again we fix $(x,y) \in \bS^{k,n}$ with $\spadesuit = k$ and $\clubsuit = n$), when $l = 3$ we have six local systems
\begin{align*}
	& u_0 \times u_1 \xlra{\conj{p_{\bullet, 0}}} \cQ^\otimes_{t,k,n,1}, & & u_1 \times u_2 \xlra{\conj{p_{\bullet, 1}}} \cQ^\otimes_{t,k,n,1}, & & u_2 \times u_3 \xlra{\conj{p_{\bullet, 2}}} \cQ^\otimes_{t,k,n,1} \\
	& u_0 \times u_2 \xlra{\conj{q_\bullet}} \cQ^\otimes_{t,k,n,1}, & & u_1 \times u_3 \xlra{\conj{r_{\bullet}}} \cQ^\otimes_{t,k,n,1}, & & u_0 \times u_3 \xlra{\conj{s_{\bullet}}} \cQ^\otimes_{t,k,n,1}
\end{align*}
on the $1$-dimensional faces of $F$ which classify six chains
\begin{equation*}
	\begin{tikzcd}[row sep = small]
		\conj{p}_{0,0} \ar[d] & \conj{p}_{0,1} \ar[d] & \conj{p}_{0,2} \ar[d] & \conj{q}_{0} \ar[d] & \conj{r}_{0} \ar[d] & \conj{s}_{0} \ar[d] \\
		\conj{p}_{1,0} \ar[d] & \conj{p}_{1,1} \ar[d] & \conj{p}_{1,2} \ar[d] & \conj{q}_{1} \ar[d] & \conj{r}_{1} \ar[d] & \conj{s}_{1} \ar[d] \\
		\vdots \ar[d] & \vdots \ar[d] & \vdots \ar[d] & \vdots \ar[d] & \vdots \ar[d] & \vdots \ar[d] \\
		\conj{p}_{n-1,0} \ar[d] & \conj{p}_{n-1,1} \ar[d] & \conj{p}_{n-1,2} \ar[d] & \conj{q}_{n-1} \ar[d] & \conj{r}_{n-1} \ar[d] & \conj{s}_{n-1} \ar[d]\\
		\conj{p}_{n,0} & \conj{p}_{n,1} & \conj{p}_{n,2} & \conj{q}_n & \conj{r}_n & \conj{s}_n
	\end{tikzcd}
\end{equation*}
of $tk$-tuples of $1$-morphisms their respective $\infty$-categories $\sQ(-)$. The four local systems
\begin{align*}
	& u_0 \times u_1 \times u_2 \to \cQ^\otimes_{t,k,n,2}, & & u_1 \times u_2 \times u_3 \to \cQ^\otimes_{t,k,n,2} \\
	& u_0 \times u_1 \times u_3 \to \cQ^\otimes_{t,k,n,2}, & & u_0 \times u_2 \times u_3 \to \cQ^\otimes_{t,k,n,2}
\end{align*}
on the $2$-dimensional faces of $F$ determine equivalences
\begin{align*}
	& \conj{q}_i \simeq (\pi_{0,2})_\ast (\pi_{0,1}^\ast (\conj{p}_{i,0}) \otimes \pi_{1,2}^\ast (\conj{p}_{i,1})), & & \conj{r}_i \simeq (\pi_{1,3})_\ast (\pi_{1,2}^\ast (\conj{p}_{i,1}) \otimes \pi_{2,3}^\ast (\conj{p}_{i,2})), \\
	& \conj{s}_i \simeq (\pi_{0,3})_\ast (\pi_{0,1}^\ast (\conj{p}_{i,0}) \otimes \pi_{1,3}^\ast (\conj{r}_{i})), & & \conj{s}_i \simeq (\pi_{0,3})_\ast (\pi_{0,2}^\ast (\conj{q}_{i}) \otimes \pi_{2,3}^\ast (\conj{p}_{i,2}))
\end{align*}
via the push-pull maps. Finally, the local system
\begin{equation*}
	u_0 \times u_1 \times u_2 \times u_3 \to \cQ^\otimes_{t,k,n,3}
\end{equation*}
on the only $3$-dimensional face of $F$ ensures that these equivalences fit into the commutative diagrams
\begin{equation*}
	\begin{tikzcd}
		\conj{s}_i \ar[dr, "\simeq"] \ar[r, "\simeq"] \ar[d, "\simeq"'] & (\pi_{0,3})_\ast (\pi_{0,1}^\ast (\conj{p}_{i,0}) \otimes \pi_{1,3}^\ast (\conj{r}_{i})) \ar[d, "\simeq"] \\
		(\pi_{0,3})_\ast (\pi_{0,2}^\ast (\conj{q}_{i}) \otimes \pi_{2,3}^\ast (\conj{p}_{i,2})) \ar[r, "\simeq"'] & (\pi_{0,3})_\ast (\pi_{0,1}^\ast (\conj{p}_{i,0}) \otimes \pi_{1,2}^\ast(\conj{p}_{i,1}) \otimes \pi_{2,3}^\ast (\conj{p}_{i,2}))
	\end{tikzcd}
\end{equation*}
of $tk$-tuples in $\sQ(u_0 \times u_3)$. Thus $\conj{s}_i$ is a specified witness for the vertical composition of $\conj{p}_{i,0}$, $\conj{p}_{i,1}$, and $\conj{p}_{i,2}$ as $2$-morphisms, and also contains the data of all possible choices of bracketing.

\subsubsection{The $(\infty,3)$-category of push-pull spans}

Our next order of business is to show that $\what{\PP}(\sC; Q^\otimes)$ determines a symmetric monoidal $(\infty, 3)$-category. First we show that it determines a functor from $\bG^\opp \times \bD^{2 + 1, \opp}$.

\begin{prp}
	The projection $\what{\PP}(\sC; Q^\otimes) \sub \what{\GS}(\sC; \cQ^\otimes)^{2, 1} \to \bG^\opp \times \bD^{2 + 1, \opp}$ is a cocartesian fibration.
\end{prp}

\begin{proof}
	Since the inclusion $\what{\PP}(\sC; Q^\otimes) \sub \what{\GS}(\sC; \cQ^\otimes)^{2, 1}$ is full it will be enough to prove that the action of $(\alpha, \beta, \gamma, \delta)$ in $\bG^\opp \times \bD^{2 + 1}$ on $\what{\PP}(\sC; Q^\otimes)$ is closed, i.e. that for any $(F,\eta) \in \what{\PP}(\sC; Q^\otimes)_{t,k,n,l}$ the new span $(\alpha_\ast, \beta^\ast, \gamma^\ast, \delta^\ast) (F,\eta)$ is in $\what{\PP}(\sC; Q^\otimes)$. Acting by $\alpha$ and $\beta$ only changes the number of components in each push-pull tuple by either adding identity morphisms (which are trivially push-pull), composing some morphisms (which preserves equivalences), or forgetting some morphisms; these operations evidently preserve the push-pull condition. Acting by $\gamma$ changes the height of the commutative blocks by possibly adding identities, composing, or forgetting morphisms, but this does not affect the push-pull maps themselves. Acting by $\delta$ preserves the push-pull condition by induction. Suppose $\delta$ goes from $[m]$ to $[l]$, assume $(F,\eta)$ is push-pull, and let $\veps : [h] \to [m]$ be another map in $\bD$. Then the push-pull maps in $(\delta \circ \veps)^\ast (F,\eta)$ are equivalences by \Cref{dfn:push-pull}, and since
	\begin{equation*}
		\veps^\ast (\delta^\ast (F,\eta)) \simeq (\delta \circ \veps)^\ast (F,\eta)
	\end{equation*}
	we see that the push-pull maps in $\veps^\ast (\delta^\ast (F,\eta))$ are also equivalences. But $\veps$ was arbitrary, so it follows $\delta^\ast (F,\eta)$ is push-pull, as required.
\end{proof}

\begin{dfn} \label{dfn:PP}
	Denote by $\PP(\sC; Q^\otimes) : \bG^\opp \times \bD^{2 + 1, \opp} \to \sCat_1$ a functor obtained by straightening the cocartesian fibration $\what{\PP}(\sC; Q^\otimes)$
\end{dfn}

\begin{prp} \label{prp:push-pull-spans-are-segal-1}
	The functor $\PP(\sC; Q^\otimes)$ satisfies the Segal condition in the first three variables.
\end{prp}

\begin{proof}
	We simplify the notation by letting $\sX = \what{\PP}(\sC; Q^\otimes)$ and $\sY = \what{\GS}(\sC; \cQ^\otimes)^{2, 1}$. The inclusion $X \sub Y$ is compatible with the Segal maps: for every quadruple of indices $(t, k, n, l)$ we have pullback diagrams
	\begin{equation*}
		\begin{tikzcd}
			\sX_{t,k,n,l} \ar[r] \ar[d] & \sY_{t,k,n,l} \ar[d] \\
			\sX_{b} \times_{\sX_{c}} \dotsb \times_{\sX_{c}} \sX_{b} \ar[r] & \sY_{b} \times_{\sY_{c}} \dotsb \times_{\sY_{c}} \sY_{b}
		\end{tikzcd}
	\end{equation*}
	one for each row in the matrix with columns $(b, c)$ given by
	\begin{equation*}
		b = 
		\begin{pmatrix}
			(1, k, n, l) \\
			(t, 1, n, l) \\
			(t, k, 1, l) \\
			(t, k, n, 1)
		\end{pmatrix}, \quad
		c = 
		\begin{pmatrix}
			(0, k, n, l) \\
			(t, 0, n, l) \\
			(t, k, 0, l) \\
			(t, k, n, 0)
		\end{pmatrix}.
	\end{equation*}
	By \Cref{prp:Q-otimes-Segal} the functor $\cQ^\otimes : \bG^\opp \times \bD^{2 + 1, \opp} \to \sC$ satisfies the Segal condition in the first three variables $(\brak{t}, [k], [n]) \in \bG^\opp \times \bD^{2, \opp}$ and thus by \Cref{prp:P-Segal-implies-GS-Segal} (applied to the induced functor $\bD^{2 + 1, \opp} \to \CAlg(\sC)$) the right vertical maps in the first three diagrams are equivalences. Since the diagrams are pullbacks this implies that the left vertical maps are equivalences and so $\PP(\sC, Q^\otimes)$ satisfies the Segal condition in the first three variables.
\end{proof}

\begin{prp} \label{prp:push-pull-spans-are-segal-2}
	The functor $\PP(\sC; Q^\otimes)$ satisfies the Segal condition in the last variable.
\end{prp}

\begin{proof}
	We again let $\sX = \what{\PP}(\sC; Q^\otimes)$ and we fix a tuple $d = (t,k,n)$ of indices. We have Segal face maps
	\begin{equation*}
		\rho_i : \sX_{d,l} \to \sX_{d,1},
	\end{equation*}
	one for each inert inclusion $\{i, i+1\} \sub [0,l]$, which assemble to
	\begin{equation*}
		\rho : \sX_{d,l} \to \sX_{d,1} \times_{\sX_{d,0}} \dotsb \times_{\sX_{d,0}} \sX_{d,1}
	\end{equation*}
	and we want to show that $\rho$ is an equivalence. We will prove the following, which readily implies that $\rho$ is an equivalence since each of its fibers is non-empty: if $(F,\eta^F)$ and $(G,\eta^G)$ are in the same fiber of $\rho$ then there is an essentially unique equivalence $(F,\eta^F) \simeq (G,\eta^G)$. The cases $l = 0$ and $l = 1$ are trivial and the case $l = 2$ is proved in \Cref{lmm:pp-segal-2}, so by induction we can assume that it holds up to $l-1$. By assumption $F$ and $G$ agree on their spine, namely the collection of $1$-dimensional faces indexed by the inert inclusions $\{i, i+1\} \sub [0,l]$, so by the inductive hypothesis they must agree on all of their $m$-dimensional faces for $m \leq l-1$. This means that $F \simeq G$ and thus $(F,\eta^F)$ and $(G,\eta^G)$ have a common underlying cartesian span with (parametrized) vertices $u_0, \dotsc, u_l$ and common equivalent local systems
	\begin{equation*}
		\eta_S^F \simeq \eta^G_S : u_S \simeq \prod_{i \in S} u_i \to \cQ^\otimes_{t,\spadesuit,\clubsuit,\vnorm{S}-1}
	\end{equation*} 
	for all proper subsets $S \subset [0,l]$. Since $l \geq 3$ we can now apply the setup of \Cref{lmm:inductive-structure-dxl}: we have two natural transformation $\eta^F$ and $\eta^G$ of functors $A_l, B_l : \wp(l)^\opp \to \Fun(\bS^{k,n}, \sC_t^\times)$ which are equivalent upon restricting to $\wp(l)^\opp \setminus \{[0,l]\}$. The lemma implies that the common restriction has an essentially unique extension to $\wp(l)^\opp$, which readily implies $\eta^F \simeq \eta^G$ via an essentially unique equivalence. This concludes the proof.  
\end{proof}

\begin{thm} \label{thm:push-pull-3-cat}
	Let $\sC$ be an $\infty$-category with finite limits and let $Q : \bG^\opp \times \bD^\opp \to \sC$ be a symmetric monoidal category object in $\sC$ satisfying \Cref{ass:assumption}. The functor $\PP(\sC;Q^\otimes) : \bG^\opp \times \bD^{3, \opp} \to \sCat_1$ determines a symmetric monoidal $(\infty,3)$-category.
\end{thm}

\begin{proof}
	By \Cref{prp:push-pull-spans-are-segal-1} and \Cref{prp:push-pull-spans-are-segal-2} $\PP(\sC;Q^\otimes)$ satisfies the Segal condition, so it defines an object of $\Cat^3(\CAlg(\sCat_1)) \simeq \CAlg(\Cat^3(\sCat_1))$. Passing first to $\CAlg(\Cat^3(\Spaces))$ via the functor $\sCat_1 \to \Spaces$ which takes the maximal subspace of an $\infty$-category (i.e. keeping all and only the invertible morphisms) and then passing to complete Segal spaces via \Cref{prp:from-cat-to-seg} gives an object of $\CAlg(\sCat_3)$.
\end{proof}

\begin{dfn} \label{dfn:push-pull-3-cat}
	We denote by $\sPP(\sC;Q^\otimes) \in \CAlg(\sCat_3)$ a symmetric monoidal $(\infty,3)$-category obtained in \Cref{thm:push-pull-3-cat} and we call it the \emph{$3$-category of push-pull spans in $\sC$ valued in $Q^\otimes$}.
\end{dfn}

By construction, the $(\infty, 3)$-category $\sPP(\sC;Q^\otimes)$ has the property that the $3$-morphisms between two $2$-morphisms $P, Q \in \sQ(Z_1 \cap Z_2)$ are given by a pair $(W, \alpha)$, where $W \in \sC_{/Z_1 \cap Z_2}$ comes with a map $i : W \to Z_1 \cap Z_2$ and $\alpha : i^\ast P \to i^\ast Q$ is a map in $\sQ(W)$. It is straightforward to pass to the subcategory where all such $W$ are equivalent to $Z_1 \cap Z_2$ (i.e. where $i$ is an equivalence), thus recovering exactly the desired property (c'). We do this explicitly in \Cref{sec:rw-3-cat}, but doing so now recovers exactly \Cref{thm:thmA}.

One note must be made about completeness here. Recall the functors $R^3$ and $L^3$ from \Cref{prp:from-cat-to-seg}. If $\sC$ is an $\infty$-topos and $Q^\otimes : \bG^\opp \to \Fun(\bD^{\opp}, \sC)$ is a symmetric monoidal \emph{complete} Segal object in $\sC$ (see \cite[Definition 7.4]{Haugseng2018}) then using similar arguments to those in \cite[Section 9]{Haugseng2018} one can prove that $R^3 (\PP(\sC;Q^\otimes)) \in \CAlg(\Seg^3(\Spaces))$ is already complete, in which case applying $L^3$ does noting. In any case, though, completing an $n$-fold Segal space does not affect the dualizability of its $k$-morphisms, so for the sake of computations it is good enough to work with the $3$-fold Segal space underlying $\PP(\sC;Q^\otimes)$. This is (implicitly) what we will do in \Cref{sec:rw-3-cat}.

Since the data of $Q^\otimes$ is equivalent to that of the functor $\sQ^\otimes : \sC^\opp \to \CAlg(\Cat(\Spaces))$ which it represents, we may sometimes say that the local systems are valued in or determined by $\sQ^\otimes(-)$. This terminology is more convenient for the section to come.

\subsubsection{Relative push-pull spans}

The construction of $\sPP(\sC;Q^\otimes)$ requires that $Q^\otimes$ be a (parametrized) object of $\sC$ satisfying some nice properties associated to the functor which it represents. Say $\sD^\odot$ is a symmetric monoidal $\infty$-category (not necessarily cartesian monoidal) that admits a symmetric monoidal functor $U : \sD^\odot \to \sC^\times$ to $\sC$ with its cartesian monoidal structure.
\begin{dfn}
	We let
	\begin{equation*}
		\what{\GS}(\sD^\odot)^{m,n,\rmcart} \sub \what{\GS}(\sD^\odot)^{m,n}
	\end{equation*}
	denote the full subcategory spanned by the \emph{generalized $\odot$-cartesian spans in $\sD$}, namely those spans 
	\begin{equation*}
		F : \bS^{\conj{k}} \times \bsd^{\conj{l}} \to \sD^\odot_t
	\end{equation*}
	such that the composite 
	\begin{equation*}
		U_t \circ F : \bS^{\conj{k}} \times \bsd^{\conj{l}} \to \sC^\times_t
	\end{equation*}
	is a cartesian span in $\sC$. That is, $\what{\GS}(\sD^\odot)^{m,n,\rmcart}$ fits into the pullback diagram
	\begin{equation*}
		\begin{tikzcd}
			\what{\GS}(\sD^\odot)^{m,n,\rmcart} \ar[r] \ar[d, "U_\ast"'] & \what{\GS}(\sD^\odot)^{m,n} \ar[d, "U_\ast"] \ar[d] \\
			\what{\GS}(\sC)^{m,n,\rmcart} \ar[r, hook] & \what{\GS}(\sC)^{m,n}
		\end{tikzcd}
	\end{equation*}
	of cocartesian fibrations over $\bG^\opp \times \bD^{m + n, \opp}$.
\end{dfn}

It is possible that $Q^\otimes$ does not lift to $\sD$ and so, strictly speaking, we cannot talk about $Q^\otimes$-valued local systems for spans in $\sD$. However, we can do the next best thing: first pass to $\sC$ via $U_\ast$ and then ask for local systems on the image, so as to get local systems on spans in $\sD$ determined by $\sQ^\otimes \circ U^\opp$. More precisely, consider the following diagram of cocartesian fibrations, where the vertical map on the right picks out the underlying generalized cartesian span of a push-pull span in $\sC$:
\begin{equation} \label{eqn:relative-push-pull}
	\begin{tikzcd}
		& \what{\PP}(\sC;Q^\otimes) \ar[d] \ar[ddr] & \\
		\what{\GS}(\sD^\odot)^{2,1, \rmcart} \ar[r, "U_\ast"] \ar[drr] & \what{\GS}(\sC)^{2,1, \rmcart} \ar[dr] & \\
		& & \bG^\opp \times \bD^{2+1, \opp}
	\end{tikzcd}
\end{equation}

\begin{dfn} \label{dfn:relative-push-pull-fibration}
	Let $\sD^\odot$ be a symmetric monoidal $\infty$-category and let $U : \sD^\odot \to \sC^\times$ be a symmetric monoidal functor. Denote by $\what{\PP}(U; Q^\otimes) \to \bG^\opp \times \bD^{2+1, \opp}$ a cocartesian fibration obtained by taking a pullback of the diagram in \Cref{eqn:relative-push-pull} and let $\PP(U; Q^\otimes) : \bG^\opp \times \bD^{2 + 1, \opp} \to \sCat_1$ be a straightening of it.
\end{dfn}
A vertex of $\what{\PP}(U; Q^\otimes)$ is a generalized $\odot$-cartesian span $F$ in $\sD$ together with a push-pull local system on $U_\ast(F)$ valued in $Q^\otimes$ -- equivalently, a push-pull local system on $F$ valued in the $\infty$-categories determined by $\sQ^\otimes \circ U^\opp(-)$.

\begin{prp}
	The functor $\PP(U; Q^\otimes)$ satisfies the Segal condition in all variables.
\end{prp}

\begin{proof}
	Let $\sX = \what{\PP}(U; Q^\otimes)$ and $\sY = \what{\PP}(\sC; Q^\otimes)$. Then the Segal maps for $\sX$ are obtained from those for $\sY$ via $U_\ast$: for every quadruple of indices $(t,k,n,l)$ we have pullback diagrams
	\begin{equation*}
		\begin{tikzcd}
			\sX_{t,k,n,l} \ar[r, "U_\ast"] \ar[d] & \sY_{t,k,n,l} \ar[d] \\
			\sX_{b} \times_{\sX_{c}} \dotsb \times_{\sX_{c}} \sX_{b} \ar[r] & \sY_{b} \times_{\sY_{c}} \dotsb \times_{\sY_{c}} \sY_{b}
		\end{tikzcd}
	\end{equation*}
	one for each row in the matrix with columns $(b, c)$ given by
	\begin{equation*}
		b = 
		\begin{pmatrix}
			(1, k, n, l) \\
			(t, 1, n, l) \\
			(t, k, 1, l) \\
			(t, k, n, 1)
		\end{pmatrix}, \quad
		c = 
		\begin{pmatrix}
			(0, k, n, l) \\
			(t, 0, n, l) \\
			(t, k, 0, l) \\
			(t, k, n, 0)
		\end{pmatrix}.
	\end{equation*}
	That these squares are pullback squares follows immediately from the fact that $U_\ast$ preserves cartesian spans and from the definition of $\sX$. By \Cref{prp:push-pull-spans-are-segal-1} and \Cref{prp:push-pull-spans-are-segal-2} the right vertical map in all four cases is an equivalence, so the left vertical map is an equivalence as well and we are done.
\end{proof}

\begin{dfn} \label{dfn:relative-push-pull-spans}
	The symmetric monoidal $(\infty,3)$-category induced by $\PP(U; Q^\otimes)$ is denoted $\sPP(U; Q^\otimes)$ and contains the \emph{$U$-relative push-pull spans in $\sD$ valued in $Q^\otimes$}, or just \emph{relative push-pull spans} for short.
\end{dfn}

%% file: sec_approx-rw.tex
\section{Applications of push-pull spans} \label{sec:applications}

We just showed that given a symmetric monoidal functor $U : \sD^\odot \to \sC^\times$, where $\sC$ is cartesian monoidal, and any symmetric monoidal category object $Q : \bG^\opp \times \bD^\opp \to \sC$ satisfying a certain assumption there is a symmetric monoidal $(\infty, 3)$-category $\sPP(U : \sD \to \sC;Q^\otimes)$ of spans in $\sD$ with a push-pull system in $\sC$ valued in $Q^\otimes$. Following in the footsteps of \cite{CHS2022} we will use derived algebraic geometry to define suitable objects that we can plug into our general construction. We will then restrict it a little further to obtain the $(\infty,3)$-category $\sSRW$ and compare it to Kapustin and Rozansky's $\bRW$ from \cite{KR2010}.

It's worth noting that derived algebraic geometry is not the only setting suitable for Rozansky-Witten field theories. A different and maybe more natural choice would be derived differential geometry, which has derived smooth manifolds (see \cite{Spivak2010,Steffens2023}) as its main objects of study. We feel that derived algebraic geometry works better for this paper in particular for several key reasons:
\begin{enumerate}
	\item it allows us to compare our results to those of \cite{CHS2022} (because of the explicit use of derived algebro-geometric methods) and \cite{BCR2023,BCFR2023} (because matrix factorizations are inherently algebraic, and thus it's easier to compute field theories with algebraic rather than smooth stacks);
	\item the universal properties of derived stacks are easier to apply to the construction of \Cref{sec:construction}, in particular the passage from a limit-preserving functor into symmetric monoidal $\infty$-categories to its representing $\bG^\opp \times \bD^\opp$-object which is the key to the arguments in \Cref{sec:modifying-smco,sec:enforcing-push-pull};
	\item we\footnote{The author.} are overall more comfortable with algebraic stacks than with smooth stacks (for educational reasons more than anything else).
\end{enumerate}
However, we want to stress that \Cref{thm:push-pull-3-cat} is general enough that it can be applied to derived smooth stacks, given a well-behaved choice of $Q^\otimes$ -- the representing object determining the $2$- and $3$-morphisms. We leave the task of producing such a $Q^\otimes$ to future work.

\subsection{Background on derived algebraic geometry} \label{sec:background-dag}

Here we recall a number of definitions and results from derived algebraic geometry, just enough for our purposes. We refer the reader to \cite{TV2008,PTVV2013} for a detailed exposition of this theory and the symplectic structures definable in it. We will forgo citing the primary sources directly and instead use the convenient Appendix B of \cite{CHS2022} since the great majority of the material we will need is already covered in that paper.

\subsubsection{Derived stacks} \label{sec:derived-stacks-def}

Fix a base field $\KK$ of characteristic zero. Then there is an $\infty$-category $\Mod^\rmc_{\KK}$ of connective $\KK$-modules (= cochain complexes of vector spaces over $\KK$ localized at the quasi-isomorphisms, with vanishing positive cohomology) which is symmetric monoidal with respect to the tensor product of modules (= derived tensor product of cochain complexes). Let $\sCAlg_\KK^\rmc := \CAlg(\Mod^\rmc_{\KK})$ denote the $\infty$-category of connective $\KK$-algebras (= commutative algebras within connective $\KK$-modules) -- see \cite[Definition B.1.1]{CHS2022}. Define $\dAff := (\sCAlg_\KK^\rmc)^\opp$. For any very large presentable $\infty$-category $\sD$ there is a full subcategory ${\dSt}(\sD) \sub \Fun(\dAff^\opp, \sD)$ of presheaves on $\dAff$ with the following properties:
\begin{enumerate}[label=(\alph*)]
	\item \cite[first part of Lemma B.3.7]{CHS2022} $\dSt := \dSt(\conj{\Spaces})$ is an $\infty$-topos (see \cite{HTT2009}), where $\conj{\Spaces}$ is a very large $\infty$-category of large spaces;
	\item \cite[Corollary B.3.5]{CHS2022} the Yoneda embedding 
	\begin{equation*}
		\Spec : \dAff \to \Fun(\dAff^\opp, \Spaces), \quad A \mapsto \dAff(-, A) \simeq \sCAlg_{\KK}^\rmc(A,-)
	\end{equation*}
	factors through $\dSt$, so we can view it as a fully faithful functor $\Spec : \dAff \to \dSt$;
	\item \cite[second part of Lemma B.3.7]{CHS2022} $\Spec^\opp$ induces a restriction functor $\Fun(\dSt^\opp, \sD) \to \Fun(\dAff^\opp, \sD)$ which specializes to an equivalence
	\begin{equation*}
		\Fun^{\mathrm{R}}(\dSt^\opp, \sD) \xlra{\simeq} \dSt(\sD)
	\end{equation*}
	where the "$\mathrm{R}$" superscript denotes the full subcategory spanned by those functors which preserve (large) limits; the inverse of this equivalence is given by right Kan extension along $\Spec^\opp$.
\end{enumerate}
The $\infty$-category $\dSt$ is that of \emph{derived stacks}, a far reaching generalization of schemes and stacks in algebraic geometry, and its full subcategory $\dAff$ is that of \emph{derived affine stacks} (\cite[Definition B.3.6]{CHS2022}). We will denote an element of $\dAff$ by $A$ if we are thinking of it as an algebra (in its opposite category) and by $\Spec A$ if we are thinking of it as a functor $\dAff^\opp \to \conj{\Spaces}$. By (a) $\dSt$ is complete, cocomplete, and cartesian closed, and admits an embedding
\begin{equation*}
	(-)_\mc{B} : \Spaces \hookrightarrow \dSt, \quad X_\mc{B} := \colim_{X} (\Spec \KK)
\end{equation*}
called the \emph{Betti stack} (\cite[Definition B.3.9]{CHS2022}). By (c) there is an equivalence
\begin{equation*}
	\begin{tikzcd}
		\Fun^{\mathrm{R}}(\dSt^\opp, \conj{\Spaces}) \ar[rr, "(\Spec^\opp)^\ast"] \ar[dr, "\simeq"'] & & \Fun(\dAff^\opp, \conj{\Spaces}) \\
		& \dSt \ar[ur, "\sub"'] &
	\end{tikzcd}
\end{equation*}
whose inverse is given by right Kan extending along $\Spec^\opp$. 
\begin{lmm} \label{lmm:dst-rep}
	Let $\sD$ be an $\infty$-categories enriched over large spaces, let $\sE \sub \Fun(\sD^\opp, \conj{\Spaces})$ be a full subcategory, and let $f : \sD \hookrightarrow \sE$ be the inclusion of a dense full subcategory. Assume further that the restriction $f^{\opp,\ast} : \Fun(\sE^\opp, \conj{\Spaces}) \to \Fun(\sD^\opp, \conj{\Spaces})$ gives an equivalence
		\begin{equation*}
			\Fun^\mathrm{R}(\sE^\opp, \conj{\Spaces}) \simeq \sE,
		\end{equation*}
		where $\Fun^\mathrm{R}$ denotes the limit-preserving functors, whose inverse is right Kan extension along $f^\opp : \sD^\opp \to \sE^\opp$. Then, for any limit-preserving $F : \sE^\opp \to \conj{\Spaces}$, the element $f^{\opp,\ast}(F) \in \sE$ is a representing object for $F$:
	\begin{equation*}
		F(X) \simeq \sE(X, f^{\opp,\ast}(F))
	\end{equation*}
	naturally in $X \in \sE^\opp$.
\end{lmm}

\begin{proof}
	Denseness of $f$ implies that the Yoneda embedding $y : \sD \to \Fun(\sD^\opp, \conj{\Spaces})$ factors through $\sE$. First we will show that for any $X \in \sE$ we have a natural equivalence
	\begin{equation*}
		X \simeq \colim_{Z \in \sD_{/X}} y(Z).
	\end{equation*}
	Indeed, the limit-preserving functor $\sE(-,Y) : \sE^\opp \to \conj{\Spaces}$ is right Kan extended from $\Fun(\sD^\opp, \conj{\Spaces})$ so we have natural equivalences
	\begin{equation*}
		\sE(X,Y) \simeq \lim_{Z \in \sD^\opp_{X/}} \sE(f^\opp(Z), Y) \simeq \sE \left( \colim_{Z \in \sD_{/X}} f(Z), Y \right)
	\end{equation*}
	for any $X \in \sE$, and so the Yoneda lemma yields
	\begin{equation*}
		X \simeq \colim_{Z \in \sD_{/X}} f(Z) \simeq \colim_{Z \in \sD_{/X}} y(Z)
	\end{equation*}
	since $f$ and $y$ agree on objects. Remember that if $X$ is considered as an object of $\sE^\opp$ then
	\begin{equation*}
		X \simeq \lim_{Z \in \sD^\opp_{X/}} f^\opp(Z)
	\end{equation*}
	and so, after a few careful applications of the Yoneda lemma,
	\begin{align*}
		\sE(X, f^{\opp,\ast}(F)) & \simeq \Hom_{\Fun(\sE^\opp, \conj{\Spaces})}(\RKan_{f^\opp}(X), F) \\
		& \simeq \Hom_{\Fun(\sD^\opp, \conj{\Spaces})} \left(\RKan_{f^\opp} \left(\lim_{Z \in \sD^\opp_{X/}} f^\opp(Z) \right), F \right) \\
		& \simeq \lim_{Z \in \sD^\opp_{X/}} \Hom_{\Fun(\sD^\opp, \conj{\Spaces})} (\RKan_{f^\opp}(f^\opp(Z)), F) \\
		& \simeq \lim_{Z \in \sD^\opp_{X/}} \Hom_{\Fun(\sD^\opp, \conj{\Spaces})} (\RKan_{f^\opp}(\sD(-,Z)), F) \\
		& \simeq \lim_{Z \in \sD^\opp_{X/}} \Hom_{\Fun(\sD^\opp, \conj{\Spaces})} (\sE(-, f^\opp(Z)), F) \\
		& \simeq \lim_{Z \in \sD^\opp_{X/}} F(f^\opp(Z)) \\
		& \simeq F \left( \lim_{Z \in \sD^\opp_{X/}} f^\opp(Z) \right) \\
		& \simeq F(X). \qedhere
	\end{align*}
\end{proof}
Thanks to the lemma we can see that every derived stack is canonically the colimit of the derived affine stacks that map into it:
\begin{equation} \label{eqn:dst-is-colimit-of-daff}
	X \simeq \colim_{A \in \dAff_{/X}} \Spec(A).
\end{equation}

We are more interested in these abstract properties than in the actual implementation of derived stacks, but the interested reader should know that we can take $\dSt$ to be the full subcategory of $\Fun(\sCAlg_{\KK}^\rmc, \conj{\Spaces})$ spanned by the \emph{\'etale sheaves}, those product-preserving functors $F$ such that, for any \'etale map $f : A \to B$ in $\sCAlg_{\KK}^\rmc$, $F(A)$ is realized as a limit of the cosimplicial diagram
\begin{equation*}
	\begin{tikzcd}[column sep = small]
		F(B) \ar[r, shift left] \ar[r, shift right] & F(B \otimes_A B) \ar[r, shift left = 2] \ar[r, shift right = 2] \ar[r] & F(B \otimes_A B \otimes_A B) \ar[r, shift left = 1] \ar[r, shift right = 1] \ar[r, shift left = 3] \ar[r, shift right = 3] & \dotsb
	\end{tikzcd}
\end{equation*}
induced by $f$ (see \cite[Definition B.3.1]{CHS2022}). Thus we can think of a derived stack as a presheaf on $\dAff$ which preserves products and satisfies an appropriate notion of descent, which in virtue of the consequences of (c) is given as a canonical colimit of affine presheaves. Given an $X \in \dSt$ we can relativize $\dSt(\sD)$ to $\dSt_X(\sD) \sub \Fun(\dAff^\opp_{/X}, \sD)$ and obtain an equivalence
\begin{equation} \label{eqn:etale-sheaf-gives-functor}
	\Fun^\mathrm{R}(\dSt^\opp_{/X}, \sD) \simeq \dSt_X(\sD)
\end{equation}
(see \cite[Lemma B.3.7]{CHS2022}).

\subsubsection{Quasicoherent sheaves}

Now let $\Mod : \sCAlg_{\KK}^{\rmc} \to \CAlg(\sCat_1)$ denote the functor sending an algebra $A \in \sCAlg_{\KK}^{\rmc}$ to the symmetric monoidal $\infty$-category $\Mod_A$ of $A$-modules in $\Mod_{\KK}$. This satisfies a very strong form of descent that characterizes it as an element of $\dSt(\CAlg(\sCat_1))$, and so by condition (c) (or \cite[second part of Lemma B.3.7]{CHS2022}) we obtain a limit-preserving functor $\QCoh^\otimes : \dSt^\opp \to \CAlg(\sCat_1)$ (see \cite[Definition B.5.2]{CHS2022})\footnote{The target of $\Mod$ should actually be $\mathrm{Pr}^\mathrm{L}$, the very large $\infty$-category of \emph{presentably symmetric monoidal $\infty$-categories}. Since this is itself presentable we can apply (c) with $\sD = \mathrm{Pr}^\mathrm{L}$ and the argument will follow through. Note, however, that by \cite[Definition B.5.12]{CHS2022} the passage to $\CAlg(\sCat_1)$ is justified and does not significantly affect the functor $\QCoh^\otimes$ that we obtain in the end.}. From \Cref{eqn:dst-is-colimit-of-daff} we see that
\begin{equation*}
	\QCoh^\otimes(X) \simeq \lim_{A \in \dAff_{/X}} \QCoh^\otimes(\Spec(A)) \simeq \lim_{\Spec A \in \dAff_{/X}} \Mod_{A}.
\end{equation*}
Define $\cO_X \in \QCoh^\otimes(X)$ to be the element of the limit obtained by taking the algebra $A \in \Mod_A$ for every $\Spec A \to X$. Then from \cite[Proposition B.5.3]{CHS2022} we have that $\QCoh^\otimes(X)$ corresponds to the symmetric monoidal $\infty$-category of \emph{quasicoherent sheaves of $\cO_X$-modules on $X$}. Locally over an affine stack $\Spec A$ this is the $\infty$-category of $A$-modules, and in particular the sheaf $\cO_X$ serves as the unit for the tensor product. For each morphism $f : X \to Y$ of derived stacks the corresponding functor $f^\ast : \QCoh^\otimes(Y) \to \QCoh^\otimes(X)$ is the \emph{pullback} of sheaves; if $\QCoh(X)$ denotes the underlying $\infty$-category of $\QCoh^\otimes(X)$ then $f^\ast$ has a right adjoint $f_\ast : \QCoh(X) \to \QCoh(Y)$, the \emph{pushforward} of sheaves. After composing with the inclusion $\CAlg(\sCat_1) \sub \Fun(\bG^\opp \times \bD^\opp, \Spaces)$ and adjoining the two factors we obtain a functor
\begin{equation*}
	Q^\otimes : \bG^\opp \times \bD^\opp \to \Fun^\mathrm{R}(\dSt^\opp, \Spaces) \simeq \dSt
\end{equation*}
which satisfies the Segal condition in both variables and such that
\begin{equation*}
	\sQ^\otimes_{t,n} := \QCoh^\otimes(X)_{t,n} \simeq \dSt(X, Q^\otimes_{t,n})
\end{equation*}
by \Cref{lmm:dst-rep}; i.e., $Q^\otimes$ parametrizes a representable object for the symmetric monoidal $\infty$-category of quasicoherent sheaves over a derived stack. In particular, since pushforwards of quasicoherent sheaves by maps of derived stacks always exist, the functor
\begin{equation*}
	\sQ \simeq \dSt(-, Q^\otimes_{1, \bullet}) : \dSt^\opp \to \sCat_1
\end{equation*}
satisfies \Cref{ass:assumption}. 

\subsubsection{Aside: $\ZZ_2$-graded quasicoherent sheaves}

Fix $A \in \sCAlg^\rmc_\KK$ and consider the $\infty$-category $\Mod_A^{\ZZ_2}$ of $\ZZ_2$-graded dg-modules over $A$ -- this can be identified with the dg-category of $\ZZ_2$-graded $A$-modules, with $A$ having the standard even-odd $\ZZ_2$-grading\footnote{The underlying module is obtained by taking the direct sum of $A_j$ for even $j$, in degree $0$, and for odd $j$, in degree $1$.}, localized at the quasi-isomorphisms. $\Mod_A^{\ZZ_2}$ is clearly $\KK$-linear and stable (the suspension functor swaps the even and odd degrees), and it is also presentably symmetric monoidal under the $\ZZ_2$-graded tensor product. We can then consider the functor
\begin{equation*}
	F : \sCAlg^\rmc_{\KK} \to \CAlg(\sCat_1), \quad A \mapsto \Mod_A^{\ZZ_2}.
\end{equation*}
By \cite[Theorem D.3.5.2, Remark D.3.5.3]{Lurie2018} the functor $F$ satisfies \'etale descent. Since $F$ also preserves finite products ($\Mod^{\ZZ_2}_{A \oplus B} \simeq \Mod^{\ZZ_2}_{A} \times \Mod^{\ZZ_2}_{B}$) we have from \cite[Definition B.3.1]{CHS2022} that $F$ is an \'etale sheaf and hence $F$ determines a limit preserving functor
\begin{equation*}
	\QCoh^{\ZZ_2} : \dSt^\opp \to \CAlg(\sCat_1).
\end{equation*}
Similarly as for $\QCoh$ we obtain a representing object
\begin{equation*}
	Q^{\ZZ_2, \otimes} : \bG^\opp \times \bD^\opp \to \dSt
\end{equation*}
for $\QCoh^{\ZZ_2}$.

\subsubsection{(Pre)Symplectic derived stacks} \label{sec:symplectic-dag}

Quite a lot of calculus can be developed over special kinds of derived stacks. When we restrict ourselves to the nice subcategory $\dSt^\Art$ of \emph{Artin stacks} (see \cite[Definition B.7.4]{CHS2022}) then each $X \in \dSt^\Art$ has a unique \emph{cotangent complex} $\LL_X \in \QCoh(X)$ (see \cite[Definition B.10.2, Theorem B.10.8]{CHS2022}), a dualizable quasicoherent sheaf on $X$ which behaves similarly to the cotangent bundle of smooth manifolds; we denote its dual by
\begin{equation*}
	\TT_X := \LL_X^\vee
\end{equation*}
and call it the \emph{tangent complex}. When $X = \Spec A$ is affine, then $\LL_X =: \LL_A \in \Mod_A$ is the corepresenting object for the functor $\mathrm{Der}_\KK(A, -)$ sending a $\KK$-module $M$ to the space of derivations $A \to M$ (see \cite[Section B.2]{CHS2022}). The tangent and cotangent complexes are functorial in $X$: each map $f : Y \to X$ induces a map $f^\ast \LL_X \to \LL_Y$ and a dual map $\TT_Y \to f^\ast \TT_X$. Moreover, one can define the \emph{$p$th exterior power} of the cotangent complex, $\bigwedge^p \LL_X \in \QCoh(X)$, and the associated \emph{space of $p$-forms on $X$}, $\cA^p(X) \in \Spaces$, which satisfy the equivalence
\begin{equation*}
	\cA^p(X) \simeq \Hom_{\QCoh(X)}(\cO_X, \textstyle \bigwedge^p \LL_X) \simeq \Hom_{\QCoh(X)}(\textstyle \bigwedge^p \TT_X, \cO_X)
\end{equation*}
(see \cite[Proposition B.12.2]{CHS2022}). A $p$-form $\omega$ together with the data of a closed null-homotopy $d\omega \simeq 0$ is a \emph{closed $p$-form}, and these also form a space $\cA^{p, \rmcl}(X)$ (see \cite[Definition B.12.1]{CHS2022}). The differential $d$ is the external de Rham differential (see \cite[Section 1.2]{PTVV2013}, appearing as $\epsilon$), not the internal differential on the complex of forms. Both $\cA^p(X)$ and $\cA^{p, \rmcl}(X)$ are functorial in $X$ and preserve limits so they are represented by some derived stacks $A^p$ and $A^{p, \rmcl}$, respectively (see \cite[Proposition 12.3]{CHS2022}).

Recall that a symplectic manifold is a smooth manifold $M$ together with a closed non-degenerate $2$-form $\omega \in \Omega^2(X)$. A \emph{presymplectic derived stack} is a derived Artin stack $X$ together with a closed $2$-form $\omega \in \cA^{2,\rmcl}(X)$ on it (see \cite[Definition 2.1.3]{CHS2022}). Equivalently, it is a morphism $X \to A^{2, \rmcl}$ whose source is a derived Artin stack, so we can define the $\infty$-category of presymplectic stacks as
\begin{equation*}
	\PreSymp := \dSt^{\Art} \times_{\dSt} \dSt_{/A^{2, \rmcl}}
\end{equation*}
(see \cite[Definition 2.3.2]{CHS2022}). This category has a symmetric monoidal structure determined by
\begin{equation} \label{eqn:presymp-product}
	(X, \omega_X) \odot (Y, \omega_Y) \simeq (X \times Y, \pi_X^\ast \omega_X + \pi_Y^\ast \omega_Y)
\end{equation}
where $\pi_X$ and $\pi_Y$ are the projections of $X \times Y$ to $X$ and $Y$, respectively. The point $\ast = \Spec \KK \in \dSt$ has a trivial presymplectic form $\omega_\ast = 0$ (\cite[Example 2.1.4]{CHS2022}) and acts as the unit for $\odot$. The forgetful functor $U : \PreSymp^{\odot} \to \dSt^\times$ is symmetric monoidal, as can be easily seen from \Cref{eqn:presymp-product}.

Note that
\begin{equation*}
	\cA^2(X) \simeq \Hom_{\QCoh(X)}(\cO_X, \textstyle \bigwedge^2 \LL_X) \simeq \Hom_{\QCoh(X)}(\TT_X, \LL_X)
\end{equation*}
by the dualizability of $\LL_X$. If $\omega \in \cA^2(X)$ is a closed $2$-form such that the corresponding map $\TT_X \to \LL_X$ is an equivalence then the presymplectic stack $(X, \omega)$ is called \emph{symplectic} (see \cite[Definition 2.1.2]{CHS2022}); this mirrors the non-degeneracy condition for symplectic forms on smooth manifolds.

\subsubsection{Lagrangian correspondences}

A span 
\begin{equation*}
	\begin{tikzcd}
		& (L, \omega_L) \ar[dl, "f"'] \ar[dr, "g"] & \\
		(X, \omega_X) & & (Y, \omega_Y)
	\end{tikzcd}
\end{equation*}
in $\PreSymp$ is called an \emph{isotropic correspondence} (see \cite[Definition 2.4.3]{CHS2022}). This is really the data of a commutative diagram
\begin{equation*}
	\begin{tikzcd}
		& L \ar[dl, "f"'] \ar[dr, "g"] & \\
		X \ar[dr, "\omega_X"'] & & Y \ar[dl, "\omega_Y"] \\
		& A^{2, \rmcl} &
	\end{tikzcd}
\end{equation*}
in $\dSt$, with $X, Y, L \in \dSt^\Art$, which includes the data of an equivalence $f^\ast \omega_X - g^\ast \omega_Y \simeq 0$ in $\cA^{2,\rmcl}(L)$. If $X \simeq \ast$ is the point then $\omega_X \simeq 0$ and so $0 \simeq g^\ast \omega_Y$, in which case we call $g$ an \emph{isotropic morphism}. 

Note that an isotropic correspondence as above gives rise to a commutative hexagon
\begin{equation*}
	\begin{tikzcd}
		& \TT_L \ar[dr] \ar[dl] & \\
		f^\ast \TT_X \ar[d] & & g^\ast \TT_Y \ar[d] \\
		f^\ast \LL_X \ar[dr] & & g^\ast \LL_Y \ar[dl] \\
		& \LL_L &
	\end{tikzcd}
\end{equation*}
in $\QCoh(L)$. We say that the isotropic correspondence is \emph{Lagrangian} if this is a limit diagram (see \cite[Definition 2.4.7]{CHS2022}). The definition of a Lagrangian span extends to that of a \emph{Lagrangian $k$-fold span} (see \cite[Definition 2.6.11, Definition 2.9.3]{CHS2022}): informally, these are functors $\bS^k \to \PreSymp$ such that each generating span is a Lagrangian correspondence and each composite, obtained by taking pullbacks, is Lagrangian as well. It was shown that the collection of such spans forms a subcategory $\ms{L}\mathrm{ag}$ of the $\infty$-category of spans in $\PreSymp$ (this is $\Lag^{\ast,0}_1$ in \cite[Section 2.9]{CHS2022}). Thus there is a full subcategory $\what{\Lag} \sub \what{\GS}(\PreSymp)^{1,0, \rmcart}$ which straightens to the $\infty$-category $\ms{L}\mathrm{ag}$.

\subsection{Commutative Rozansky-Witten theories} \label{sec:rw-3-cat}

\subsubsection{Push-pull spans in presymplectic stacks}

Set $\sC = \dSt$ and $\sD = \PreSymp$ be the two $\infty$-categories defined above with their symmetric monoidal forgetful functor $U : \sD^\odot \to \sC^\times$, and let $Q^{\ZZ_2,\otimes}$ be the symmetric monoidal category object of $\sC$ classifying the symmetric monoidal $\infty$-category of $\ZZ_2$-graded quasicoherent sheaves. From \Cref{dfn:relative-push-pull-spans} we have a symmetric monoidal $(\infty,3)$-category $\sPP(U;Q^{\ZZ_2,\otimes})$ of relative push-pull spans with the following informal description:

\textbf{Objects:} The objects of $\sPP(U;Q^{\ZZ_2,\otimes})$ are presymplectic derived stacks $(X, \omega)$. The monoidal product $(X, \omega_X) \odot (Y, \omega_Y)$ is defined in \Cref{eqn:presymp-product}.

\textbf{$1$-morphisms:} The $1$-morphisms are isotropic correspondences $(X, \omega_X) \leftarrow L \rightarrow (Y, \omega_Y)$, i.e. spans of presymplectic derived stacks. If $(-)^\diamondsuit : \PreSymp \to \PreSymp$ denotes the dualization functor taking $(X, \omega)$ to $(X, -\omega)$ then there is an equivalence
\begin{equation} \label{eqn:morph-dual-in-pp}
	\sPP(X, Y) \simeq \sPP(\ast, X^\diamondsuit \odot Y)
\end{equation}
for any $X, Y \in \PreSymp$. This follows from the equivalence of Lagrangian structures explained in \cite[Section 2.4]{CHS2022}. The composition of two correspondences is given by taking pullbacks, and the identity $1$-morphism $X \to X$ is the diagonal correspondence $X \leftarrow X \to X$. 

\textbf{$2$-morphisms:} The $2$-morphisms between correspondences $(X, \omega_X) \leftarrow L \rightarrow (Y, \omega_Y)$ and $(X, \omega_X) \leftarrow M \rightarrow (Y, \omega_Y)$ are given by $\ZZ_2$-graded quasicoherent sheaves over the derived intersection $L \cap M := L \times_{X \times Y} M$ of $L$ and $M$. The horizontal composition is simply given by pulling back and tensoring: given composable spans
	\begin{equation*}
		\begin{tikzcd}
			& L \ar[dl] \ar[dr] & & L' \ar[dl] \ar[dr] & \\
			X & & Y & & Z \\
			& M \ar[ul] \ar[ur] & & M' \ar[ul] \ar[ur] &
		\end{tikzcd}
	\end{equation*}
	and sheaves $\cM \in \QCoh^{\ZZ_2}(L \cap M)$ and $\cM' \in \QCoh^{\ZZ_2}(L' \cap M')$, their composite is
	\begin{equation*}
		\pi^\ast \cM \otimes (\pi')^\ast \cM' \in \QCoh^{\ZZ_2}((L \times_Y L') \cap (M \times_Y M'))
	\end{equation*} 
	where $\pi$ and $\pi'$ are the canonical projections as in the diagram
	\begin{equation*}
		\begin{tikzcd}
			& (L \times_Y L') \cap (M \times_Y M') \ar[d, "\simeq"] \ar[ddl, "\pi"', bend right] \ar[ddr, "\pi'", bend left] & \\
			& (L \cap M) \times_{Y} (L' \cap M') \ar[dl] \ar[dr] & \\
			L \cap M & & L' \cap M'
		\end{tikzcd} 
	\end{equation*}
	The identity for this composition is the structure sheaf $\cO_{X \times X} \in \QCoh^{\ZZ_2}(X \cap X) = \QCoh^{\ZZ_2}(X \times X)$: in fact, given a sheaf $\cM \in \QCoh^{\ZZ_2}(L \cap M)$ where $L$ and $M$ are Lagrangian correspondences from $X$ to $Y$ we have that the projections $\pi : (L \times_Y Y) \times (M \times_Y Y) \to L \cap M$ and $\pi' : (L \times_Y Y) \times (M \times_Y Y) \to Y \cap Y = Y \times Y$ are none other than the identity map in the first case and the natural projection in the second case, so
	\begin{equation*}
		\pi^\ast \cM \otimes (\pi')^\ast \cO_{Y \times Y} \simeq \cM \otimes \cO_{L \cap M} \simeq \cM.
	\end{equation*}
	The vertical composition is given by the push-pull formula: given three spans
	\begin{equation*}
		\begin{tikzcd}
			& L \ar[dl] \ar[dr] & \\
			X & M \ar[l] \ar[r] & Y \\
			& N \ar[ul] \ar[ur] &
		\end{tikzcd}
	\end{equation*}
	and sheaves $\cM \in \QCoh^{\ZZ_2}(L \cap M)$ and $\cN \in \QCoh^{\ZZ_2}(M \cap N)$, their composite is
	\begin{equation*}
		i_{M,\ast} (i_N^\ast \cM \otimes i_L^\ast \cN) \in \QCoh^{\ZZ_2}(L \cap N)
	\end{equation*}
	where
	\begin{equation*}
		 \begin{tikzcd}
		 	& L \cap M \cap N \ar[dl, "i_L"'] \ar[d, "i_M"] \ar[dr, "i_N"] & \\
		 	M \cap N & L \cap N & L \cap M
		 \end{tikzcd} 
	\end{equation*}
	are the three canonical inclusions of derived intersections. The identity for this composition is given by $\nabla_\ast \cO_{L}$, where $\nabla : L \to L \cap L = L \times_{X \times Y} L$ is the diagonal map. This follows from a quick computation: since the diagram
	\begin{equation*}
		\begin{tikzcd}
			L \cap M \ar[r, "\nabla"] \ar[d, "j"'] & L \cap M \cap M \ar[d, "i_L"] \\
			M \ar[r, "\nabla"'] & M \cap M
		\end{tikzcd}
	\end{equation*}
	commutes and $\nabla$ satisfies the base change and the projection formula (see \cite[Definitions B.8.1, B.8.2, and Theorem B.8.12]{CHS2022}, if $\cM \in \QCoh^{\ZZ_2}(L \cap M)$ is another sheaf then
	\begin{align*}
		\nabla_\ast \cO_M \circ \cM & :=  i_{M, \ast}(i_M^\ast \cM \otimes i_L^\ast \nabla_\ast \cO_{M}) \\
		& \simeq i_{M, \ast} (i_M^\ast \cM \otimes \nabla_\ast j^\ast \cO_{M}) & \text{(base change)} \\
		& \simeq i_{M, \ast} \nabla_\ast (\nabla^\ast i_M^\ast \cM \otimes j^\ast \cO_{M}) & \text{(projection)} \\
		& \simeq (i_M \circ \nabla)_\ast ((i_M \circ \nabla)^\ast \cM \otimes \cO_{L \cap M}) \\
		& \simeq \cM & (i_M \circ \nabla \simeq \id_{L \cap M})
	\end{align*}
	and similarly for the other identity.

\textbf{$3$-morphisms:} The $3$-morphisms are given by maps of sheaves over a span: a $3$-morphism between two sheaves $\cM$ and $\cN$ over $L \cap M$ is given by a span $L \cap M \xlongleftarrow{f} N \xlra{g} L \cap M$ and a morphism $\alpha : f^\ast \cM \to g^\ast \cN$ in $\QCoh^{\ZZ_2}(N)$. If we reduce to the spans of the form
\begin{equation*}
	L \cap M  \xlongleftarrow{\id} L \cap M \xlra{\id} L \cap M
\end{equation*}
then we obtain the space of maps between $\cM$ and $\cN$ in $\QCoh^{\ZZ_2}(L \cap M)$. The horizontal composition is given by pulling back and tensoring in the only way which is compatible with the composition of $1$-morphisms; the vertical composition is given by the push-pull formula for the vertical composition of $2$-morphisms, but applied to the morphisms; the transversal composition is simply the composition of maps of sheaves after suitable pullbacks. The identities for these compositions are the obvious ones induced by the formulas above.

\subsubsection{Further restrictions}

With a little bit of care we can restrict to a smaller $(\infty, 3)$-category which better approximates the idea in \cite{KR2010}. Consider the projections
	\begin{align*}
		\sigma : \what{\PP}(U ;Q^{\ZZ_2,\otimes})_{t,k,n,l} & \to \PreSymp^{\times tk(n+1)(l+1)} \\
		\lambda : \what{\PP}(U;Q^{\ZZ_2,\otimes})_{t,k,n,l} & \to (\what{\GS}(\PreSymp)^{1}_{k})^{\times t(n+1)(l+1)} \\
		\rho : \what{\PP}(U;Q^{\ZZ_2,\otimes})_{t,k,n,l} & \to (\what{\GS}(\PreSymp)^{1}_{n})^{\times t(k+1)(l+1)}
	\end{align*}
	obtained by composing various Segal maps: $\sigma$ sends a generalized span to the collection of its generating objects in $\PreSymp$, $\lambda$ sends it to the collection of its horizontal slices (spans in $\PreSymp$ underlying the composition of $1$-morphisms), and $\rho$ send it to the collection of its transversal slices (spans in $\PreSymp$ underlying the composition of $3$-morphisms). Since $\what{\Lag}_k$ is contained in $\what{\GS}(\PreSymp)^1_k$ we can make the following definitions:
\begin{dfn} \label{dfn:horizontal-Lagrangian}
	A generalized span $(F,\eta) \in \what{\PP}(U;Q^{\ZZ_2,\otimes})_{t,k,n,l}$
	\begin{enumerate}
		\item \emph{has symplectic objects} if $\sigma(F)$ is a tuple of symplectic derived stacks,
		\item \emph{has Lagrangian $1$-morphisms} if $\lambda(F)$ is in $(\what{\Lag}_k)^{\times t(n+1)(l+1)}$, 
		\item \emph{is horizontally Lagrangian} if it has both symplectic objects and Lagrangian $1$-morphisms, and
		\item \emph{has standard maps of sheaves} if every component of $\rho(F)$ is a constant iterated span, i.e. all the projection morphisms are equivalences.
	\end{enumerate}
	Denote by $\what{\SRW} \sub \what{\PP}(U;Q^{\ZZ_2,\otimes})$ the full subcategory spanned by the generalized spans which are horizontally Lagrangian and have standard maps of sheaves.
\end{dfn}

Being horizontally Lagrangian implies that, in our desired $(\infty,3)$-category, the objects are all symplectic stacks and the $1$-morphisms are all Lagrangian correspondences. Having standard maps of sheaves implies that if $L$ and $M$ are (now Lagrangian) spans between (now symplectic) stacks $X$ and $Y$, then the $\infty$-category of $2$-morphisms between $L$ and $M$ is precisely $\QCoh^{\ZZ_2}(L \cap M)$. We will conclude this section by proving that enforcing these properties still returns an $(\infty,3)$-category.

\begin{prp}
	All four properties in \Cref{dfn:horizontal-Lagrangian} are stable under pulling back by maps in $\bG^\opp \times \bD^{3, \opp}$.
\end{prp}

\begin{proof}
	It's easy to see that the fourth property is stable since equivalences are preserved by pullbacks. Pulling back by maps in the second and third $\bD^\opp$ factors simply changes the number of objects of $k$-folds spans under consideration, and this does not affect the properties. Pulling back by maps in $\bG^\opp$ amounts to taking products of spans or inserting trivial spans; taking products preserves both the symplectic property and the Lagrangian property (see \cite[Section 2.10]{CHS2022}) and so we're done. Finally, the maps in the first $\bD^\opp$ factor induce the structure maps of $\Lag$; the latter is obviously closed under these structure maps since it forms an $\infty$-category.
\end{proof}

\begin{crl} \label{crl:rw-is-category}
	The projection $\what{\SRW} \sub \what{\PP}(U;Q^{\ZZ_2,\otimes}) \to \bG^\opp \times \bD^{3, \opp}$ is a cocartesian fibration.
\end{crl}

\begin{proof}
	Follows immediately from the previous proposition since the inclusion of $\what{\SRW}$ into $\what{\PP}(U;Q^{\ZZ_2,\otimes})$ is full -- compare, for example, with the proof of \Cref{crl:cart-spans-are-cocart-fibration}.
\end{proof}

\begin{dfn}
	Let $\SRW : \bG^\opp \times \bD^{3, \opp} \to \sCat_1$ denote a straightening of the projection $\what{\SRW} \to \bG^\opp \times \bD^{3, \opp}$, which exists by \Cref{crl:rw-is-category}.
\end{dfn}

\begin{prp} \label{prp:rw-is-segal}
	The functor $\SRW$ satisfies the Segal condition in all four variables.
\end{prp}

\begin{proof}
	It is easy to see that the Segal maps for $\what{\SRW}$ are pulled back from the Segal maps for $\what{\PP}(U;Q^{\ZZ_2,\otimes})$ along the inclusion $\what{\SRW} \sub \what{\PP}(U;Q^{\ZZ_2,\otimes})$: in fact, the projections $\sigma,\lambda,\rho$ in \Cref{dfn:horizontal-Lagrangian} are the composition of the Segal maps for the first, third, and fourth variable variables followed by further projections, which are themselves stable under pullbacks. Since the Segal maps for $\what{\PP}(U;Q^{\ZZ_2,\otimes})$ are equivalences by the results of \Cref{sec:enforcing-push-pull} we obtain that the Segal maps for $\what{\SRW}$ are also equivalences, proving the claim.
\end{proof}

\begin{crl} \label{crl:srw}
	The functor $\SRW : \bG^\opp \times \bD^{3, \opp} \to \sCat_1$ determines a symmetric monoidal $(\infty, 3)$-category. This is denoted $\sSRW$, and we call it the \emph{$(\infty,3)$-category of commutative Rozansky-Witten theories}.
\end{crl}

\begin{proof}
	Combine \Cref{prp:rw-is-segal} with the construction of a complete $3$-fold Segal space from a $3$-uple Segal space.
\end{proof}

%% file: sec_appendix.tex
\appendix

\section{Higher categorical background} \label{sec:appendix}

Here we collect a list of definitions and results about higher categries and higher algebra that we need for \Cref{sec:construction,sec:applications}. Some results are well known, but we wrote down a proof of those for which we could not track down a precise reference.

\subsection{Functors versus fibrations} \label{sec:functors-vs-fibrations}

Throughout this work we make extensive use of Lurie's \emph{straightening/unstraightening equivalence}, which asserts that for any $\infty$-category $\sD$ there is an adjoint equivalence
\begin{equation} \label{eqn:str-unstr}
	\begin{tikzcd}
		{\sCat_1}_{/\sD}^{\rmcocart} \ar[rr, bend left, "\mathrm{St}"{name=U}] & & \Fun(\sD, \sCat_1) \ar[ll, bend left, "\mathrm{Un}"{name=D}]
		\ar[from=U, to=D, "\rotatebox{-90}{$\dashv$}" description, phantom]
	\end{tikzcd}
\end{equation}
between the category of functors $\sD \to \sCat_1$ and the category of cocartesian fibrations with target $\sD$ and maps preserving cocartesian morphisms. See \cite[Section 2]{HTT2009} for information about cocartesian fibrations and \cite[Section 3.2]{HTT2009} or \cite{HHR2021} for a detailed proof of straightening/unstraightening. 

For our purposes it will be sufficient to have just a few properties of this equivalence. Recall:
\begin{dfn}
	Let $P : \sE \to \sD$ be a functor and $f : e_1 \to e_2$ a morphism in $\sE$. Then $f$ is \emph{$P$-cocartesian} (or just \emph{cocartesian} if $P$ can be inferred from the context) if for every $e_3 \in \sE$ the diagram
	\begin{equation*}
		\begin{tikzcd}
			\sE(e_2, e_3) \ar[r, "f^\ast"] \ar[d] & \sE(e_1, e_3) \ar[d] \\
			\sD(P(e_2), P(e_3)) \ar[r, "P(f)^\ast"] & \sD(P(e_1), P(e_3))
		\end{tikzcd}
	\end{equation*}
	is a pullback square, or equivalently the diagram of undercategories
	\begin{equation*}
		\begin{tikzcd}
			\sE_{e_2/} \ar[r, "f^\ast"] \ar[d] & \sE_{e_1/} \ar[d] \\
			\sD_{P(e_2)/} \ar[r, "F(f)^\ast"] & \sD_{P(e_1)/}
		\end{tikzcd}
	\end{equation*}
	is a pullback square. We also say that $P$ is a \emph{cocartesian fibration} if for every $g : d_1 \to d_2$ in $\sD$ and every $e_1 \in \sE$ with $P(e_1) \simeq d_1$ there is a cocartesian morphism $f$ in $\sE$ such that $P(f) \simeq g$. If $P^\opp$ is a cocartesian fibration then $P$ is a \emph{cartesian fibration}.
\end{dfn}
If $g : d_1 \to d_2$ is a morphism in $\sD$, $e_1$ is a lift of $d_1$, and $f$ is a cocartesian lift of $g$ starting at $e_1$, we sometimes denote the target of $f$ by $g_!(e_1)$. This is meant to signify that if the cocartesian fibration $P$ comes from a functor $F : \sD \to \sCat_1$ via \Cref{eqn:str-unstr} then $g_!(e_1) \simeq F(g)(e_1)$ in the $\infty$-category $F(d_2)$. In particular, precomposing by $f$ gives an equivalence
\begin{equation*}
	\sE_{/\id_{d_2}}(g_!(e_1), e') \simeq \sE_{/g}(e_1, e')  
\end{equation*}
for any lift $e'$ of $d_2$ and this can be used to show that any two cocartesian lifts of $g$ have equivalent targets (i.e. $g_!(e_1)$ is essentially unique).

The equivalence in \Cref{eqn:str-unstr} is functorial in $\sD$: if $G : \sC \to \sD$ and $F : \sD \to \sCat_1$ are functors then
\begin{equation*}
	\mathrm{Un}(F \circ G) \simeq \sC \times_{\sD} \mathrm{Un}(F)
\end{equation*}
as cocartesian fibrations over $\sC$. In particular when $G = \{d\} : \Delta^0 \to \sD$ picks out an object $d \in \sD$ then we have an equivalence of $\infty$-categories between the fiber of $\mathrm{Un}(F)$ at $d$ and the value $F(d)$:
\begin{equation*}
	F(d) \simeq \mathrm{Un}(F \circ \{d\}) \simeq \Delta^0 \times_{\sD} \mathrm{Un}(F) \simeq \mathrm{Un}(F)_d.
\end{equation*}

Passing to $\sD^\opp$ gives an an equivalence for cartesian fibrations instead, and replacing $\sCat_1$ with $\Spaces$ gives another version for left and right fibrations. This equivalence happens inside an $\infty$-category of large $\infty$-categories within which $\sD$, ${\sCat_1}_{/\sD}^{\rmcocart}$, $\sCat_1$, and $\Fun(\sD, \sCat_1)$ live. We will not dwell too much on size issues, but the interested reader can assume that we're fixing a hierarchy of Grothendieck universes associated to cardinals $\kappa_1 < \kappa_2 < \dotsb$ and using ``small'' to refer to $\kappa_1$-small $\infty$-categories, ``large'' to refer to $\kappa_2$-small $\infty$-categories, and ``very large'' for the subsequent ones.

Given a functor $F : \sD \to \sCat_1$ we will denote by $\what{F} \to \sD$ or $\what{F} \to \sD^\opp$ the associated cocartesian or cartesian fibration; the variance is usually implicit from the context, but when we need to disambiguate between the two we will write $\mathrm{Un}(F) \to \sD$ for the cocartesian unstraightening and $\mathrm{Un}^\opp(F) \to \sD^\opp$ for the cartesian unstraightening. By \cite[Theorem 7.4, Theorem 7.6]{GHN2017} we have equivalences
\begin{equation*}
	\mathrm{Un}(F) \simeq \int^{\sC^\opp} F(-) \times \sC_{(-)/}, \quad \mathrm{Un}^\opp(F) \simeq \int^{\sC^\opp} F(-) \times (\sC^\opp)_{/(-)}
\end{equation*}
(see \Cref{sec:ends} for the notation). At the level of ordinary categories (i.e. when $\sD = \mathbf{D}$ is ordinary and $F$ lands in $\bCat_1$) the above formulas turn into explicit constructions for $\what{F}$, the \emph{Grothendieck (op)fibration associated to $F$} or \emph{category of elements of $F$}, in terms of objects of $\mathbf{D}$ and elements of the categories $F(d)$ for $d \in \mathbf{D}$.

If $\sE \to \sD$ is a cocartesian fibration and $f : x \to y$ is a morphism in $\sD$ then we denote the induced functor on fibers by $f_\ast : \sE_x \to \sE_y$. If $\sD = \sC^\opp$ and $f : x \to y$ is a morphism in $\sC$ then we denote the induced functor by $f^\ast : \sE_y \to \sE_x$ instead.

\subsection{A model for higher categories} \label{sec:higher-cats}

Here we briefly discuss the model for $(\infty, n)$-categories that will be used throughout the paper. The following material is presented in great detail (including the relevant sources) in Sections 3, 4, and 7 of \cite{Haugseng2018}; we direct the interested reader to it for a more complete exposition. For the remainder of this section, $\sD$ is an $\infty$-category with finite limits.

\begin{dfn} \label{dfn:category-object}
	A simplicial object $X : \bD^\opp \to \sD$ is called a \emph{category object} in $\sD$ if, for any $n \geq 0$, the maps
	\begin{equation*}
		\begin{tikzcd}[row sep = 0.05em]
			{[0]} \ar[r, "\sigma_i"] & {[n]}, & & {[1]} \ar[r, "\rho_j"] & {[n]} \\
			0 \ar[r, mapsto] & i & & (0, 1) \ar[r, mapsto] & (j, j+1)
		\end{tikzcd}
	\end{equation*}
	in $\bD$ induce an equivalence
	\begin{equation*}
		(\rho_0 \times_{\sigma_1} \dotsb \times_{\sigma_{n-1}} \rho_{n-1})_\ast : X_n \xlra{\simeq} X_1 \times_{X_0} \dotsb \times_{X_0} X_1
	\end{equation*}
	in $\sD$, where the right-hand side contains $n$ copies of $X_1$. The above equation is called a \emph{Segal condition}. The full subcategory of $\Fun(\bD^\opp, \sD)$ spanned by the category objects is denoted by $\Cat(\sD)$. A \emph{Segal space} is a category object in $\Spaces$.
\end{dfn}

The data of a Segal space $X$ includes a space $X_0$ of ``objects'' $x, y, z, \dotsc$ and a space $X_1$ of ``morphisms'' $f, g, h, \dotsc$, each of which has a source $s : X_1 \to X_0$ and target $t : X_1 \to X_0$ determined by the two maps $\sigma_0, \sigma_1 : [0] \to [1]$ in $\bD$. The inclusion $[1] \to [2]$ defined by $c(0) = 0$ and $c(1) = 2$ produces a map $c : X_2 \to X_1$ and thus the Segal condition guarantees that there is a ``composition'' operation $X_1 \times_{X_0} X_1 \xleftarrow{\simeq} X_2 \xlra{c} X_1$, defined up to equivalence, sending a pair $(f, g)$ of composable edges to a candidate for the composite $g \circ f \in X_1$. The space of these candidates is contractible and so, up to equivalence, $g \circ f$ is unique. The assumption that the Segal condition holds for all other values of $n$ guarantees that this composition is unital and associative up to equivalence.

Given a category object $X \in \Cat(\sD)$ and an object $d \in \sD$ we obtain a Segal space $X(d) \in \Cat(\Spaces)$ by taking the composite
\begin{equation*}
	\bD^\opp \xlra{X} \sD \xlra{\Hom_\sD(d,-)} \Spaces.
\end{equation*}
It will be helpful to think of an arbitrary category object in this way, as a representing object for Segal spaces parametrized (contravariantly) by $\sD$.

Note that, for any $\infty$-category $\sA$, limits in $\Fun(\sA, \sD)$ are computed componentwise in $\sD$ so $\Cat(\sD)$ is again an $\infty$-category with finite limits.

\begin{dfn}
	The $\infty$-category $\Cat^n(\sD)$ of \emph{$n$-uple category objects} in $\sD$ is defined inductively as follows:
	\begin{itemize}
		\item $\Cat^0(\sD) := \sD$;
		\item $\Cat^n(\sD) := \Cat(\Cat^{n-1}(\sD))$ for $n \geq 1$.
	\end{itemize}
	An \emph{$n$-uple Segal space} is an $n$-uple category object in $\Spaces$. If $X$ is an $n$-uple Segal space and $a$ is an $n$-tuple of $0$s and $1$s containing exactly $k$ many $1$s, then a vertex of the space $X_a$ is called a \emph{$k$-morphism} of $X$; a $0$-morphism is also called an \emph{object} of $X$.
\end{dfn}

By unwinding the definition we see that an $n$-uple category object in $\sD$ is equivalently a functor $X : \bD^{n, \opp} \to \sD$ which satisfies the Segal condition in each variable independently. The data of $X$ can be thought of as a collection of $n$-dimensional cubes: the legs of each cube are finite sequences of composable $1$-morphisms between objects, and the higher dimensional faces are represented by the $k$-morphisms glued along their common $(k-1)$-morphisms. Here is a picture of a $2$-cube in $X_{m,n}$:
\begin{equation*}
	\begin{tikzcd}
		x_{0, 0} \ar[r, "f_{1,0}"] \ar[d, "f_{0,1}"] & x_{1, 0} \ar[r, "f_{2,0}"] \ar[d, "f_{1,1}"] & x_{2, 0} \ar[r, "f_{3,0}"] \ar[d, "f_{2,1}"] & \dotsb \ar[r, "f_{m,0}"] \ar[d] & x_{m,0} \ar[d, "f_{m,1}"] \\
		x_{0, 1} \ar[r, "f_{1,1}"] \ar[d, "f_{0,2}"] & x_{1,1} \ar[r, "f_{2,1}"] \ar[d, "f_{1,2}"] & x_{2, 1} \ar[r, "f_{3,1}"] \ar[d, "f_{2,2}"] & \dotsb \ar[r, "f_{m,1}"] & x_{m,1} \ar[d, "f_{m,2}"] \\
		\vdotsb \ar[d, "f_{0,n}"] & \vdotsb \ar[d, "f_{1,n}"] & \vdotsb \ar[d, "f_{2,n}"] & \ddotsb & \vdotsb \ar[d, "f_{m,n}"] \\
		x_{0,n} \ar[r, "f_{1, n}"] & x_{1,n} \ar[r, "f_{2, n}"] & x_{2,n} \ar[r, "f_{3, n}"] & \dotsb \ar[r, "f_{m, n}"] & x_{m,n}
	\end{tikzcd}
\end{equation*}
The commutative squares here contain the data represented by $X_{1,1}$.

Note that an $n$-uple Segal space has exactly $\binom{n}{k}$ many possibly different kinds of $k$-morphisms. To pass from the ``cubical'' data of an $n$-uple category to the ``globular'' data of an $n$-category we must enforce the condition that there be only one kind of $k$-morphism for every $k$, so that we can talk about $(k+1)$-morphisms as going from one $k$-morphism to another. This degeneracy condition on the $k$-morphisms is encoded by letting certain representing objects be trivial or constant. In general we say that an $n$-uple category object $X$ is \emph{essentially constant} if it the canonical natural transformation $c_{X_{0,\dotsc,0}} \to X$, where $c_{X_{0,\dotsc,0}} : \bD^{n,\opp} \to \sD$ is the constant functor with value $X_{0, \dotsc, 0}$, is an equivalence.

\begin{dfn}
	The $\infty$-category $\Seg^n(\sD)$ of \emph{$n$-fold Segal objects} in $\sD$ is defined inductively as follows:
	\begin{itemize}
		\item $\Seg^0(\sD) := \sD$;
		\item $\Seg^n(\sD) \sub \Cat(\Seg^{n-1}(\sD))$ is the full subcategory spanned by the functors $X : \bD^\opp \to \Seg^{n-1}(\sD)$ such that $X_0$ is essentially constant as an $(n-1)$-uple category object in $\sD$.
	\end{itemize}
	An \emph{$n$-fold Segal space} is an $n$-fold Segal object in $\Spaces$.
\end{dfn}

Up to equivalence, the data of an $n$-fold Segal object reduces to that of the objects $X_{1, \dotsc, 1, 0, \dotsc, 0}$ ($k$ many $1$s) representing the $k$-morphisms and the induced source, target, composition, identity, and higher associativity maps.

We might be tempted at this point to use $n$-fold Segal spaces as our models of $(\infty, n)$-categories. However, as it turns out, the homotopy theory of $\Seg^n(\Spaces)$ does not reflect the desired homotopy theory of $(\infty, n)$-categories in that fully faithful and essentially surjective maps of $n$-fold Segal spaces, adapted to that setting, are not necessarily equivalences. Localizing with respect to those maps gives the correct model of $(\infty, n)$-categories, which also happens to be equivalent to the subcategory of $\Seg^n(\Spaces)$ spanned by those $n$-fold Segal spaces $X$ which are \emph{complete}, meaning that all the equivalences internal to the $n$-fold category $X$ arise through degeneracies from the space of objects $X_{0, \dotsc, 0}$; the case $n = 1$ was proved by Joyal and Tierney in \cite{JT2007}, and Barwick generalized the result to $n > 1$ in his thesis.

\begin{dfn} \label{dfn:complete-segal-space}
	An \emph{$(\infty, n)$-category} is a complete $n$-fold Segal space. We denote by $\sCat_n \sub \Seg^n(\Spaces)$ the full subcategory spanned by the $(\infty, n)$-categories.
\end{dfn}

The following proposition allows us to bypass checking for completeness when building $(\infty, n)$-categories:

\begin{prp} \label{prp:from-cat-to-seg}
	The inclusion $\Seg^n(\Spaces) \sub \Cat^n(\Spaces)$ admits a right adjoint $R^n : \Cat^n(\Spaces) \to \Seg^n(\Spaces)$ -- the \emph{underlying $n$-fold Segal space} -- and the inclusion $\sCat_n \sub \Seg^n(\Spaces)$ admits a left adjoint $L^n : \Seg^n(\Spaces) \to \sCat_n$ -- the \emph{free complete $n$-fold Segal space} -- which preserves products. Therefore we obtain a product-preserving functor
	\begin{equation*}
		\mc{U}^n = L^n \circ R^n : \Cat^n(\Spaces) \to \sCat_n.
	\end{equation*}
\end{prp}

\begin{proof}
	The first three assertions are from Proposition 4.12, Definition 7.4, and Proposition 7.10 in \cite{Haugseng2018}. Since $R^n$ preserves all limits and $L^n$ preserves products we see that $\mc{U}^n$ preserves products.
\end{proof}

\subsection{Ends in $\infty$-categories} \label{sec:ends}

In ordinary category theory, the \emph{end} of a bifunctor $F : \bA^\opp \times \bA \to \bC$ is the equalizer
\begin{equation*}
	\int_{\bA} F = \int_{a \in \bA} F(a,a) := \mathrm{eq} \left( \prod_{a \in \bA} F(a,a) \rightrightarrows \prod_{a \to a'} F(a, a') \right),
\end{equation*}
which is equivalent to a limit over the twisted arrow category of $\bA$ -- see \cite[Section 1]{Loregian2021}. The latter definition readily translates to the $\infty$-categorical context:
\begin{dfn}
	Given a bifunctor $F : \sA^\opp \times \sA \to \sC$ of $\infty$-categories we define its \emph{end} to be the limit
	\begin{equation*}
		\int_\sA F = \int_{a \in \sA} F(a,a) := \lim (\Tw(\sA) \to \sA^\opp \times \sA \xlra{F} \sC).
	\end{equation*}
	Here $\Tw(\sA) \to \sA^\opp \times \sA$ is a left fibration\footnote{Warning: our $\Tw(\sA)$ is defined as $\Tw(\sA)^\opp$ in \cite{GHN2017}. Since we are borrowing results from that paper, we encourage the readers who are comparing the two works to pay attention to the change in notation.} corresponding to the bifunctor $\sA(-, -) : \sA^\opp \times \sA \to \Spaces$.
\end{dfn}

The dual notion is that of a \emph{coend}:
\begin{equation*}
	\int^\sA F := \colim (\Tw(\sA^\opp)^\opp \to \sA^\opp \times \sA \xlra{F} \sC).
\end{equation*}

\subsubsection{Useful results about ends}

Mirroring the standard result for ordinary categories, Proposition 5.1 in \cite{GHN2017} shows that if $F, G : \sA \to \sC$ are two functors of $\infty$-categories then
\begin{equation} \label{eqn:nat-formula}
	\Nat(F, G) \simeq \int_\sA \sC(F, G) = \int_{a \in \sA} \sC(F(a), G(a)).
\end{equation}
This can be used to prove an $\infty$-categorical analogue of the classical formula for Kan extensions using ends (see \cite{Loregian2021}).
\begin{prp} \label{prp:end-formula}
	Let $F : \sA \to \sB$ and $G : \sA \to \sC$ be functors of $\infty$-categories. If a right Kan extension of $G$ along $F$ at $b \in B$ exists, then
	\begin{equation*}
		(\RKan_F G)(b) \simeq \int_{a \in \sA} \sB(b, F(a)) \pitchfork G(a).
	\end{equation*}
\end{prp}

\begin{proof}
	We first recall (see for example \cite[Section 7.3]{kerodon}) that right Kan extensions are given by the formula
	\begin{equation} \label{eqn:formula-right-Kan}
		(\RKan_F G)(b) \simeq \lim_{\substack{a \in \sA \\ (b \to F(a)) \in \sB}} G(a),
	\end{equation}
	where the limit is taken over the $\infty$-category
	\begin{equation*}
		\sB_{b/} \times_{\sB} \sA.
	\end{equation*}
	Let $c \in \sC$. Then
	\begin{align*}
		\sC(c, (\RKan_F G)(b)) & \simeq \lim_{\substack{a \in \sA \\ (b \to F(a)) \in \sB}} \sC(c, G(a)) \\
		& = \lim \left( \sB_{b/} \times_{\sB} \sA \xlra{\pi} \sA \xlra{\sC(c, G(-))} \Spaces \right). \tag{a}
	\end{align*}
	We will show that the right-hand side is equivalent to
	\begin{equation*}
		\Nat(\sB(b, F(-)), \sC(c, G(-))),
	\end{equation*}
	the space of natural transformations of functors $\sA \to \Spaces$ between $\sB(b, F(-))$ and $\sC(c, G(-))$. Notice first that $\pi$ is a left fibration classifying $\sB(b, F(-))$. If $\omega : \sE \to \sA$ is a left fibration that classifies $\sC(c, G(-))$, then 
	\begin{equation*}
		\pi \times_\sA \omega : (\sB_{b/} \times_{\sB} \sA) \times_{\sA} \sE \to \sA
	\end{equation*}
	is a left fibration that classifies the composite $\sC(c, G(\pi(-))) : \sB_{b/} \times_{\sB} \sA \to \Spaces$ since the unstraightening functor takes pre-composition to pullbacks. Hence
	\begin{align*}
		\Nat(\sB(b, F(-)), \sC(c, G(-))) & \simeq \Map^{\rmcocart}_{/\sA}(\sB_{b/} \times_{\sB} \sA, \sE) \\
		& \simeq \Map^{\rmcocart}_{/(\sB_{b/} \times_{\sB} \sA)}(\sB_{b/} \times_{\sB} \sA, (\sB_{b/} \times_{\sB} \sA) \times_{\sA} \sE) \\
		& \simeq \lim \left( \sB_{b/} \times_{\sB} \sA \xlra{\sC(c, G(\pi(-)))} \Spaces \right),
	\end{align*}
	where $\Map^{\rmcocart}_{/\sD}$ denotes the space of maps in ${\sCat_1}^\rmcocart_{/\sD}$ and the last equivalence follows from \cite[Corollary 3.3.3.4]{HTT2009}, which states that the space of sections of a left fibration is equivalent to the limit of the functor it classifies.
	
	Now we can apply the end formula \Cref{eqn:nat-formula} for $\Nat(-,-)$:
	\begin{align*}
		\lim_{\substack{a \in \sA \\ (b \to F(a)) \in \sB}} \sC(c, G(a)) & \simeq \Nat(\sB(b, F(-)), \sC(c, G(-))) \\
		& \simeq \int_{a \in \sA} \Hom_{\Spaces}(\sB(b, F(a)), \sC(c,G(a))) \\
		& \simeq \int_{a \in \sA} \sC (c, \sB(b, F(a)) \pitchfork G(a)) \\
		& \simeq \sC \left( c, \int_{a \in \sA} \sB(b, F(a)) \pitchfork G(a) \right). \tag{b}
	\end{align*}
	The claim follows if we combine the two equivalences (a) and (b) and then apply the Yoneda lemma.
\end{proof}

\begin{crl} \label{crl:formula-for-end-of-cotensor}
	Assume $F : \sA \to \Spaces$ and $G : \sA \to \sC$ are functors. Then
	\begin{equation*}
		\int_\sA F \pitchfork G \simeq (\RKan_{y} G)(F) \simeq \lim_{\substack{ a \in \sA \\ x \in F(a)}} G(a) \quad \left( \simeq \lim (\sA_{F/} \to \sA \xlra{G} \sC) \right)
	\end{equation*}
	where $y : \sA \to \Fun(\sA, \Spaces)^\opp$ is the co-Yoneda embedding given on objects by $a \mapsto \sA(a, -)$, and $\sA_{F/}$ falls into the pullback square
	\begin{equation*}
		\begin{tikzcd}
			\sA_{F/} \ar[r] \ar[d] & \Fun(\sA, \Spaces)^\opp_{F/} \ar[d] \\
			\sA \ar[r, "y"] & \Fun(\sA, \Spaces)^\opp
		\end{tikzcd}
	\end{equation*}
	of $\infty$-categories.
\end{crl}

\begin{proof}
	Let $\sB = \Fun(\sA, \sC)^\opp$. Then
	\begin{equation*}
		\int_{\sA} F \pitchfork G = \int_{a \in \sA} F(a) \pitchfork G(a) \simeq \int_{a \in \sA} \sB(F, y(a)) \pitchfork G(a) \simeq (\RKan_{y} G)(F)
	\end{equation*}
	where the second equivalence comes from the co-Yoneda lemma: $F(a) \simeq \sB(F, y(a))$ naturally in $a \in \sA$. The limit formula for right Kan extensions (\cite[Definition 7.3.1.2]{kerodon}) provides the last equivalence in the claim, and the shorthand for the limit comes from the Yoneda lemma since morphisms $x : F \to y(a)$ in $\Fun(\sA, \Spaces)^\opp$ exactly correspond to vertices of $x \in F(a)$.
\end{proof}

For ends of this form, given any $a \in \sA$ and $c \in \sC$ we have a component map
\begin{multline} \label{eqn:component-map}
	\sC\left( c, \int_{\sA} F \pitchfork G \right) \simeq \int_{\sA} \Spaces(F, \sC(c, G)) \simeq \Nat(F, \sC(c, G)) \\
	\to \Spaces(F(a), \sC(c, G(a))) \to \sC(c, F(a) \pitchfork G(a))
\end{multline}
which allows us to pick specific morphisms in $\sC$ represented by the end. 

We conclude with a simplification for certain kinds of ends.

\begin{prp}
	Let $\sD = \sD_0^\lhd$ be the cone on an $\infty$-category $\sD_0$ and denote the initial object of $\sD$ by $i$. Moreover, let $F : \sD^\opp \times \sD \to \sC$ be a functor and let $F_0$ denote the restriction of $F$ to $\sD_0^\opp \times \sD_0$. Then we have a pullback square
	\begin{equation*}
		\begin{tikzcd}
			\displaystyle \int_{\sD} F \ar[r] \ar[d] & \displaystyle \int_{\sD_0} F_0 \ar[d] \\
			F(i,i) \ar[r] & \displaystyle \lim_{\sD_0} F(i,-)
		\end{tikzcd}
	\end{equation*}
	where the top map is the canonical restriction map induced by the inclusion $\sD_0 \sub \sD$, the vertical maps are two component maps coming from the limit, and the bottom map is induced by post-composition in $\sD$.
\end{prp}

\begin{proof}
	Consider the pushout square
	\begin{equation*}
		\begin{tikzcd}
			\sD_0 \ar[rr, "\id_{\sD_0} \times \{0\}"] \ar[d] & & \sD_0 \times \Delta^1 \ar[d, "g"] \\
			\Delta^0 \ar[rr, "\{i\}"] & & \sD 
		\end{tikzcd}
	\end{equation*}
	which identifies $\sD$ with a cone on $\sD_0$. Since $\Tw : \sCat_1 \to \sCat_1$ preserves colimits (see \cite[Remark 8.1.1.4]{kerodon}) we have an induced pushout square
	\begin{equation*}
		\begin{tikzcd}
			\Tw(\sD_0) \ar[r, "\alpha"] \ar[d, "\beta"'] & \Tw(\sD_0 \times \Delta^1) \ar[d, "\gamma"] \\
			\Tw(\Delta^0) \ar[r, "\delta"] & \Tw(\sD)
		\end{tikzcd}
	\end{equation*}
	If we set $\conj{F} := \pi \circ F : \Tw(\sD) \to \sD^\opp \times \sD \to \sC$ then \cite[Proposition 7.6.3.26]{kerodon} gives a \emph{pullback} square
	\begin{equation} \label{eqn:pullback-square-for-end}
		\begin{tikzcd}
			\lim \conj{F} \ar[r] \ar[d] & \lim \conj{F}\vert_{\Tw(\sD_0 \times \Delta^1)} \ar[d] \\
			\lim \conj{F}\vert_{\Tw(\Delta^0)} \ar[r] & \lim \conj{F}\vert_{\Tw(\sD_0)}
		\end{tikzcd}
	\end{equation}
	of limits of $\conj{F}$ restricted to the various $\infty$-categories in the diagram. We will now identify the four corners of this pullback square and show that it reduces to the diagram in the claim.
	\begin{itemize}
		\item By definition we have $\lim \conj{F} = \int_\sD F$.
		\item Since $\Tw(\Delta^0) \simeq \Delta^0$ and the map $\delta : \Tw(\Delta^0) \to \Tw(\sD)$ is the inclusion of the object $\id_i$, the identity map of the initial object $i$, the limit $\lim \conj{F}\vert_{\Tw(\Delta^0)}$ is equivalent to $\conj{F}(\id_i) = F(i,i)$.
		\item The map $\delta \circ \beta : \Tw(\sD_0) \to \Tw(\sD)$ (and hence also $\gamma \circ \alpha$, since the two are equivalent) factors through $\beta$ and thus $\conj{F}\vert_{\Tw(\sD_0)}$ is the constant functor at $F(i,i)$. Alternatively we can write it as the composite functor
		\begin{equation*}
			\Tw(\sD_0) \to \sD_0^\opp \times \sD_0 \to \sD_0 \to \sC
		\end{equation*}
		where the last map $\sD_0 \to \sC$ is the constant functor $c_{F(i,i)}$ and so
		\begin{equation*}
			\lim \conj{F}\vert_{\Tw(\sD_0)} \simeq \int_{\sD_0} c_{F(i,i)}.
		\end{equation*}
		The end of a bifunctor $H : \sD_0^\opp \times \sD_0 \to \sC$ which factors through the projection to $\sD_0$ is simply the limit of the functor out of $\sD_0$, so
		\begin{equation*}
			\lim \conj{F}\vert_{\Tw(\sD_0)} \simeq \int_{\sD_0} c_{F(i,i)} = \lim_{\sD_0} c_{F(i,i)}.
		\end{equation*}
		\item The map $g : \sD_0 \times \Delta^1 \to \sD$ is given on objects by
		\begin{equation*}
			g(x, n) = 
			\begin{cases}
				i & n = 0, \\
				x & n = 1,
			\end{cases}
		\end{equation*}
		where $x \in \sD_0$ and $n \in \{0,1\}$ stands for one of the two objects of $\Delta^1$. This means that
		\begin{equation*}
			\lim \conj{F}\vert_{\Tw(\sD_0 \times \Delta^1)} \simeq \int_{\sD_0 \times \Delta^1} F(g(-), g(-)) \simeq \int_{\Delta^1} \int_{\sD_0} F(g(-),g(-))
		\end{equation*}
		where the last equivalence follows from the Fubini theorem for ends (see, for example, \cite[Proposition 2.21]{Haugseng2018}). It can be easily calculated that $\Tw(\Delta^1)$ is isomorphic to the ordinary cospan
		\begin{equation*}
			\bullet \to \bullet \leftarrow \bullet
		\end{equation*}
		where, in terms of morphisms in $\Delta^1$, the left and right points correspond to $\id_0$ and $\id_1$ and the middle point corresponds to the unique nontrivial morphism $e : 0 \to 1$. This readily implies that we have a pullback square
		\begin{equation*}
			\begin{tikzcd}
				\displaystyle \int_{\Delta^1} \int_{\sD_0} F(g(-),g(-)) \ar[r] \ar[d] & \displaystyle \int_{\sD_0} F(-,-) \ar[d] \\
				\displaystyle \int_{\sD_0} F(i,i) \ar[r] & \displaystyle \int_{\sD_0} F(i,-)
			\end{tikzcd}
		\end{equation*}
		which simplifies to
		\begin{equation*}
			\begin{tikzcd}
				\displaystyle \int_{\Delta^1} \int_{\sD_0} F(g(-),g(-)) \ar[r] \ar[d] & \displaystyle \int_{\sD_0} F_0 \ar[d] \\
				\displaystyle \lim_{\sD_0} c_{F(i,i)} \ar[r] & \displaystyle \lim_{\sD_0} F(i,-)
			\end{tikzcd}
		\end{equation*}
	\end{itemize}
	After substituting each limit and simplifying, the pullback square in \Cref{eqn:pullback-square-for-end} turns into the desired square
	\begin{equation*}
		\begin{tikzcd}
			\displaystyle \int_{\sD} F \ar[r] \ar[d] & \displaystyle \int_{\sD_0} F_0 \ar[d] \\
			F(i,i) \ar[r] & \displaystyle \lim_{\sD_0} F(i,-)
		\end{tikzcd}
	\end{equation*}
	concluding the proof.
\end{proof}

\begin{crl} \label{crl:nat-trans-inductive}
	Let $\sD = \sD_0^\lhd$ be the cone on an $\infty$-category $\sD_0$ and denote the initial object of $\sD$ by $i$. If $F, G : \sD \to \sC$ are two functors of $\infty$-categories and $F_0, G_0$ are their restrictions to $\sD_0 \sub \sD$ then there is a pullback square
	\begin{equation*}
		\begin{tikzcd}
			\Nat(F,G) \ar[r] \ar[d] & \Nat(F_0, G_0) \ar[d] \\
			\sC(F(i), G(i)) \ar[r] & \displaystyle \lim_{\sD_0} \sC(F(i), G(-))
		\end{tikzcd}
	\end{equation*}
\end{crl}

\begin{proof}
	The statement follows immediately by the previous proposition and the end formula for natural transformations.
\end{proof}

Informally, the data of a natural transformation $\eta : F \to G$ is equivalent to that of its restriction $\eta_0 : F_0 \to G_0$ together with an extra component map $F(i) \to G(i)$ that is compatible with the existing components $F(x) \to G(x)$, for $x \in \sD_0$. It would be interesting to have a similar formula in the more general case where $\sD_0$ is obtained from $\sD$ by removing an arbitrary object $d$ (in our case we removed the very special initial object $i$).

\subsection{Commutative monoids} \label{sec:commutative-monoids}

\begin{dfn}
	A \emph{symmetric monoidal $\infty$-category} is a functor $\sC^\otimes : \bG^\opp \to \sCat_1$ such that the Segal maps
	\begin{equation*}
		\begin{tikzcd}[row sep = 0.05em]
			\brak{n} \ar[r, "\tau_j"] & \brak{1}, & \\
			j \ar[r, mapsto] & 1 & \\
			k \ar[r, mapsto] & \ast & \text{if $k \neq j$}
		\end{tikzcd}
	\end{equation*}
	in $\bG^\opp$ induce an equivalence of $\infty$-categories
	\begin{equation*}
		\sC^\otimes_n \xlra{\simeq} (\sC^\otimes_1)^{\times n}.
	\end{equation*}
	This is also called a \emph{Segal condition}. The \emph{underlying category} of $\sC^\otimes$ is $\sC := \sC^\otimes_1$. We say that an $\infty$-category $\sD$ \emph{admits a symmetric monoidal structure} if there exists a symmetric monoidal $\infty$-category $\sD^\otimes$ with underlying category $\sD^\otimes_1 \simeq \sD$.
\end{dfn}

If $\sC$ is an $\infty$-category with finite limits then it admits a symmetric monoidal structure, called the \emph{cartesian monoidal structure}, and we denote the associated functor by $\sC^\times : \bG^\opp \to \sCat_1$. The construction of this functor can be found in \cite[Section 2.4.1]{HA2017}.

\begin{dfn}
	A \emph{commutative monoid} or \emph{commutative algebra} in a cartesian monoidal category $\sC$ is a functor $M^\otimes : \bG^\opp \to \sC$ satisfying the Segal conditions. We denote by $\CAlg(\sC)$ the full subcategory of $\Fun(\bG^\opp, \sC)$ spanned by the commutative monoids in $\sC$.
\end{dfn}

We have an auxiliary characterization of commutative monoids that will turn out to be useful later. We recall that an \emph{inert morphism} in $\bG^\opp$ is a map $f : \brak{m} \to \brak{n}$ such that $f^\inv(k)$ is a singleton for each $k \neq \ast$.

\begin{prp} \label{prp:calg-are-sections}
	Commutative monoids $M^\otimes \in \CAlg(\sC)$ are equivalent to sections $\wilde{M} : \bG^\opp \to \what{\sC^\times}$ of a cocartesian fibration $\what{\sC^\times} \to \bG^\opp$ which straightens to $\sC^\times$ such that $\wilde{M}$ sends inert morphisms to cocartesian morphisms.
\end{prp}

\begin{proof}
	A more general statement, Proposition 2.4.2.5 in \cite{HA2017}, is true for operads and this is the case $\ms{O} = \mathrm{Comm} \simeq \bG^\opp$. The condition on inert morphisms is embedded in the definition of maps of $\infty$-operads: a section of $\what{\sC^\times} \to \bG^\opp$ is a map of operads $\bG^\opp \to \what{\sC^\times}$ over $\bG^\opp$, and such maps must send inert morphisms to cocartesian morphisms by \cite[Definition 2.1.1.10]{HA2017}.
\end{proof}
(To see why the condition on inert morphisms is needed we note that if $\wilde{M}$ sent \emph{all} morphisms to cocartesian morphisms in $\what{\sC^\times}$ then the corresponding functor $M : \bG^\opp \to \sC$ would characterize an idempotent element $M_1 \in \sC$, i.e. $M_1 \times M_1 \simeq M_1$.) 

There is a commutative diagram
\begin{equation*}
	\begin{tikzcd}
		\what{\sC^\times} \ar[rr, "\times"] & & \sC \\
		& \bG^\opp \ar[ul, "\wilde{M}"] \ar[ur, "M^\otimes"'] &
	\end{tikzcd}
\end{equation*}
where the top map is given on fibers by the product maps $\sC^\times_t \simeq \sC^{\times t} \to \sC$, and the proposition informally states that
\begin{equation*}
	\sC^\times_t \ni \wilde{M}(\brak{t}) \simeq (M^\otimes_1, M^\otimes_1, \dotsc, M^\otimes_1) \in \sC^{\times t}
\end{equation*} 

\begin{prp} \label{prp:alg-to-alg-alg}
	Precomposition with the smash product of finite sets $\wedge : \bG^\opp \times \bG^\opp \to \bG^\opp$ promotes a commutative monoid $M^\otimes \in \CAlg(\sC)$ in $\sC$ to another commutative monoid $W(M^\otimes) \in \CAlg(\CAlg(\sC))$ in $\CAlg(\sC)$ whose carrier is equivalent to $M^\otimes$.
\end{prp}

\begin{proof}
	Note that $\wedge$ sends $(\brak{m}, \brak{n})$ to $\brak{mn}$. If we view $W(M^\otimes) = M^\otimes \circ \wedge$ as a functor $\bG^\opp \to \CAlg(\sC^\times)$ then the carrier is obtained by evaluating it at $\brak{1}$:
	\begin{equation*}
		W(M^\otimes) : \brak{1} \mapsto (\brak{n} \mapsto M^\otimes_{1 \cdot n} = M^\otimes_n) = M^\otimes
	\end{equation*}
	as desired.
\end{proof}

\subsubsection{Monoidal structures on products}

\begin{prp} \label{prp:comm-mon-on-prod}
	If $M \in \sC$ admits a commutative monoidal structure (i.e. it extends to some $M^\otimes \in \CAlg(\sC)$ satisfying $M^\otimes_1 \simeq M$) then so does $M^{\times k}$ for any $k \geq 0$.
\end{prp}

\begin{proof}
	The smash-by-$\brak{k}$ functor
	\begin{equation*}
		- \wedge \brak{k} : \bG^\opp \to \bG^\opp
	\end{equation*}
	sends $\brak{n}$ to $\brak{nk}$. In particular, it sends the Segal maps $\tau_i : \brak{n} \to \brak{1}$ to the maps $\kappa_i : \brak{nk} \to \brak{k}$ defined by 
	\begin{equation*}
		\kappa_i (m) =
		\begin{cases}
			l-1 & \text{if } m = i + lk, \\
			\ast & \text{if } m \not\equiv i \mod k.
		\end{cases}
	\end{equation*}
	Precomposing $M^\otimes$ by $- \wedge \brak{k}$ yields a functor $N^\otimes : \bG^\opp \to \sC$ such that $N^\otimes_1 \simeq M^{\times k}$ and whose images fit into the commutative diagram
	\begin{equation*}
		\begin{tikzcd}
			N^\otimes_n \ar[r] \ar[d, "\simeq"'] & (N^\otimes_1)^{\times n} \ar[d, "\simeq"] & \\
			M^\otimes_{nk} \ar[d, "\simeq"] \ar[r, "\prod_i \kappa_i"] & (M^\otimes_k)^{\times n} \ar[d, "\simeq"] \\
			M^{\times nk} \ar[r, "\simeq"] & (M^{\times k})^{\times n} 
		\end{tikzcd}
	\end{equation*}
	Therefore the top map, the product of the induced Segal maps, is an equivalence, which tells us that $N^\otimes \in \CAlg(\sC)$ is a commutative monoid extending $N^\otimes_1 \simeq M_k \simeq M^{\times k}$.
\end{proof}

Interestingly enough, we note that the wedge sum of pointed finite sets (which is the coproduct in that category) is precisely the symmetric monoidal structure on $\bG^\opp$ such that symmetric monoidal functors $\bG^\opp \to \sC$ are exactly commutative monoids in $\sC$ for any cartesian monoidal $\sC$:
\begin{equation*}
	\CAlg(\sC) \simeq \Fun^\otimes((\bG^\opp, \vee), (\sC, \times)).
\end{equation*}
This fact can be used to give a shorter proof of the previous proposition. 
\begin{proof}[Alternative proof of \Cref{prp:comm-mon-on-prod}]
	Since $\wedge$ distributes over $\vee$ we can directly compute that $N^\otimes = M^\otimes \circ (- \wedge \brak{k})$ is a symmetric monoidal functor:
	\begin{align*}
		N^\otimes(\brak{m} \vee \brak{n}) & = M^\otimes((\brak{m} \vee \brak{n}) \wedge \brak{k}) \\
		& = M^\otimes((\brak{m} \wedge \brak{k}) \vee (\brak{n} \wedge \brak{k})) \\
		& \simeq M^\otimes(\brak{m} \wedge \brak{k}) \times M^\otimes(\brak{n} \wedge \brak{k}) \\
		& \simeq N^\otimes(\brak{m}) \times N^\otimes(\brak{n}). \qedhere
	\end{align*}
\end{proof}

\begin{prp} \label{prp:cart-mon-str-on-prod}
	If $\sC$ is a cartesian monoidal $\infty$-category then, for every $n \in \NN$, $\sC^\times_n$ has finite limits and the induced symmetric monoidal structure on it is cartesian.
\end{prp}

\begin{proof}
	Let $\sA$ be a finite $\infty$-category and let $D : \sA \to \sC^\times_n$ be a diagram. By assumption each component $D_i : \sA \to \sC^\times_n \xlra{(\tau_i)_\ast} \sC$ has a limit $x_i$, and together the $x_i$ assemble to an element $\conj{x} \in \sC^{\times n}$ corresponding to some element $x \in \sC^\times_n$. But for any $y \in \sC^\times_n$ we have
	\begin{align*}
		\Hom_{\sC^\times_n}(y, x) \simeq \prod_{i = 1}^n \Hom_{\sC}(y_i, x_i)\simeq \lim_\sA \prod_{i = 1}^n \Hom_{\sC} \left( y_i, D_i \right)\simeq \lim_\sA \Hom_{\sC^\times_n}(y, D)
	\end{align*}
	and so $x$ is a limit of $D$. To show that the induced monoidal structure $\sD^\times$ on $\sC^\times_n$ is cartesian it will suffice to show that the map $\mu_\ast : \sD^\times_p \to \sD^\times$, induced by the morphism $\mu : \brak{p} \to \brak{1}$ sending everything except the basepoint to $1$, is the same as the product map. But $\mu_\ast$ is $\sC^\times \circ (\mu \wedge \brak{n})$ by \Cref{prp:comm-mon-on-prod} and each component of the latter is the product map because $\sC$ is cartesian monoidal, so we're done.
\end{proof}

\subsection{Techniques to increase and decrease the category number} \label{sec:increase-cat-number}

At times it will be fruitful to pass from $n$-categories to related $m$-categories for some $m > n$. Here we collect some techniques that allow us to do this.

\subsubsection{Categories of commutative $n$-cubes}

We start with a generalization of Rezk's classifying diagram (\cite{Rezk2001}).

\begin{dfn} \label{dfn:n-cubes}
	Let $i : \bD \to \Cat^1(\Set)$ denote the inclusion of the linearly ordered sets $i(k) = \Delta^k$ via the Yoneda embedding. By taking products we get a functor $i^n : \bD^{\times n} \to \Cat^1(\Set)$. If $X : \bD^\opp \to \sD$ is a simplicial object in an $\infty$-category $\sD$, the \emph{commutative $n$-cube} of $X$ is the functor
	\begin{equation*}
		\square^n(X) : \bD^{n, \opp} \to \sD, \quad (k_1, \dotsc, k_n) \mapsto \int_{\bD^\opp} i^n(k_1, \dotsc, k_n) \pitchfork X.
	\end{equation*}
\end{dfn}

First note that since $i(0) = \Delta^0 \simeq \ast$ we have
\begin{equation*}
	\square^n(X)_{k_1, \dotsc, k_p, 0, \dotsc, 0} \simeq \square^p(X)_{k_1, \dotsc, k_p}
\end{equation*}
for any $p \leq n$. If $X \in \Cat(\sD)$, the commutative $n$-cube is supposed to classify commutative diagrams in $X$ with shape the $n$-cube with sides of lengths $k_1, k_2, \dotsc, k_n$ (the length here indicates the number of $1$-morphisms). We will demonstrate this by computing $\square^2(X)$, which is the only case used in the paper. By the equation above and the Yoneda lemma we have that
\begin{equation} \label{eqn:square-2-X-k0}
	\square^2(X)_{k_1, 0} \simeq \square^1(X)_{k_1} \simeq \int_{\bD^\opp} \Delta^{k_1} \pitchfork X \simeq X_{k_1}, 
\end{equation}
i.e. the $(k_1, 0)$-simplices of $\square^2(X)$ correspond to $k_1$-simplices of $X$ and, since $X$ is a category object, these in turn correspond to chains of $k_1$ many composable $1$-morphisms. When $(k_1, k_2) = (1,1)$, the simplicial set $\Delta^1 \times \Delta^1$ has exactly two non-degenerate $2$-simplices with their composite edges identified: $\Delta^1 \times \Delta^1 \simeq \Delta^2 \sqcup_{\Delta^1} \Delta^2$. Hence
\begin{equation} \label{eqn:square-2-X-11}
	\square^2(X)_{1, 1} = \int_{\bD^\opp} (\Delta^1 \times \Delta^1) \pitchfork X \simeq \int_{\bD^\opp} (\Delta^2 \sqcup_{\Delta^1} \Delta^2) \pitchfork X \simeq X_2 \times_{X_1} X_2,
\end{equation}
and since the map $X_2 \to X_1$ along which the pullback is taken is the one defining the composition in $X$ we see that the $(1,1)$-simplices of $\square^2(X)$ correspond to commutative squares of $1$-morphisms of $X$.

\begin{prp}
	The commutative $n$-cube of $X \in \Cat^1(\sD)$, where $\sD$ is an $\infty$-category with finite limits, is an $n$-uple category object in $\sD$.
\end{prp}

\begin{proof}
	We will first prove the claim in the case $\sD = \Spaces$, in which case $- \pitchfork - = \Hom(-,-)$ and so
	\begin{align*}
		\square^n(X)_{k_1, \dotsc, k_i, \dotsc, k_n} & = \int_{\bD^\opp} \Hom(\Delta^{k_1} \times \dotsb \times \Delta^{k_i} \times \dotsb \times \Delta^{k_n}, X) \\
		& \simeq \int_{\bD^\opp} \Hom(\Delta^{k_i}, [\Delta^{k_1} \times \dotsb \times \Delta^{k_{i-1}} \times \Delta^{k_{i+1}} \times \dotsb \times \Delta^{k_n}, X]) \\
		& \simeq \square^1([\Delta^{k_1} \times \dotsb \times \Delta^{k_{i-1}} \times \Delta^{k_{i+1}} \times \dotsb \times \Delta^{k_n}, X])_{k_i} \\
		& \simeq [\Delta^{k_1} \times \dotsb \times\Delta^{k_{i-1}} \times \Delta^{k_{i+1}} \times \dotsb \times \Delta^{k_n}, X]_{k_i}
	\end{align*}
	for any $1 \leq i \leq n$, where $[Y, Z]$ is the internal hom in simplicial spaces. If both $Y$ and $Z$ are category objects in spaces then $[Y, Z] \in \Cat^1(\Spaces)$ since $- \times Y : \Cat^1(\Spaces) \to \Cat^1(\Spaces)$ preserves colimits. Hence
	\begin{equation*}
		[\Delta^{k_1} \times \dotsb \times\Delta^{k_{i-1}} \times \Delta^{k_{i+1}} \times \dotsb \times \Delta^{k_n}, X] : \bD^\opp \to \Spaces
	\end{equation*}
	satisfies the Segal condition in its variable $k_i$, and since this holds for all $i$ we are done.
	
	When $\sD$ is a general $\infty$-category with finite limits we can reduce to the case of $\Spaces$ by applying $\sD(d, -)$ for an arbitrary object $d$ and commuting that past the end and the cotensor, obtaining
	\begin{equation*}
		\sD(d, \square^n(X)_{k_1, \dotsc, k_n}) \simeq \square^n(\sD(d, X))_{k_1, \dotsc, k_n}.
	\end{equation*}
	The Segal condition holds for the right-hand side by our previous argument and therefore it holds for $\square^n(X)$ since the above equivalence is natural in $d$.
\end{proof}

\subsubsection{Monoids}

\begin{dfn}
	A \emph{monoid} in a cartesian monoidal category $\sC$ is a category object $X : \bD^\opp \to \sC$ with $X_0 \simeq \ast$. The $\infty$-category of monoids in $\sC$ is the full subcategory $\Mon(\sC) \sub \Cat(\sC)$ spanned by the monoids.
\end{dfn}

A monoid $X$ still comes equipped with equivalences $X_n \simeq X_1^{\times n}$ but is no longer invariant under the action of the symmetric groups, so a priori the multiplication map $X_2 \to X_1$ is not commutative.

\begin{prp} \label{prp:underlying-monoid}
	There is a functor $u : \bD^\opp \to \bG^\opp$ such that precomposition by $u$ extracts a monoid $\underline{M^\otimes} \in \Mon(\sC)$, called the \emph{underlying monoid}, from a commutative monoid $M^\otimes \in \CAlg(\sC)$. 
\end{prp}

\begin{proof}
	The functor $u$ appears in \cite[Construction 4.1.2.9]{HA2017}: it is given on objects by $[n] \mapsto \brak{n}$ and on morphisms by sending $\vphi : [m] \to [n]$ to the map of pointed finite sets
	\begin{equation*}
		\vphi^\ast : \brak{n} \to \brak{m},
		\quad \vphi^\ast(k) =
		\begin{cases}
			\min \vphi^\inv(k) & \text{if $\vphi^\inv(k) \neq \emptyset$}, \\
			\ast & \text{if $\vphi^\inv(k) = \emptyset$}.
		\end{cases} 
	\end{equation*}
	It is easy to see that $u$ sends the Segal maps in $\bD^\opp$ to the Segal maps in $\bG^\opp$, and this shows that $\underline{M^\otimes} = M^\otimes \circ u$ is a monoid in $\sC$ whenever $M^\otimes$ is a commutative monoid in $\sC$.
\end{proof}

\begin{crl} \label{crl:symm-mon-to-2cat}
	The underlying monoid of a symmetric monoidal $\infty$-category is an $(\infty, 2)$-category with a contractible space of objects. 
\end{crl}

\begin{proof}
	A symmetric monoidal $\infty$-category is, by definition, a commutative monoid $\sC^\otimes \in \CAlg(\sCat_1)$. Recall that $\sCat_1$ can be modeled by a special subcategory of $\Seg^1(\Spaces)$. So we have an inclusion $\CAlg(\sCat_1) \sub \CAlg(\Seg^1(\Spaces))$ and, by adjunction, we can consider $\sC^\otimes$ as a functor $\bG^\opp \times \bD^\opp \to \Spaces$ which satisfies the Segal condition in both variables. Precomposing with $u$ in the first variable yields a $2$-fold Segal space $\underline{\sC^\otimes} : \bD^{2, \opp} \to \Spaces$ with a contractible space of objects since $\underline{\sC^\otimes}_{0, n} \simeq \sC^\otimes_{0, n} \simeq \ast$ for any $n$. After completing (i.e. applying $\mc{U}^2$) we obtain the desired $(\infty, 2)$-category.
\end{proof}